%% file: HaleOlver2015.tex
\documentclass{siamltex}

\input{preamble}

\usepackage[linewidth=.5pt]{mdframed}
\usepackage{lipsum}
\usepackage{url}

\title{A fast and spectrally convergent algorithm for rational-order fractional integral and differential equations}

\author{\sc Nicholas Hale \thanks{Department of Mathematical Sciences, Stellenbosch University, Stellenbosch, 7602, South Africa. (nickhale@sun.ac.za)}
\and {Sheehan Olver \thanks{Department of Mathematics, Imperial College, London, SW7 2AZ, United Kingdom
(s.olver@imperial.ac.uk)}}}

\begin{document}
\maketitle

\input{abstract}



\section{Introduction}\label{sec:intro}%
\input{intro.tex}

\section{Preliminaries}\label{sec:prelim}
\input{prelim.tex}

\section{Half-integer order integral equations}\label{sec:half}
\input{half}

\section{FDEs: Riemann--Liouville definition}\label{sec:half_RL}
\input{half_RL}

\section{FDEs: Caputo definition}\label{sec:half_Cap}
\input{half_Cap}

\section{Computational issues}\label{sec:comp}
\input{comp}

\section{Rational-order equations}\label{sec:ext}
\input{extensions}

\section{Conclusion}\label{sec:conc}
\input{conc}%
\section*{Acknowledgements}
We thank Daniel Hauer (U. Sydney) for discussions related to convergence in higher order norms, Marcus Webb (K.U. Leuven) for discussions on fractional differential equations, and Alex Townsend (Cornell) for some useful suggestions.

\appendix
\section{Miscellaneous proofs}\label{sec:appendixA}
\input{appendixA}

\section{Convergence and stability results}\label{sec:appendixC}
\input{half_analysis}


\bibliographystyle{siam}
\bibliography{HaleOlver2015}

\end{document}

%% file: preamble.tex
\usepackage[T1]{fontenc}
\usepackage[utf8]{inputenc}
\usepackage{amsmath}
\usepackage{amssymb}
\usepackage{url}
\usepackage{fullpage}
\usepackage{verbatim}
\usepackage{listings}
\usepackage{color}
\usepackage{array}
\usepackage{graphicx}
\usepackage{overpic}
\usepackage{epstopdf}
\usepackage{cite}
\usepackage{upquote}
\usepackage{epsfig}
\usepackage{todonotes}
\usepackage{relsize}
\usepackage{mathtools}
\usepackage{courier} 
\usepackage{xcolor} 
\usepackage{tikz} 
\usetikzlibrary{shapes,arrows}
\usepackage{soul}
\usepackage{float}

\usepackage{MnSymbol}

\DeclareMathAlphabet{\pazocal}{OMS}{zplm}{m}{n}

\def\urltilde{\raise.17ex\hbox{$\scriptstyle\mathtt{\sim}$}}

\newcommand{\be}{\begin{equation}}
\newcommand{\ee}{\end{equation}}

\newcommand{\foo}{^{\vphantom{(k)}}_{\vphantom{(k)}}}

\definecolor{gray}{rgb}{0.5,0.5,0.5}
\definecolor{dkgreen}{rgb}{.068,.578,.068}
\definecolor{dkpurple}{rgb}{.320,.064,.680}
\usepackage{upquote}
\usepackage{listings}
\definecolor{mygreen}{rgb}{0.2,0.4,0.2}
\definecolor{mymauve}{rgb}{0.58,0,0.82}
\lstdefinestyle{customc}{
  language=Matlab,
  showstringspaces=false,
  basicstyle=\scriptsize\ttfamily,
  keywordstyle=\color{blue},
  commentstyle=\color{dkgreen},
  identifierstyle=\color{black},
  stringstyle=\color{dkpurple},
  deletekeywords={sech, zeros, plot, eye, sum, isnan, diff, hold, max, expm, sign, sqrt, diag, ans, format, surf, axis, subplot, pi, sin, cos, eps, schur, sinh, imag, abs, cosh, asinh, realmax, nargin, exp},
  morekeywords={parfor},
  numbers=left,                    
  numbersep=10pt,
  numberstyle=\tiny\color{gray},
  xleftmargin=1in,
}
\usepackage{todonotes}

\usepackage{changes}
\definecolor{DarkGreen}{rgb}{0,0.6,0}
\definechangesauthor[name={Sheehan}, color=DarkGreen]{so}

\input socommon

\newcommand{\D}{\vphantom{\frac{1+2\lambda}{2+\lambda}\ddots}} 

\def\C{\mathbb{C}}
\def\R{\mathbb{R}}
\def\PP[#1,#2,#3]{{\mathbf{P}}^{(#2,#3)}_{#1}}
\def\PPZ[#1,#2]{\mathbf{P}^{(#1,#2)}}
\def\DD[#1,#2,#3]{{D}_{#1, #2,#3}}

\def\calQ{\mathcal{Q}}
\def\calD{\mathcal{D}}
\def\P{\mathbf{P}}
\def\T{\mathbf{T}}
\def\U{\mathbf{U}}
\def\C{\mathbf{C}}
\def\A{\mathbf{A}}
\def\B{\mathbf{B}}
\def\Cl[#1]{\mathbf{C}^{(#1)}}

\def\defeq{:=}

\def\PoU{{}}

\def\shrink{\!}

\def\Cal{\mathbf{C}^{(\lambda)}_\gamma}

\newcounter{example}
\newenvironment{example}[1][]{\refstepcounter{example}\par\medskip
   \textbf{Example~\theexample: #1} \rmfamily}{\medskip}

\def\myab{\left[\begin{array}{c}\underline{\hat a} \vphantom{D_1^{1/2}} \\\underline{\hat b}\end{array}\right]}
\def\myvec[#1,#2]{\left[\begin{array}{c}{#1} \vphantom{D_1^{1/2}} \\ {#2}\end{array}\right]}
\def\myvecu[#1,#2]{\left[\begin{array}{c}\underline{#1} \vphantom{D_1^{1/2}} \\\underline{#2}\end{array}\right]}

%% file: socommon.tex
\def\pr(#1){\left({#1}\right)}
\def\br[#1]{\left[{#1}\right]}
\def\set#1{\left\{{#1}\right\}}
\def\ip<#1>{\left\langle{#1}\right\rangle}
\def\iip<#1>{\left\langle\!\langle{#1}\right\rangle\!\rangle}

\def\fpr(#1){\!\pr({#1})}

\def\function#1#2{\expandafter\def\csname #1\endcsname(##1){#2\fpr({##1})}
				\expandafter\def\csname #1p\endcsname(##1){#2'\fpr({##1})}
				\expandafter\def\csname #1pp\endcsname(##1){#2''\fpr({##1})}
				\expandafter\def\csname #1pn\endcsname##1(##2){#2^{\pr({##1})}\fpr({##2})}				}
\def\defoperator#1#2{\expandafter\def\csname #1\endcsname[##1]{#2\!\br[{##1}]}}

\def\ffunction#1{\function{#1}{#1}}

\def\O(#1){{\cal O}\!\left(#1\right)}

\def\Oo(#1){{\rm o}\!\left({#1}\right)}

\def\fO(#1){\Oh\fpr({#1})}

\def\dkfn^#1#2{\,{\rm d}^#1 #2}
\def\dkf#1{\,{\rm d}#1}

\def\H_#1^#2(#3){H_#1^{(#2)}\fpr({#3})}
\def\J_#1(#2){{\rm J}_#1\!\pr(#2)}

\def\ddxn^#1{{\dkf{}^#1 \over \dkfn^#1{x}}}

\def\seq_#1^#2#3{\set{#3_{#1},\ldots,#3_{#2}}}

\def\abs#1{\left|{#1}\right|}

\ffunction{f}
\ffunction{g}

\def\mapengine#1,#2.{\mapfunction{#1}\ifx\void#2\else\mapengine #2.\fi }

\def\map[#1]{\mapengine #1,\void.}

\def\mapenginesep_#1#2,#3.{\mapfunction{#2}\ifx\void#3\else#1\mapengine #3.\fi }

\def\mapsep_#1[#2]{\mapenginesep_{#1}#2,\void.}

\def\vcbr[#1]{\pr({#1})}

\def\bvect[#1,#2]{
{
\def\dots{\cdots}
\def\mapfunction##1{\ | \  ##1}
	\sopmatrix{
		 \,#1\map[#2]\,
	}
}
}

\def\vect[#1]{
{\def\dots{\ldots}
	\vcbr[{#1}]
}}

\def\vectt[#1]{
{\def\dots{\ldots}
	\vect[{#1}]^{\top}
}}

\def\Vectt[#1]{
{
\def\mapfunction##1{##1 \cr} 
\def\dots{\vdots}
	\sopmatrix{
		\map[#1]
	}
}}

\def\qand{\quad\hbox{and}\quad}

\def\simlimit_#1{\,\,\,\sim \!\!\!\!\!\!\!\!\!{ \atop \scriptscriptstyle #1 }}

\def\XXint#1#2#3{{\setbox0=\hbox{$#1{#2#3}{\int}$}
     \vcenter{\hbox{$#2#3$}}\kern-.5\wd0}}

\def\rad^#1#2{\,\,{}^#1\!\!\!\!\sqrt{#2}\,}

\def\Figuretwofixed[#1,#2]#3\par{
\Figuretwow[#1,#2]{0.48 \hsize}{#3}\par
}

\def\sopmatrix#1{\begin{pmatrix}#1\end{pmatrix}}

\def\meeq#1{\def\ccr{\\\addtab}
\tabskip=\@centering
 \begin{align*}
 \addtab#1
 \end{align*}
  }

%% file: abstract.tex
\begin{abstract}
A fast algorithm (linear in the degrees of freedom) for the solution of linear 
variable-coefficient rational-order fractional integral and differential equations is described. The approach is 
related to the ultraspherical method for ordinary differential equations~\cite{olver2013}, and involves constructing two different 
bases, one for the domain of the operator and one for the range of the operator.  
The bases are constructed from direct sums of suitably weighted ultraspherical or Jacobi polynomial expansions, for which explicit 
representations of fractional integrals and derivatives are known, and are carefully chosen so that the resulting operators are 
banded or almost-banded.
Geometric convergence is demonstrated for numerous model problems
when the variable coefficients and right-hand side are sufficiently smooth. 
\end{abstract}

%% file: intro.tex
Fractional derivatives and fractional differential equations (FDEs) are becoming increasingly prevalent in the 
mathematical modelling of biological and physical processes \cite{dalir2010, hilfer2000, magin2006, magin2010, oldham2010, Ortigueira2006, sabatier2007, scalas2000, sheng2011}.
Numerical techniques for computing solutions are typically based on finite differences~\cite{cui2009, meerschaert2006, yuste2005} or finite elements~\cite{deng2008, ford2011, liu2004},
but these usually provide only low accuracy solutions due to the global nature of fractional derivatives.
There have been some recent developments in spectral methods for FDEs~\cite{chen2016,li2009,zayernouri2014}, but these are only observed to achieve spectral accuracy for special solutions.

This paper concerns the numerical solution of linear equations involving {\em rational}-order fractional 
integrals and derivatives on the interval $[-1,1]$.\footnote{Problems defined
on any other bounded interval may be mapped to $[-1,1]$ by a suitable affine transformation.}  
For $p,q\in \mathbb{N}$ with $0 < p < q$, the left-sided $p/q$-integral is defined  as~\cite{riesz1949}\footnote{The  {right-sided rational-integral}, 
$\prescript{}{x}{\calQ^{p/q}_1}$, is similar, but with the limits on the 
integral changed from $[-1, x]$ to $[x, 1]$ and the bracketed term in the denominator of the integrand negated.
Without loss of generality, we focus on the left-sided case.}
\be\label{eqn:ratint_def}
    \prescript{}{-1}{{\calQ}^{p/q}_x}f(x)  = \frac{1}{\Gamma(p/q)}\int_{-1}^x\frac{f(t)}{(x-t)^{1-p/q}}\,dt,
\ee
and for $m \in \mathbb N$, $(m+{p/q})$-order derivatives of {\em Riemann--Liouville} (RL) and {\em Caputo} types are given by  
\be\label{eqn:ratint_def2}
    \prescript{RL\!}{-1}{\calD}_x^{m+p/q}f(x) = \frac{d^{m+1}}{dx^{m+1}}\left(\prescript{}{-1}{\calQ}_x^{1-p/q}f(x)\right) \quad \text{ and} 
    \quad \prescript{C\!}{-1}{\calD}^{m+p/q}_xf(x) = \prescript{}{-1}{\calQ}_x^{1-p/q}\left(\frac{d^{m+1}}{dx^{m+1}}f(x)\right),
\ee
respectively. We propose an approach which achieves spectral convergence in linear complexity for a broad class of linear fractional integral equations (FIEs) and FDEs 
composed of such rational-order operators.   We demonstrate the accuracy and flexibility of the method 
on numerous examples, such as in Figure~\ref{fig:ex10} where we solve the generalised second-kind Abel integral equation 
\begin{eqnarray}
u(x) +  \lambda \int_{-1}^x\frac{u(t)}{(x-t)^{1/3}} = f(x),&& \qquad x\in[-1,1],
\end{eqnarray}
and in Figure~\ref{fig:fracairy}, where we solve a highly-oscillatory fractional Airy equation 
\begin{eqnarray}
i^{3/2} \prescript{RL}{-1}{D^{3/2}_x}u(x) - 10^{4}xu(x) = 0,&& \qquad x\in[-1,1],\qquad u(-1) = 0, \ u(1) = 1.
\end{eqnarray}

The approach is related to the ultraspherical spectral (US) method for ordinary differential equations \cite{olver2013} and singular integral equations \cite{slevinsky2016},
where the key idea is that the underlying operators are banded when represented by their action on appropriately chosen bases, built out of ultraspherical polynomials.
Here, for integral equations, the idea is similar: we exploit the fact that fractional integration is a banded operator between 
a suitable direct sum space formed of $q$  
{\em weighted} Jacobi polynomial bases
(which can be related to the ``generalised Jacobi functions'' of~\cite{chen2016} and ``polyfractinomials'' of \cite{zayernouri2014}), 
for which an explicit representation of the fractional derivative is available. 
However, a critical difficulty arises for differential equations: the 
bases are not compatible, in the sense that the weights in the range of the operators differ from those of the domain.  
To overcome this difficulty we 
expand the solution as a direct sum of weighted Jacobi polynomials as before, but then consider 
{\em another} basis formed as a direct sum of $q$ {\em different} weighted Jacobi polynomial 
bases to represent the range of the operator (for a total of 2$q$ bases).%
\footnote{When $q = 2$ the integral in~(\ref{eqn:ratint_def}) is called the left-sided {\em half}-integral, and we shall 
see below that more elegant formulae can be obtained in this instance using ultraspherical rather than Jacobi polynomials.}
If these two direct sum spaces are chosen appropriately, then the resulting operators are banded.

There have been two recent additions to the literature which also provide spectral accuracy for FDEs, namely the works of
Zayernouri and Karniadakis \cite{zayernouri2014} and Chen, Shen, and Wang \cite{chen2016}.\footnote{There has also been recent work 
in spectral methods for {\em tempered} fractional differential equations (see, for example, \cite{zhao2016}), but it is not 
clear that these approaches provide spectral accuracy in the limit $\alpha\rightarrow0$, i.e., the non-tempered case.}
The foundation of
both is the same formula for the fractional integral of weighted 
Jacobi polynomials (i.e., \cite[Theorem 6.72(b)]{andrews1999}) which 
also forms the basis of our own approach (see Theorems~\ref{thm:half_int_C} and~\ref{thm:frac_int_J} below).
Whereas in this paper we limit our attention to rational-order derivatives, both \cite{chen2016} and 
\cite{zayernouri2014} deal with arbitrary orders, and so are in a sense more general.
However, the algorithm proposed by Zayernouri and Karniadakis is collocation based, leading to dense 
matrices and ${\cal{O}}(N^3)$ complexity. Spectral accuracy is demonstrated for a few select problems, 
but it is typically sub-geometric. Furthermore, 
the discussion is limited to zero Dirichlet boundary conditions. The algorithm of Chen, Shen, and Wang
has linear complexity, but applies only to FDEs of the form $_{-1}{\cal D}^{\nu}_xu(x) = f(x)$ and 
$_{x}{\cal D}^{\nu}_1u(x) = f(x)$ (for both RL and Caputo definitions). 
In this work we shall consider FDEs which are linear combinations of rational-integer order derivatives with more 
general boundary conditions and demonstrate {\em geometric} convergence with {\em linear} complexity.

The bulk of this paper is dedicated to introducing the proposed algorithm specifically for the 
case of half-integral order integrals and derivatives (i.e., $p = 1$ and $q=2$ in~(\ref{eqn:ratint_def}) and~(\ref{eqn:ratint_def2})),
for which the approach and derivation are more intuitive to follow.
However, the extension to more general rational-order derivatives and integrals follows readily
once the approach is understood for the half-integer order case, and in the 
penultimate section we describe in some detail how this is achieved and give further examples.
As such, the outline of this paper is as follows. In Section~\ref{sec:prelim} we introduce some necessary preliminaries
regarding ultraspherical polynomials, in particular an explicit formula for their half-integrals 
and various transformations between different weighted ultraspherical polynomial expansions. 
In Sections~\ref{sec:half}--\ref{sec:half_Cap} we use these to derive a fast and 
geometrically convergent algorithm for a certain class of half-integer order FIEs and 
FDEs of Riemann--Liouville and Caputo type, respectively. 
Section~\ref{sec:comp} discusses some computational issues relating to these
algorithms, such as the efficient computation of the required polynomial coefficients and solution
of the linear systems describing the FIEs/FDEs. 
In Section~\ref{sec:ext} we describe how the ideas of the previous sections
may be adapted to consider more general rational-order FIEs, before concluding 
in Section~\ref{sec:conc} with one final example and some suggestions for future work.

{\bf Remark:}  The experiments in this paper were conducted in MATLAB (code to reproduce all figures is available online at~\cite{fracspeccode}), 
and a Julia implementation of the algorithm is available in ApproxFun.jl~\cite{ApproxFun}.

%% file: prelim.tex

In this section we consider the required preliminaries needed for working with half-integrals
\be\label{eqn:halfint_def}
     \prescript{}{-1}{{\calQ}^{1/2}_x}f(x)  = \frac{1}{\sqrt{\pi}}\int_{-1}^x\frac{f(t)}{(x-t)^{1/2}}\,dt,
\ee
and half-integer order derivatives 
\be\label{eqn:halfint_def2}
    \prescript{RL\!}{-1}{\calD}_x^{m+1/2}f(x) = \frac{d^{m+1}}{dx^{m+1}}\left(\prescript{}{-1}{\calQ}_x^{1/2}f(x)\right) \quad \text{ and} 
    \quad \prescript{C\!}{-1}{\calD}^{m+1/2}_xf(x) = \prescript{}{-1}{\calQ}_x^{1/2}\left(\frac{d^{m+1}}{dx^{m+1}}f(x)\right), \quad m \in \mathbb N.
\ee
Our primary tools here are ultraspherical polynomials, specifically Legendre and Chebyshev polynomials, as described below. 

{\bf Remark:}  We shall see in Section~\ref{sec:ext} that in the case of general rational-order integrals and derivatives one
must instead work with Jacobi polynomials. Whilst it is possible to unify these approaches and use Jacobi 
polynomials in the half-order case, we find that using the ultraspherical polynomials here leads to cleaner and more elegant
formulae, and so choose to formulate our algorithm with these instead.

\subsection{Ultraspherical polynomials}\label{subsec:jacpoly}
The ultraspherical (or Gegenbauer) polynomials, $C_n^{(\lambda)}(x)$, 
are orthogonal with respect to the weight function $(1-x^2)^{\lambda-1/2}$ on the interval $[-1, 1]$, where $\lambda > -{1 \over 2}$ and $\lambda \neq 0$. 
For any $\lambda>0$ the degree $n$ ultraspherical polynomial may be 
defined via the recurrence~\cite[18.9.1]{DLMF}
\begin{equation}\label{eqn:ultrarec}
 C_{-1}^{(\lambda)}(x) = 0, \quad C_0^{(\lambda)}(x) = 1, \quad (n+1)C^{(\lambda)}_{n+1}(x) = 2(n+\lambda)xC^{(\lambda)}_{n}(x) - (n+2\lambda-1)C^{(\lambda)}_{n-1}(x).
\end{equation}
The Legendre polynomials, $P_n(x)$, 
and the second-kind Chebyshev polynomials, $U_n(x)$, are special cases of the ultraspherical polynomials with $\lambda = \frac12$ and $\lambda=1$, respectively. 
These two will be of particular importance in our algorithms described in Sections~\ref{sec:half}--\ref{sec:half_Cap} for half-integer order FIEs and FDEs.

For any $x\in\mathbb{C}$, $\lambda>0$, and $\gamma \in \R$, we define $\Cal(x)$ as the quasimatrix --- a  ‘matrix’ whose ‘columns’ are functions defined on an interval~\cite{stewart1998} ---
whose $j$th column is the degree $(j-1)$th ultraspherical polynomial
with parameter $\lambda$ weighted by $(1+x)^\gamma$, i.e., 
\be\label{eq:basis}
    \Cal(x) \defeq \Big[(1+x)^\gamma C_0^{(\lambda)}(x), \,\,(1+x)^\gamma C_1^{(\lambda)}(x), \,\,\dots\Big].
\ee
We refer to these as {\em weighted} ultraspherical bases and note that the columns of $\Cal(x)$
are related to the ``generalised Jacobi functions'' of~\cite{chen2016} and ``polyfractinomials'' of \cite{zayernouri2014}.
With each such basis~\eqref{eq:basis} we may associate a space of coefficients, $\Cal \cong \mathbb{C}^\infty$,
and if $\underline{u} = \vectt[u_0,u_1,\dots] \in \Cal$ such that
\be\label{eqn:complicated}
    \sum_{k=0}^\infty \abs{u_k} \sup_{-1\leq x \leq 1} \abs{C_k^{(\lambda)}(x)} = 
	\sum_{k=0}^\infty \abs{u_k}\frac{\Gamma(2\lambda+k)}{\Gamma(2\lambda)k!}
    < \infty,   
\ee
then $u(x) = \Cal(x)\underline{u}$ defines a continuous function away from $x = -1$. For convenience, we denote 
${\mathbf{P}_\gamma}:=\mathbf{C}_\gamma^{(1/2)}$, ${\mathbf{U}_\gamma}:={\mathbf{C}}^{(1)}_\gamma$, and ${\mathbf{C}}^{(\lambda)}:={\mathbf{C}}_0^{(\lambda)}$.

Linear operators which can be applied to one such weighted ultraspherical basis and expanded in another induce 
infinite-dimensional matrices that can  be viewed as  acting between different $\Cal$ spaces.    
For example, given a continuous linear operator ${\cal L} : X \rightarrow Y$ so that $(1+x)^{\lambda}C_k^{(\lambda)}(x) \in X$ and  $(1+x)^{c}C_j^{(\ell)}(x) \in Y$ 
with the property
\be
    {\cal L}[(1+\diamond)^\gamma C_k^{(\lambda)}](x)= \sum_{j=k-m}^{k+m} L_{jk} (1+x)^c C_j^{(\ell)}(x),
\ee
we  can associate it with an $m$-{\em banded} (i.e., banded  with bandwidth $m$) infinite-dimensional matrix
\be\label{eqn:L}
	L \defeq \sopmatrix{L_{00} & \cdots & L_{0m} \D\cr 
			\vdots  & \ddots & L_{1m} & L_{1,m+1}\D\cr 
			L_{m0} & L_{m1} &  \ddots & L_{mm} & \ddots \cr
					  & L_{m+1,1} & \ddots & \ddots& \ddots \cr
					  && \ddots & \ddots& \ddots
					  }.
\ee
Since $L$ is banded, multiplication is a well-defined operation on $\mathbb{C}^\infty$ and~(\ref{eqn:L}) can be viewed as an operator $L: \Cal \rightarrow \mathbf{C}_c^{(\ell)}$.  To relate the operator $L$ and the operator ${\cal L}$ we note that, by construction, we have\footnote{Here and throughout we use $\diamond$ to represent the dummy variable in an operator.}
\be
	 {\cal L}[(1+\diamond)^\gamma C_k^{(\lambda)}](x)= {\cal L} \Cal(x) \underline{e}_k= \mathbf{C}^{(\ell)}_c(x) L \underline{e}_k.
\ee
If $u(x) \in X$, then, assuming that the $\Cal(x) \underline{e}_k$ are dense in $X$, there exists $\underline{u}\in\Cal$ so that $u(x) = \Cal(x) \underline{u}$.  Because ${\cal L}$ is continuous, we have 
\be\label{eqn:mapsto}
	{\cal L} u = {\cal L} \Cal(x)\underline{u} = \mathbf{C}^{(\ell)}_c(x) L \underline{u},
\ee
and therefore applying $L$ to $u(x)$ is equivalent to applying ${\cal L}$ to $\underline{u}$.

The US method~\cite{olver2013} for differential equations requires three such banded operators which act on ultraspherical polynomials: conversion, multiplication, and differentiation. 
We now revisit these operators in the case of weighted ultraspherical polynomials and introduce new operators corresponding to fractional integration and fractional differentiation of half-integer order.  

{\bf Remark:} In Sections~\ref{sec:half}--\ref{sec:half_Cap} we will seek solutions to FIEs and FDEs formed as linear combinations of Legendre polynomials, $P_n(x)$, and {weighted} Chebyshev polynomials 
of the second kind, $\sqrt{1+x}U_n(x)$. Another possibility 
is to choose a direct sum of Legendre polynomials and weighted Chebyshev polynomials of the {\em first} kind, $T_n(x)/\sqrt{1+x}$ (for which one 
can also find explicit and compact formulae for half-integer order integrals and derivatives).
We make the decision to use second-kind polynomials for the following reasons: Firstly,  $T_n(x)$ is not an ultraspherical polynomial.
In particular, this means that the formulae involving $T_n(x)$ in the next few sections must be treated separately from $C_n^{(\lambda)}(x)$,
which greatly clutters the exposition. Secondly, in most applications of interest the solution remains finite, 
so a basis which remains bounded in the computational interval is preferred. (See also the remark in Section~\ref{subsec:rlhalf}.)

\subsection{Conversion operators}
We consider two representations of the identity operator, $\mathcal{I}$, which map between different $\Cal$ spaces. First, the relationship~\cite[18.9.7]{DLMF}
\be\label{eqn:C_convert}
C^{(\lambda)}_{n}(x) = \frac{\lambda}{n+\lambda}\big(C^{(\lambda+1)}_{n}(x)-C^{(\lambda+1)}_{n-2}(x)\big),
\ee
induces operators $S_\lambda:\Cl[\lambda]_\gamma\rightarrow\Cl[\lambda+1]_\gamma$
defined by
\be\label{eq:ultraconv}
S_\lambda \defeq \sopmatrix{ 1 & 0 & \shrink\frac{-\lambda}{\lambda+2} \cr 
 & \shrink\frac{\lambda}{\lambda+1} & 0 & \shrink\frac{-\lambda}{\lambda+3} \cr
  &  & \shrink\frac{\lambda}{\lambda+2} & 0 & \shrink\frac{-\lambda}{\lambda+4}\cr  &&&\shrink\ddots& \ddots&\shrink\ddots},
\ee
so that if $\underline{u}\in\Cal$ then $u(x) = \Cal(x)\underline{u} = \C^{(\lambda+1)}_\gamma S_\lambda\underline{u}$.
These are precisely the conversion operators ${\cal{S}_\lambda}$ as described in~\cite{olver2013}.
Here, and in the other operators that follow, 
when $S$ is acting on either $\P$ or $\U$, we shall subscript with these, rather than the corresponding value of $\lambda$.
That is, 
\be
S_\P \defeq S_{1/2} = \sopmatrix{ 1 & 0 & -\frac15 \cr 
 & \frac13 & 0 & -\frac17 \vphantom{\ddots}\cr
  &  & \frac15 & 0 & \ddots\cr  &&&\shrink\ddots& \ddots}
\quad \text{and } \quad 
S_\U \defeq S_{1} = \sopmatrix{ 1 & 0 & -\frac13 \cr 
 & \frac12 & 0 & -\frac14 \vphantom{\ddots}\cr
  &  & \frac13 & 0 & \ddots\cr  &&&\shrink\ddots& \ddots&}.
\ee

A second relationship
\be\label{eqn:1plusxC}
(1+x)C^{(\lambda)}_n(x) = \frac{n+1}{2(n+\lambda)}C^{(\lambda)}_{n+1}(x) + C^{(\lambda)}_n(x) + \frac{n+2\lambda-1}{2(n+\lambda)}C^{(\lambda)}_{n-1}(x), \qquad \lambda > 0,
\ee
which can be readily derived from the recurrence relations~(\ref{eqn:ultrarec}), induces operators $R_\lambda:\Cl[\lambda]_\gamma\rightarrow\Cl[\lambda]_{\gamma-1}$, where
\be
R_\lambda \defeq \frac12\sopmatrix{ 2 & \frac{2\lambda}{1+\lambda} \cr  
	    \frac{1}{\lambda} & 2 &  \shrink\frac{1+2\lambda}{2+\lambda} \vphantom{\ddots}\cr 
	    & \shrink\frac{2}{1+\lambda} & 2  & \shrink\frac{2+2\lambda}{3+\lambda}\vphantom{\ddots}\cr  
	    & & \shrink\frac{3}{2+\lambda} & 2  & \ddots\cr
	    &&&\shrink\ddots&\shrink\ddots
	    },
\ee
So that if $\underline{u}\in\Cal$ then $u(x) = \Cal(x)\underline{u} = \C^{(\lambda)}_{\gamma-1} R_\lambda\underline{u}$.
Note in particular that $R_\U$ has the simple form
\be
R_\U \defeq R_1 = \frac12\sopmatrix{ 2 & 1\cr  
	    1 & 2 &  \ddots \vphantom{\ddots}\cr 
	    &\shrink\ddots&\shrink\ddots}.
\ee

\subsection{Multiplication operators}\label{subsec:mult}As outlined in \cite{slevinsky2016} and \cite{vasil2016}, polynomial multiplication can be viewed as a banded operator
acting on $\C^{(\lambda)}$ spaces.    
In particular, the basic building block is the Jacobi operator, built out of the three-term recurrence~(\ref{eqn:ultrarec}):
\be
	xC_n^{(\lambda)}(x) = \frac{n+1}{2(n+\lambda)} C_{n+1}^{(\lambda)}(x)  + \frac{n+2\lambda-1}{2(n+\lambda)} C_{n-1}^{(\lambda)}(x).
\ee
In the language of Section~\ref{subsec:jacpoly}, this amounts to choosing ${\cal L} = x$, inducing the operator $J_{\lambda} : \Cal \rightarrow \Cal$ defined as
\be
J_\lambda \defeq \frac12\sopmatrix{ 0 & \frac{2\lambda}{1+\lambda} \cr  
	    \frac{1}{\lambda} &0  &  \shrink\frac{1+2\lambda}{2+\lambda} \cr 
	    & \shrink\frac{2}{1+\lambda} &  0 & \shrink\frac{2+2\lambda}{3+\lambda}\cr  
	    & & \shrink\frac{3}{2+\lambda} &0   & \ddots\cr
	    &&&\shrink\ddots&\shrink\ddots},
 \ee
so that if $\underline{u}\in\Cal$ then $xu(x) = x\Cal(x)\underline{u} = \C^{(\lambda)}_\gamma J_\lambda\underline{u}$.
If $C^{(\ell)}_n(x)$ is another ultraspherical polynomial 
then its corresponding three-term recurrence applied to $J_\lambda$ gives 
\be
	C^{(\ell)}_{n+1}(J_\lambda) =  2\frac{n+\ell}{n+1} J_\lambda C^{(\ell)}_n(J_\lambda)  - \frac{n+2\ell-1}{n+1}C^{(\ell)}_{n-1}(J_\lambda), \qquad n \ge 1
\ee
with $C^{(\ell)}_{-1}(J_\lambda) = 0$, $C^{(\ell)}_{0}(J_\lambda)=1$, and the multiplication operator $\Pi_{\lambda}[C^{(\ell)}_n]:\C^{(\lambda)}_\gamma\rightarrow\C^{(\lambda)}_\gamma$ may be defined recursively as 
\be\label{eqn:multell}
	\Pi_{\lambda}[C^{(\ell)}_{n+1}] =  2\frac{n+\ell}{n+1} J_\lambda \Pi_{\lambda}[C^{(\ell)}_n]  - \frac{n+2\ell-1}{n+1}\Pi_{\lambda}[C^{(\ell)}_{n-1}], \qquad n \ge 1
\ee
where $\Pi_{\lambda}[C^{(\ell)}_{-1}] = 0$ and $\Pi_{\lambda}[C^{(\ell)}_{0}] = I$. 
Each term in the recursion will increase by the bandwidth by 1 (since $J_\lambda$ has bandwidth 1), so $\Pi_{\lambda}[C^{(\ell)}_{d}]$
is banded with bandwidth $d$. Then, by linearity of $\Pi_{\lambda}$ and the orthogonality of ultraspherical polynomials, given any degree $d$ polynomial $p$ we may construct  
\be
\Pi_{\lambda}\Big[p(x) = \sum_{n=0}^{d}p_n  C_n^{(\ell)}(x)\Big] = \sum_{n=0}^d p_n\Pi_{\lambda}[C_n^{(\ell)}],
\ee
which also has bandwidth $d$ and satisfies $p(x)u(x) = p(x)\Cal(x)\underline{u} = \C^{(\lambda)}_\gamma(x)\Pi_\lambda[p(x)]\underline{u}$ when $\underline{u}\in\Cal$.

One may take $\ell=\lambda$,  
in which case the columns of $\Pi_\lambda[C_d^{(\lambda)}]$ give rise to linearisation 
formulae for products of the form $C_d^{(\lambda)}(x)C_n^{(\lambda)}(x)$~\cite[18.18.22]{DLMF}.
Alternatively, one can construct a similar recurrence relationship based on Chebyshev polynomials of the first kind, in which case
\be
	\Pi_{\lambda}[T_{n+1}] =  2J_\lambda \Pi_{\lambda}[C^{(\ell)}_n]  - \Pi_{\lambda}[T_{n-1}],  \qquad n > 1,
\ee
with $\Pi_{\lambda}[T_{0}] = 1$, $\Pi_{\lambda}[T_{1}] = x$,  and
\be
\Pi_{\lambda}\Big[p(x) = \sum_{n=0}^{d}p_n  T_n(x)\Big] = \sum_{n=0}^d p_n\Pi_{\lambda}[T_n].
\ee
Since Chebyshev coefficients of a polynomial are readily computed by a discrete cosine transform, this can 
often be more convenient than~(\ref{eqn:multell}). When $p$ is not a polynomial but a sufficiently differentiable function, 
we can approximate it to high accuracy by a polynomial.  
In particular, if $p$ is analytic in some neighbourhood of $[-1,1]$, then the polynomial
approximation will converge geometrically, and the degree of the approximant (and hence bandwith of $\Pi_{\lambda}$)
will typically be small~\cite{ATAP}.


\subsection{Integral operators}\label{subsec:halfint}


The foundation of our approach is the following formula, which shows how the half-integral of
certain weighted ultraspherical polynomials may be computed in closed form:
\begin{theorem}\label{thm:half_int_C}
For any $\lambda > 0, n\ge 0$,
 \be\label{eqn:half_int_C}
    _{-1}\calQ_x^{1/2}[(1+\diamond)^{\lambda-1/2}C^{(\lambda)}_n](x) = \frac{\Gamma(\lambda+1/2)}{\Gamma(\lambda)(n+\lambda)}(1+x)^\lambda\big(C_n^{(\lambda+1/2)}(x)-C_{n-1}^{(\lambda+1/2)}(x)\big),
 \ee
\end{theorem}%
\begin{proof}
 Follows from relating $C_n^{(\lambda)}(x)$ to the Jacobi polynomial $P_n^{(\lambda-1/2,\lambda-1/2)}(x)$ 
 and using the closed form expression for fractional integrals of weighted Jacobi polynomials~\cite[Theorem 6.72(b)]{andrews1999}. 
 Applying the symmetric version of \cite[18.9.5]{DLMF} to the right-hand side and converting $P^{(\lambda, \lambda)}_n(x)$ 
 back to $C^{(\lambda+1/2)}_n(x)$ yields the required result.
\end{proof}

\begin{corollary}\label{cor:half_int}
 \be\label{eqn:half_int_P}
    _{-1}\calQ_x^{1/2}P_n(x) = \frac{2\sqrt{1+x}}{\sqrt{\pi}(2n+1)}\big({U_{n}(x)-U_{n-1}(x)}\big)
 \ee
 and
 \be\label{eqn:half_int_U}
    _{-1}\calQ_x^{1/2}[\sqrt{1+\diamond}U_n](x) = \frac{\sqrt{\pi}}{2}\big({P_{n+1}(x)+P_{n}(x)}\big)
 \ee
\end{corollary}%
\begin{proof}
 The first follows immediately from setting $\lambda=1/2$ in~(\ref{eqn:half_int_C}).
 For the second, take $\lambda=1$ in~(\ref{eqn:half_int_C}) and make the observation (see Appendix~\ref{sec:appendixA}) that $n\big(P_n(x)+P_{n-1}(x)\big) = (1+x)\big(C_{n-1}^{(3/2)}(x)-C_{n-2}^{(3/2)}(x)\big)$.
\end{proof}
We may therefore, in the language of Section~\ref{subsec:jacpoly}, 
consider half-integration as a banded operator between 
the spaces of Legendre polynomials and weighted Chebyshev polynomials, and define the associated 
banded half-integer order integral operators
$Q^{1/2}_\P:\P\rightarrow\U_{1/2}$
and
$Q^{1/2}_{\U_{1/2}}:\U_{1/2}\rightarrow\P$ 
as
\be\label{eqn:Qmats}
	Q^{1/2}_\P \defeq \frac{2}{\sqrt{\pi}} 
	\sopmatrix{ 1 		& \!-\frac{1}{3}\vphantom{\ddots} & \cr  
				& \!\phantom{-}\frac{1}{3} & -\frac{1}{5}\vphantom{\ddots}&\phantom{-1}\cr 
				&& \phantom{-}\frac{1}{5}\vphantom{\ddots} & \ddots\cr
		    \phantom{\ddots}&\phantom{\ddots}&\phantom{\ddots}&\ddots	
	} \quad \text{and } \quad
	Q^{1/2}_{\U_{1/2}} \defeq \frac{\sqrt{\pi}}{2} 
	\sopmatrix{ 1 				\vphantom{\ddots}\cr 
		    1 		& 1 		\vphantom{\ddots}\cr 
				& 1 	& 1	\vphantom{\ddots}\cr  
		    \phantom{\ddots}&\phantom{\ddots}&\ddots&\ddots}, \qquad 
\ee
respectively. Therefore, letting $\underline{u}_\P\in\P$ and $\underline{u}_{\U_{1/2}}\in\U_{1/2}$ then we have that $_{-1}Q^{1/2}_x\P(x)\underline{u}_\P = \U_{1/2}(x)Q^{1/2}_\P\underline{u}_\P$
and $_{-1}Q^{1/2}_x\U_{1/2}(x)\underline{u}_{\U_{1/2}} = \P(x)Q^{1/2}_{\U_{1/2}}\underline{u}_{\U_{1/2}}$.\footnote{
Henceforth, we cease (with a few exceptions) to explicitly state such equalities for each operator we introduce. It should be clear from the context 
which continuous operator is in question, and the range and domain of the discrete operator 
from the notation introduced in~(\ref{eqn:mapsto}).}

If we define $Q_\P \defeq Q^{1/2}_{\U_{1/2}} Q^{1/2}_\P$ so $Q_\P:\mathbf{P}\rightarrow\mathbf{P}$ 
is given by 
\be\label{eqn:Q1}
Q_\P = 
    \sopmatrix{ 1 & -\frac{1}{3} & \vphantom{\ddots}\cr 
		1 &0 & -\frac{1}{5} & \vphantom{\ddots}\cr 
		  & \frac{1}{3} &0& -\frac{1}{7} & \vphantom{\ddots}\cr  
		\phantom{\ddots}&\phantom{\ddots}&\frac{1}{5}&0 & \ddots \cr 
		&&&\ddots & \ddots
	},
\ee
we see that this is consistent with the relation
\be
 _{-1}\calQ_x^1P_n(x) = \int_{-1}^xP_n(t)\,dt = \begin{cases}
						    \frac{1}{2n+1}\big(P_n(x) - P_{n-2}(x)\big), & n \ge 1,\\
						    P_1(x) + P_0(x), & n = 0,
                                                \end{cases}
\ee
for the integral of Legendre polynomials (which can be obtained from \cite[18.16.1]{DLMF} and \cite[18.9.6]{DLMF}).
We may go farther and repeatedly combine the $Q_\P$ and $Q_{\U_{1/2}}$ operators to define banded 
operators 
between the spaces $\P$ and $\U_{1/2}$ representing integral operators of half-integer order,
by repeatedly applying the matrices~(\ref{eqn:Qmats}), 
i.e., $_{-1}\calQ^{m}_x$ and $_{-1}\calQ^{m+1/2}_x$. In particular, we have 
\be\label{eqn:Qm05}
Q^{m+1/2}_\P := Q^{1/2}_\P\big(Q^{1/2}_{\U_{1/2}}Q^{1/2}_\P\big)^m:\P\rightarrow\U_{1/2}
\ee
\be\label{eqn:Qm1}
Q^{m+1/2}_{\U_{1/2}} := Q^{1/2}_{\U_{1/2}}\big(Q^{1/2}_\P Q^{1/2}_{\U_{1/2}}\big)^m:\U_{1/2}\rightarrow\P.
\ee
Integral operators of integer order, $\calQ^m$, acting on these same spaces give rise to $m$-order banded operators
$Q^m_\P:\P\rightarrow\P$ and $Q^m_{\U_{1/2}}:\U_{1/2}\rightarrow\U_{1/2}$, which 
can be constructed likewise by omitting the terms outside the parentheses in~(\ref{eqn:Qm05}) and~(\ref{eqn:Qm1}), respectively, or from~(\ref{eqn:Q1}).

%
\subsection{Differentiation operators}\label{subsec:halfdiff}
The final key ingredient for the US method for ordinary differential equations is the  relationship
\be\label{eqn:diffC}
\frac{d}{dx}C^{(\lambda)}_n(x) = 2\lambda C_{n-1}^{(\lambda+1)}(x).
\ee
This induces banded derivative operators $D_\lambda:\C^{\lambda}\rightarrow \C^{\lambda+1}$, and more generally $D_\lambda^m:\C^{\lambda}\rightarrow \C^{\lambda+m}$, defined by
\be\label{eqn:Dmats}
	D_{\lambda} \defeq 2\lambda 
	\sopmatrix{  	0	& 1\vphantom{\ddots} & \cr  
				&  & 1\vphantom{\ddots}\cr 
			\phantom{\ddots}&\phantom{\ddots}&\phantom{\ddots}&\ddots	
	} \quad \text{and } \quad 
	D^{m}_\lambda \defeq 2^m\lambda^{(m)}
	\sopmatrix{  	\overbrace{0 \,\,\,\cdots\,\,\, 0}^{m \text{ times}}	& 1\vphantom{\ddots} & \cr  
				&  & 1\vphantom{\ddots}\cr 	
			\phantom{\ddots}&\phantom{\ddots}&\phantom{\ddots}&\ddots	
	},
\ee
respectively (where $\lambda^{(m)}=\lambda(\lambda+1)\ldots(\lambda+m-1)$ is the Pochammer function or ``rising factorial''). 

We now derive similar such operators for half-integer order derivatives of weighted ultraspherical polynomials.
For now we consider only the Riemann--Liouville definition, for which we have that:
\begin{corollary}\label{cor:half_diff_P}
    \be\label{eqn:half_diff_P}
	^{RL}_{-1}\calD_x^{1/2}P_n(x) = \frac{1}{\sqrt{\pi}\sqrt{1+x}}\big(U_{n}(x)+U_{n-1}(x)\big)	
     \ee
     and
     \be\label{eqn:half_diff_U}
	^{RL}_{-1}\calD_x^{1/2}\sqrt{1+x}U_n(x) = \frac{\sqrt{\pi}}{2}\big(C^{(3/2)}_n(x)+C^{(3/2)}_{n-1}(x)\big)	
     \ee
\end{corollary}%
\begin{proof}
 The second equation follows immediately from differentiating~(\ref{eqn:half_int_U}) in~Corollary~\ref{cor:half_int} using~(\ref{eqn:diffC}).
 For the first equation, differentiate the right-hand side of~(\ref{eqn:half_int_U}) via the product rule and make the observation (see Appendix~\ref{sec:appendixA}) that
 $2(1+x)\big(C^{(2)}_{n-1}(x)-C^{(2)}_{n-2}(x)\big) = nU_n(x) + (n+1)U_{n-1}(x)$.
\end{proof}

Therefore, similarly to the case of half-integrals above, we may consider half-differentiation as a banded operator
acting on $\P$ and $\U_{1/2}$, but now mapping to $\U_{-1/2}$ and $\mathbf{C}^{(3/2)}$, respectively.
In particular, we have half-derivative operators
$D^{1/2}_\P:\P\rightarrow\U_{-1/2}$ and $D^{1/2}_{\U_{1/2}}:\U_{1/2}\rightarrow\mathbf{C}^{(3/2)}$, given by the banded infinite dimensional matrices
\be\label{eqn:D05mats}
	D^{1/2}_\P \defeq \frac{1}{\sqrt{\pi}} \sopmatrix{ 1 & 1 \vphantom{\ddots}\cr & 1 & 1 \vphantom{\ddots}\cr & \phantom{\ddots}&\ddots&\ddots} \quad\text{and } \quad 
	D^{1/2}_{\U_{1/2}} \defeq \frac{\sqrt{\pi}}{2} \sopmatrix{ 1 & 1 \vphantom{\ddots}\cr & 1 & 1 \vphantom{\ddots}\cr & \phantom{\ddots}&\ddots&\ddots}.
\ee
so that $_{-1}D^{1/2}_x\P(x)\underline{u}_\P = \U_{-1/2}(x)D^{1/2}_\P\underline{u}_\P$ and $_{-1}D^{1/2}_x\U_{1/2}(x)\underline{u}_{\U_{1/2}} = \C^{(3/2)}(x)D^{1/2}_{\U_{1/2}}\underline{u}_{\U_{1/2}}$.

Since the composition of these operators is no longer a mapping between the same space, we cannot construct higher-order 
derivatives by repeated multiplication as we did higher-order integral operators, i.e.,~(\ref{eqn:Qm05} and~(\ref{eqn:Qm1}). 
However, applying~(\ref{eqn:diffC}) $m$ times to $P_n(x)$ and~(\ref{eqn:half_diff_U}), we may readily write
\be
\frac{d^m}{dx^m}P_n(x) = 2^m\pr(\tfrac12)^{(m)}C^{(m+1/2)}_{n-m}(x) = \frac{2^m\Gamma(m+1/2)}{\sqrt{\pi}}C^{(m+1/2)}_{n-m}(x)
\ee
and
\be
^{RL}_{-1}\calD_x^{m+1/2}\sqrt{1+x}U_n(x) = 2^m\Gamma(m+3/2)\big(C^{(m+3/2)}_{n-m}(x)+C^{(m+3/2)}_{n-m-1}(x)\big),	
\ee
and can consider derivative operators $D^{m}_\P:\P\rightarrow\C^{(m+1/2)}$ and $D^{m+1/2}_{\U_{1/2}}:\U_{1/2}\rightarrow\C^{(m+3/2)}$ defined by 
\be\label{eqn:Dmats2}
D^{m}_\P \defeq \frac{2^m\Gamma(m+1/2)}{\sqrt{\pi}}
	\sopmatrix{  	\overbrace{0 \,\,\,\cdots\,\,\, 0}^{m \text{ times}}	& 1\vphantom{\ddots} & \cr  
				&  & 1\vphantom{\ddots}\cr 
				&  & & \ddots \cr
	}
		\,\,\, \text{and } \,\,\,
D^{m+1/2}_{\U_{1/2}} \defeq 2^m\Gamma(m+3/2)
	\sopmatrix{  	\overbrace{0 \,\,\,\cdots\,\,\, 0}^{m \text{ times}}	& 1\vphantom{\ddots} & 1\cr  
				&  & 1\vphantom{\ddots} & \ddots\cr 
				&  & & \ddots \cr
	},
\ee
representing ${\cal{D}}^{m}$ and $_{-1}{\cal{D}}_x^{m+1/2}$, respectively.

{\bf Remark:} Observe that $D^{m}_\P$ and $D^{m+1/2}_{\U_{1/2}}$ are banded with bandwidths $m$ and $m+1$, respectively.

The corresponding operators for $\frac{d^m}{dx^m}\sqrt{1+x}U_n(x)$ and $^{RL}_{-1}\calD_x^{m+1/2}P_n(x)$ are complicated by the 
$\sqrt{1+x}$ weights. Here we must appeal to the product rule for differentiation 
and derive recursive formulations for $D^{m}_{\U_{1/2}}$ and $D^{m+1/2}_\P$. We first note that:
\begin{lemma}\label{thm:diff_weighted_C}
For any $n\ge0, \lambda > 0, \mu\not=0$ 
\be\label{eqn:diff_weighted_C}
    \frac{d}{dx}(1+x)^{\mu}C_n^{(\lambda)}(x) = \lambda(1+x)^{\mu-1}\left[\left(1+\frac{\mu-\lambda}{n+\lambda}\right)C_n^{(\lambda+1)}(x) + 2C_{n-1}^{(\lambda+1)}(x) + \left(1-\frac{\mu-\lambda}{n+\lambda}\right)C_{n-2}^{(\lambda+1)}(x)\right],
\ee
(where $C_{-2}^{(\lambda+1)}(x):=0)$.
\end{lemma}
\begin{proof}
Apply the product rule to the left-hand side, then use~(\ref{eqn:C_convert}), (\ref{eqn:1plusxC}), and~(\ref{eqn:diffC}).
\end{proof}\\
If we define ${D}_{\mu, \lambda}:\mathbf{C}_\mu^{(\lambda)}\rightarrow\mathbf{C}_{\mu-1}^{(\lambda+1)}$ as the differentiation operator induced by this relationship, i.e., 
\be\label{eqn:DMmats}
	{D}_{\mu,\lambda} \defeq \lambda\sopmatrix{ 1 & 2 & 1\cr & 1 & 2  & 1\phantom{\ddots}\cr && 1 & 2 &\ddots\cr  &&&\ddots&\ddots} + 
	\lambda(\mu-\lambda)\sopmatrix{ \frac{1}{\lambda} & \!0 & \!\!\!\!-\frac{1}{\lambda+2} & \phantom{\ddots}\cr & \!\!\frac{1}{\lambda+1} & \!0 & \!\!\!\!-\frac{1}{\lambda+3}\vphantom{\ddots}\cr && \!\!\frac{1}{\lambda+2} & 0 &\ddots\cr  &&&\ddots&\ddots}
\ee
we may then define
$D^{m}_{\U_{1/2}}:\U_{1/2}\rightarrow\C^{(m+1)}_{-m+1/2}$ and 
$D^{m+1/2}_\P:\P\rightarrow\C^{(m+1)}_{-m-1/2}$ as 
\be\label{eqn:prodP}
    D^{m}_{\U_{1/2}} \defeq \prod_{k=0}^{m-1} {{D}}_{-k+\frac12,k+1}
    \quad \text{and } \quad
    D^{m+1/2}_\P \defeq \left(\prod_{k=0}^{m-1} {{D}}_{-k-\frac12,k+1}\right)D^{1/2}_\P,
\ee
respectively. 

{\bf Remark:} Since the $D_{\mu,\lambda}$ operators each have bandwidth 2 (and recalling that $D^{1/2}_{\P}$ has bandwidth 1), it is readily verified that 
$D^{m}_{\U_{1/2}}$ and $D^{m+1/2}_{\P}$ will have bandwidths $2m$ and $2m+1$, respectively.
\subsection{Block operators}\label{subsec:blockops}
As we shall see in the next section, our approach for solving FIEs and FDEs of
half-integer order will be to seek solutions formed 
as a direct sum of two different weighted ultraspherical polynomials, 
namely $\P\oplus\U_{1/2}$. Here we 
introduce some notation to clarify
the exposition in the description of the algorithm that follows. 

Firstly, suppose $\A$ and $\B$ are two different $\C^{(\lambda)}_\gamma$ spaces. We define $[\A\oplus \B](x) := [\A(x), \B(x)]$
and if $\underline{u} = [\underline{a}^\top, \underline{b}^\top]^\top$ where
$\underline{a}\in A$ and $\underline{b}\in B$ then we say $\underline{u}\in \A\oplus \B$ and may write
\be
u(x) = [\A(x), \B(x)](x)\underline{u} = \sum_{n=0}^\infty a_nA_n(x) + \sum_{n=0}^\infty b_nB_n(x).
\ee
Then, for any $m\in\mathbb{N}$ we define
\be\label{eqn:blockQ}
  \hspace*{-30pt}Q_\PoU^{m/2} \defeq \left(\begin{array}{c c} 0 & Q_{\U_{1/2}}^{1/2} \\ Q_\P^{1/2} & 0\end{array}\right)^m  :\P\oplus\U_{1/2}\rightarrow\P\oplus\U_{1/2},
\ee
\be\label{eqn:blockD}
    \qquad\qquad D_\PoU^{m} \defeq  \left(\begin{array}{c c} D_\P^{m} & 0\\ 0 & D_{\U_{1/2}}^m\end{array}\right) :\P\oplus\U_{1/2}\rightarrow\C^{(m+1/2)}\oplus\C^{(m+1)}_{-m+1/2},
\ee
and
\be\label{eqn:blockD05}
  D_\PoU^{m+1/2} \defeq \left(\begin{array}{c c} 0 & D_\U^{m+1/2} \\ D_\P^{m+1/2} & 0\end{array}\right) :\P\oplus\U_{1/2}\rightarrow\C^{(m+3/2)}\oplus\C^{(m+1)}_{-m-1/2},
\ee
corresponding to half-integer order integral and derivative operators, respectively. We then have, for example, 
that if $\underline{u}=[\underline{u}_\P^\top, \underline{u}_{\U_{1/2}}^\top]^\top$, with $\underline{u}_\P\in\P$ and $\underline{u}_{\U_{1/2}}\in\U_{1/2}$, then
\be
{\cal{Q}}^{m/2}u(x) = {\cal{Q}}^{m/2}[\P(x), \U_{1/2}(x)]\underline{u} = [\P(x), \U_{1/2}(x)]Q^{m/2}\underline{u}.
\ee
and
\be
{\cal{D}}^{m}u(x) = {\cal{D}}^{m}[\P(x), \U_{1/2}(x)]\underline{u} = [\C^{(m+1/2)}(x),\C^{(m+1)}_{-m+1/2}(x)]D^{m}\underline{u}.
\ee

For convenience we also introduce the block conversion operators $E_m:\C^{(\ell)}_{\gamma_1}\oplus\C^{(m)}_{\gamma_2}\rightarrow\C^{(\ell)}_{\gamma_1}\oplus\C^{(m+1)}_{\gamma_2}$
and $E_{m+1/2}:\C^{(m+1/2)}_{\gamma_1}\oplus\C^{(m+1)}_{\gamma_2}\rightarrow\C^{(m+3/2)}_{\gamma_1}\oplus\C^{(m+1)}_{\gamma_2-1}, \,\, m \in \mathbb{N}^+$ defined by 
\be\label{eqn:blockE}
E_m \defeq \left(\begin{array}{c c} I & 0 \\ 0 & S_{m}\end{array}\right)
\quad \text{and } \quad E_{m+1/2} \defeq \left(\begin{array}{c c} S_{m+1/2} & 0 \\ 0 & R_{m+1}\end{array}\right).
\ee
{\bf Remark:} Note that all of the operators in equations~(\ref{eqn:blockQ})--(\ref{eqn:blockE}) are banded or block-banded.  Block-banded matrices become banded when the coefficients are interleaved, which is expanded on below.

As described in Section~\ref{subsec:mult}, polynomial multiplication also results in banded operators. In particular, 
multiplication by a polynomial $r(x)\in\mathbb{P}^d$ yields the following $d$-banded operator  $\Pi_0[r]:\P\oplus\U_{1/2}\rightarrow\P\oplus\U_{1/2}$,
\be
\Pi_0[r]_\PoU \defeq 
     \left(\begin{array}{c c}
	    \Pi_\P[r] &\\
	    & \Pi_\U[r]
     \end{array}     \right).
\ee
More generally, for integer values $m$, we define multiplication operators
\be
\Pi_{m}[r]_\PoU \defeq 
     \left(\begin{array}{c c}
	    \Pi_{m+1/2}[r] &\\
	    & \Pi_{m+1}[r]
     \end{array}     \right), \qquad 
\Pi_{m+1/2}[r]_\PoU \defeq 
     \left(\begin{array}{c c}
	    \Pi_{m+3/2}[r] &\\
	    & \Pi_{m+1}[r]
     \end{array}     \right),
\ee
which act on the appropriate direct sum spaces.

Multiplication by square root-weighted polynomials, $\sqrt{1+x}s(x)$, $s\in\mathbb{P}^d$, is complicated by the need 
to convert between $\C^{(\lambda)}$ and $\C^{(\lambda+1/2)}$ bases.
For example, in the case of $\sqrt{1+x}s(x)$ multiplying a vector in $\P\oplus\U_{1/2}$ 
we require the upper triangular conversion or ``connection'' 
operators $\hat M:\P\rightarrow\U$ and $\hat L:\U\rightarrow\P$~\cite{alpert1991,townsend2016} so that\footnote{We use $\hat L$ and $\hat M$ here as $L$ and $M$ are typically used to denote the conversion
operators $M:\P\rightarrow\T$ and $L:\T\rightarrow\P$, respectively, where $\T$ is the quasimatrix whose columns are formed
of Chebyshev polynomials of the {first} kind, $T_n(x)$.}
\be
\Pi_0[0,s]_\PoU \defeq 
     \left(\begin{array}{c c}
	     & \hat L\Pi_\U[(1+\diamond)s] \vphantom{Q_\U^{1/2}}\\
	    \Pi_\U[s]\hat M & 
     \end{array}     \right).
\ee
More generally, multiplication by $r(x) + \sqrt{1+x}s(x)$ yields the operator $\Pi_0[r,s]_\PoU:\P\oplus\U_{1/2}\rightarrow\P\oplus\U_{1/2}$, 
\be
\Pi_0[r,s]_\PoU \defeq 
     \left(\begin{array}{c c}
	    \Pi_\P[r] & \hat L\Pi_\U[(1+\diamond)s] \vphantom{Q_\U^{1/2}}\\
	    \Pi_\U[s]\hat M & \Pi_\U[r] \vphantom{Q_\U^{1/2}}
     \end{array}     \right).
\ee

{\bf Remark:} $\Pi_0[0,s]$ and $\Pi_0[r,s]$ are neither banded or block-banded, however they are (upon re-ordering) {\em lower}-banded. 
We give an example with such a weighted non-constant coefficient below, but will otherwise limit our attention to the case
when the non-constant coefficients are smooth (i.e., well-approximated by an unweighted polynomial).

%% file: half.tex
We now use the operators described above to derive an algorithm for integral equations of half-integer order.

\subsection{Half-integral Equations}\label{subsec:halfinteqns}
We  first consider Abel-like integral equations of the form
\be\label{eqn:halfinteq}
\sigma u(x) + _{-1}\calQ_x^{1/2}u(x) = e(x) + \sqrt{1+x}f(x), \qquad x\in[-1,1],
\ee
where $e(x)$ and $f(x)$ are smooth (typically analytic in some neighbourhood of $[-1,1]$) and $\sigma > 0$.

Motivated by the block operators in Section~\ref{subsec:blockops}, we make the ansatz that the solution $u(x)$ may be expressed as a linear combination 
of Legendre polynomials and weighted Chebyshev polynomials:\footnote{Formally this direct sum-space defines a frame\cite{christensen2003}. We discuss the consequences of this in Section~\ref{subsec:rhs} and Appendix~\ref{subsec:convergence}.}
\be\label{eqn:ansatz}
u(x) = \sum_{n=0}^\infty a_nP_n(x) + \sqrt{1+x}\sum_{n=0}^\infty b_nU_n(x).
\ee
Assuming the coefficients $a_0,a_1,\ldots$ and $b_0,b_1,\ldots$ satisfy the conditions~\eqref{eqn:complicated}, we may, in the language of Section~\ref{sec:prelim}, write
\be\label{eqn:ansatz2}
u(x) = \left[\P(x), \U_{1/2}(x)\right]\left(\begin{array}{c}\underline{a}\\ \underline{b}\end{array}\right).
\ee
Applying the block half-integral operator~(\ref{eqn:blockQ}) with $m=0$, we have that
\begin{eqnarray}\label{eqn:Qhalfu}
_{-1}\calQ_x^{1/2}u(x) &=& \left[\P(x),\U_{1/2}(x)\right]Q_\PoU^{1/2}\left(\begin{array}{c}\underline{a}\\ \underline{b}\end{array}\right)
\end{eqnarray}
and hence
\be
\sigma u(x) + _{-1}\calQ_x^{1/2}u(x) = \left[\P(x), \U_{1/2}(x)\right]\left(\sigma I + Q_\PoU^{1/2}\right)
            \left(\begin{array}{c}\underline{a}\\ \underline{b}\end{array}\right).
\ee
Letting
\be\label{eqn:eandf}
e(x) = \sum_{n=0}^\infty e_nP_n(x) \qand f(x) = \sum_{n=0}^\infty f_nU_n(x),
\ee
or equivalently 
\be\label{eqn:eandf2}
e(x) + \sqrt{1+x}f(x) = \left[\P(x), \U_{1/2}(x)\right]\left(\begin{array}{c}\underline{e}\\\underline{f}\end{array}\right),
\ee
and equating coefficients, 
we arrive at the (infinite dimensional) linear system of equations
\be\label{eqn:system_int}
\left(\sigma I + Q_\PoU^{1/2}\right)
            \left(\begin{array}{c}\underline{a}\vphantom{Q_\P^{1/2}}\\ \underline{b}\vphantom{Q_\P^{1/2}}\end{array}\right) =
\left(\begin{array}{c c}
	    \sigma I & Q_{\U_{1/2}}^{1/2}\\
	    Q_\P^{1/2} & \sigma I
     \end{array}\right)
     \left(\begin{array}{c}
	    \underline{a}\vphantom{Q_\P^{1/2}}\\\underline{b}\vphantom{Q_\P^{1/2}}
     \end{array}\right) = 
     \left(\begin{array}{c}
	    \underline{e}\vphantom{Q_\P^{1/2}}\\\underline{f}\vphantom{Q_\P^{1/2}}
     \end{array}\right).
\ee

Note that both the diagonal and off-diagonal blocks of the operator in~(\ref{eqn:system_int}) are banded, 
and by interleaving the coefficients in $\underline{a}$ and $\underline{b}$ (i.e., $(\underline{a}^\top, \underline{b}^\top)\mapsto (a_0, b_0, a_1, b_1, \ldots)^\top$) we arrive at a banded operator (in this case, tridiagonal).
Taking a finite section approximation (i.e., truncating each of the summations in~(\ref{eqn:ansatz}) and~(\ref{eqn:eandf}) and hence the block operators in~(\ref{eqn:system_int})) at a suitable length $N$, 
we arrive at a $2N\times2N$ tridiagonal matrix system, which can be solved directly in ${\cal{O}}(N)$ 
floating point operations for the approximate coefficients $\underline{a}$ and $\underline{b}$ of $u(x)$ in~(\ref{eqn:ansatz}).
Alternatively, one can use the adaptive QR approach described in~\cite{olver2013} and solve the infinite dimensional~(\ref{eqn:system_int})
system to a required accuracy without {\em a priori} truncation (see Section~\ref{subsec:solvesystem}). 
Convergence and stability are discussed in more detail in Appendix~\ref{sec:appendixC}.

\begin{example}\label{subsec:example1} We consider the second-kind Abel integral equation
\be\label{eqn:ex1}
 u(x) +  _{-1}\calQ_x^{1/2}u(x) = 1 
\ee
with solution~\cite[Section 11.4-1]{Polyanin2008}
\be\label{eqn:ex1_sol}
 u(x) = e^{1+x}\rm{erfc}(\sqrt{1+x}),
\ee
where $\rm{erfc}$ is the complimentary error function. Note that the solution $u(x)$ takes the form 
$p(x) + \sqrt{1+x}q(x)$, where $p(x)$ and $q(x)$ are functions analytic in some neighbourhood of $[-1,1]$,
and so any attempt to approximate $u(x)$ by a polynomial $u(x)\approx p_N(x)$ or a weighted polynomial
$u(x)\approx \sqrt{1+x}q_N(x)$ will achieve only algebraic convergence as the degree of the polynomial is increased. 
However, using the direct sum basis
$\P\oplus\U_{1/2}$ we are able to approximate such a function and hence solve the FDE~\eqref{eqn:ex1} with spectral accuracy.

To solve the FDE~(\ref{eqn:ex1}) we form the linear system~(\ref{eqn:system_int}) with $\sigma=1$, $\underline{e}=(1,0,\ldots)^\top$, and $\underline{f}=(0,0,\ldots)^\top$, 
truncate at some length $N$, and solve with $\backslash$ in MATLAB. The results are shown in Figure~\ref{fig:ex1}. 
The first image shows a plot of the computed solution with $N=20$. The {\tt spy} plot in the second image verifies that, upon re-ordering, the associated
linear system~(\ref{eqn:system_int}) is tridiagonal. The third image shows two measures of the error in the approximation as $N$ is increased.\footnote{We show 
both measures here to validate the use of the second, which we employ later when a closed-form expression of the true solution is not readily available.}
The first (solid line) is computed by evaluating the approximate solution on a 100-point equally-spaced grid in the interval $[-1,1]$ using Clenshaw's algorithm and
comparing to the true solution~(\ref{eqn:ex1_sol}). Geometric convergence is observed until it plateaus at around 15 digits of 
accuracy when $N$ = 15. The second measure (dashed line) is the 2-norm difference between the coefficients of the approximated solution when 
truncating at sizes $N$ and $\lceil 1.1N \rceil$. Here the convergence 
does not plateau
, suggesting that the computed coefficients maintain good relative accuracy even when their magnitude is below machine precision.   
\begin{figure}[t]
\centering\tiny
\includegraphics[height=105pt]{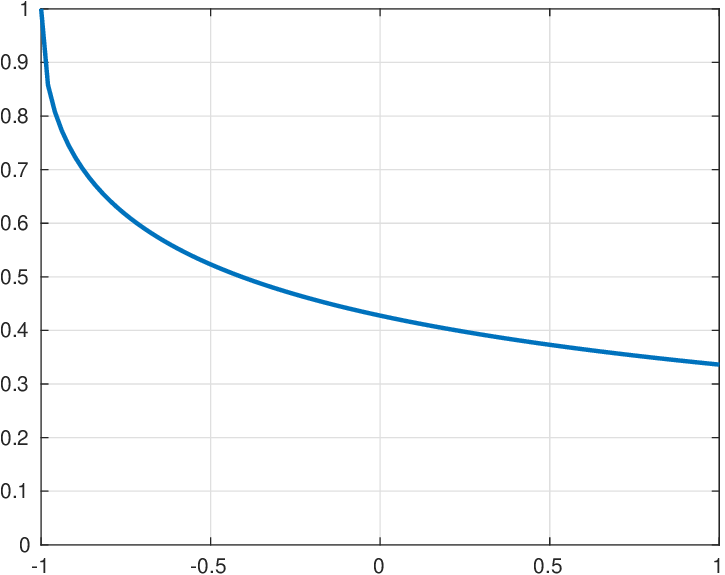}\hspace*{25pt}
\includegraphics[height=105pt,trim={0 15pt 0 1pt}]{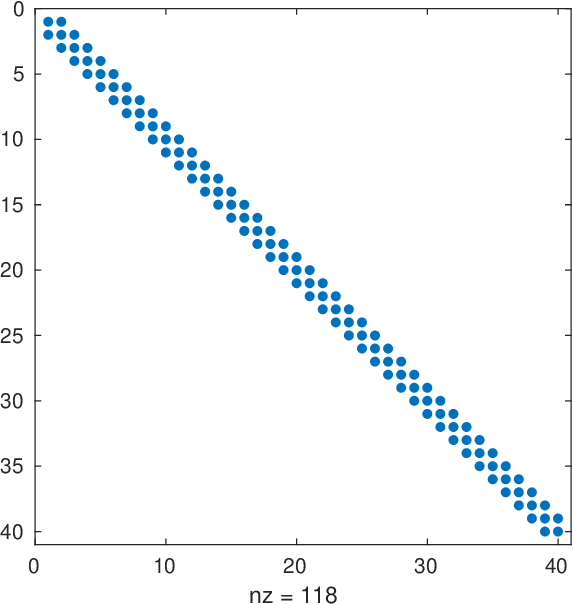}\hspace*{25pt}
\includegraphics[height=105pt]{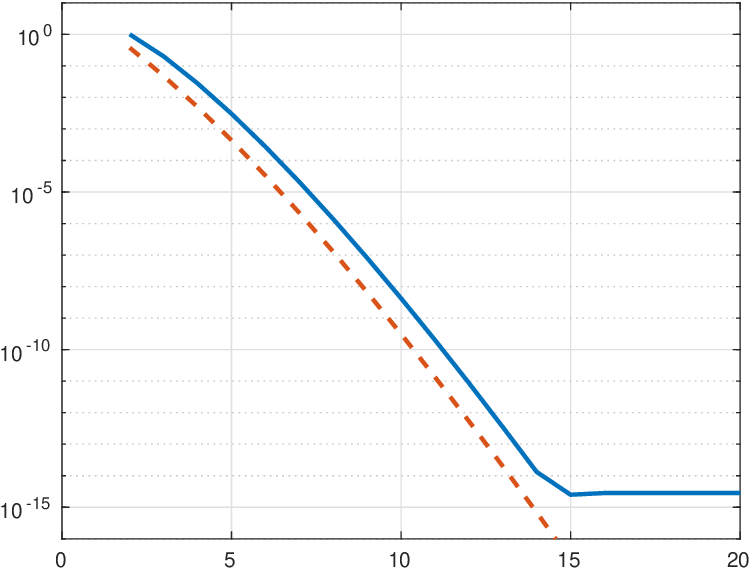}
\put(-73,107){Error}
\put(-239,107){Spy plot ($N = 20$)}
\put(-373,107){Solution}
\put(-143,45){\rotatebox{90}{error}}
\put(-438,45){\rotatebox{90}{$u(x)$}}
\put(-68,-5){$N$}
\put(-364,-5){$x$}
\caption{(a) Approximate solution to~(\ref{eqn:ex1}). (b) MATLAB {\tt spy} plot showing that truncated linear system~(\ref{eqn:system_int}) is tridiagonal. 
(c) Two measures of the error in the approximation as $N$ varies. 
Solid line: Infinity norm error of solution approximated on a 100-point equally spaced grid. 
Dashed line: 2-norm difference between the coefficients of the approximated solution when 
truncating at sizes $N$ and $\lceil 1.1N \rceil$. In both cases, geometric convergence is observed.}\label{fig:ex1}
\end{figure}
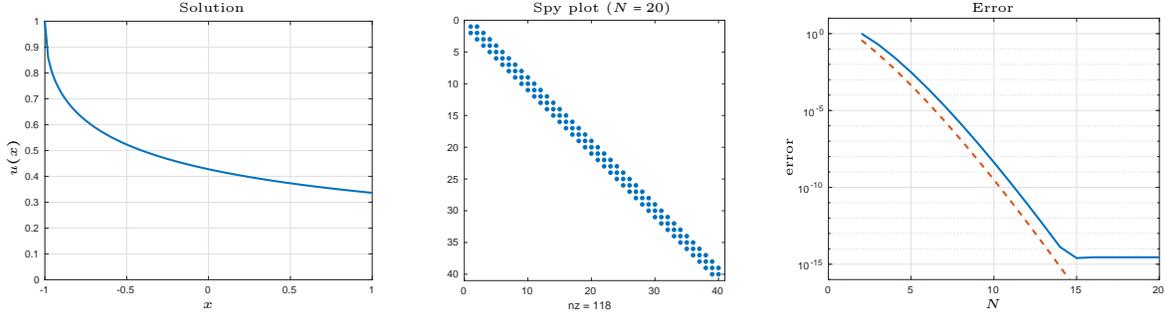
\end{example}

\subsection{Half-order integral equations with non-constant coefficients}\label{subsec:nonconstant}
In much the same way as in the US method for ODEs~\cite{olver2013}, 
the approach outlined above is readily extended to FDEs with non-constant coefficients. In particular, to solve 
problems of the form
\be
u(x) + r(x)\, _{-1}{\cal{Q}}_x^{1/2}u(x) = e(x) + \sqrt{1+x}f(x),
\ee
we can appeal to the discussion in Section~\ref{subsec:mult}
and construct multiplication operators $\Pi_\P[r]$ and $\Pi_\U[r]$, which act on
Legendre and second-kind Chebyshev series, respectively. If $r(x)$ is a polynomial
of degree $d$, then these two operators will have bandwidth $d$. If $r(x)$
is not a polynomial but is sufficiently smooth (for example, analytic),
then the coefficients in its ultraspherical polynomial expansion will decay rapidly,
and we may consider $\Pi_\P[r]$ and $\Pi_\U[r]$ as banded for practical purposes.
The resulting linear system then takes the form
\be\label{eqn:system_int_mult}
     \big(I + \Pi_0[r]Q_\PoU^{1/2}\big)
     \left(\begin{array}{c}
	    \underline{a}\\\underline{b}
     \end{array}\right) = 
     \left(\begin{array}{c}
	    \underline{e}\\\underline{f}
     \end{array}\right),
\ee
where $\Pi_0[r]$ is defined in Section~\ref{subsec:blockops}.


We may also consider problems of the form
\be
u(x) + r_1(x)_{-1}{\cal{Q}}_x^{1/2}\left[r_2u\right](x) = e(x) + \sqrt{1+x}f(x),
\ee
in which case the linear system becomes
\be\label{eqn:system_int_mult2}
\big(I + \Pi_0[r_1]Q_\PoU^{1/2}\Pi_0[r_2]\big)
     \left(\begin{array}{c}
	    \underline{a}\\\underline{b}
     \end{array}\right) = 
     \left(\begin{array}{c}
	    \underline{e}\\\underline{f}
     \end{array}\right).
\ee
Similarly, problems of the form 
\be\label{eqn:nonconstS}
u(x) + (r(x)+\sqrt{1+x}s(x))\, _{-1}{\cal{Q}}_x^{1/2}u(x) = e(x) + \sqrt{1+x}f(x).
\ee
and
\be\label{eqn:nonconstS2}
u(x) + \, _{-1}{\cal{Q}}_x^{1/2}[(r+\sqrt{1+\diamond}s)u](x) = e(x) + \sqrt{1+x}f(x).
\ee
may also be solved with linear systems
\be\label{eqn:system_int_mult4}
     \left(
     I+\Pi_0[r,s]_\PoU Q^{1/2}_\PoU
     \right)
     \left(\begin{array}{c}
	    \underline{a}\\\underline{b}
     \end{array}\right) = 
     \left(\begin{array}{c}
	    \underline{e}\\\underline{f}
     \end{array}\right),
\ee
and
\be\label{eqn:system_int_mult5}
     \left(
     I+Q^{1/2}_\PoU\Pi_0[r,s]_\PoU
     \right)
     \left(\begin{array}{c}
	    \underline{a}\\\underline{b}
     \end{array}\right) = 
     \left(\begin{array}{c}
	    \underline{e}\\\underline{f}
     \end{array}\right), 
\ee
respectively. However, these systems are no longer banded, and we will lose the linear 
complexity of our algorithm. Upon reordering they are dense above the 
diagonal and banded below, so Gaussian elimination will have ${\cal{O}}(N^2)$ complexity (see Example~\ref{subsec:ex3} below).

\begin{example}\label{subsec:ex2}
We adapt the example of Section~\ref{subsec:example1} 
so that
\be\label{eqn:ex2}
u(x) + e^{-(1+x)/2}\,_{-1}\calQ_x^{1/2}[e^{(1+x)/2} u](x) = e^{-(1+x)/2},
\ee
with solution
\be
 u(x) = e^{(1+x)/2}{\rm{erfc}}(\sqrt{1+x}).
\ee
Again taking $u(x)$ as in~(\ref{eqn:ansatz}), the resulting linear system defining the coefficients $\underline{a}$ and $\underline{b}$ is of the form
\be\label{eqn:ex2spy}
	\left( I + \Pi[e^{-(1+x)/2}]Q^{1/2}\Pi_0[e^{(1+x)/2}]\right)
     \left(\begin{array}{c}
	    \underline{a}\\\underline{b}
     \end{array}\right) = 
     \left(\begin{array}{c}
	    \underline{e}\\\underline{0}
     \end{array}\right),
\ee
where $\underline{e}$ are the Legendre coefficients of the function $e^{-(1+x)/2}$. (See Section~\ref{subsec:rhs} for discussion on how these are computed.) As in the previous example, we 
truncate each of the expansions at a suitable length $N$ and solve the resulting finite dimensional banded matrix
problem using $\backslash$ in MATLAB.

The results are shown in Figure~\ref{fig:ex2}. The left panel shows the approximate solution
computed with $N=20$. The centre panel shows a  {\tt spy} plot of the discretised, truncated,
and re-ordered operator~\eqref{eqn:ex2spy}. The non-constant coefficients in equation~\eqref{eqn:ex2}
mean the resulting matrix is no longer tridiagonal, however, it is banded
independently of $N$ and can be solved by $\backslash$ in linear time as $N\rightarrow\infty$. The precise bandwidth 
depends on the number of Chebyshev coefficients required
to approximate the non-constant coefficients, in this case $e^{\pm(1+x)/2}$, to machine precision accuracy. 
Here the number of coefficients required is around 14, and the resulting matrix has a bandwidth of approximately 43. 
The final panel shows the accuracy of the computed solution as $N$ increases, using the same two forms of the error estimate
as described in Example~\ref{subsec:example1}. Geometric convergence is again observed, but here both measures of the error
plateau due to rounding error in the computation of the Chebyshev coefficients of the functions $e^{\pm(1+x)/2}$.%
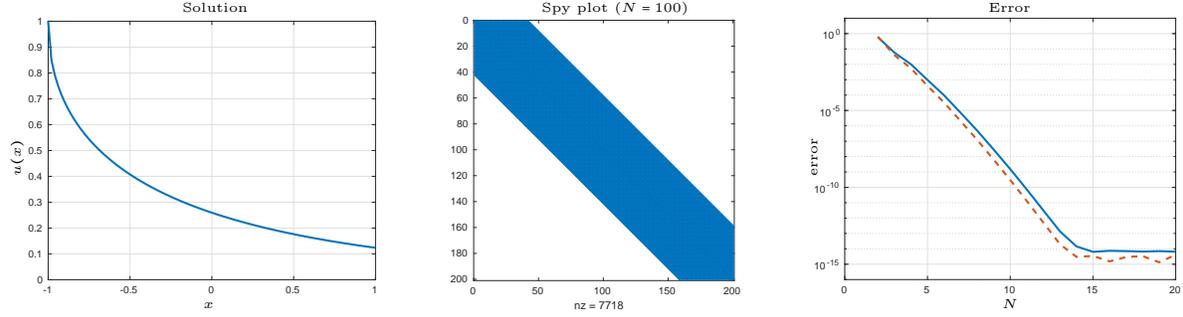
\begin{figure}[t]
\centering\tiny
\includegraphics[height=105pt]{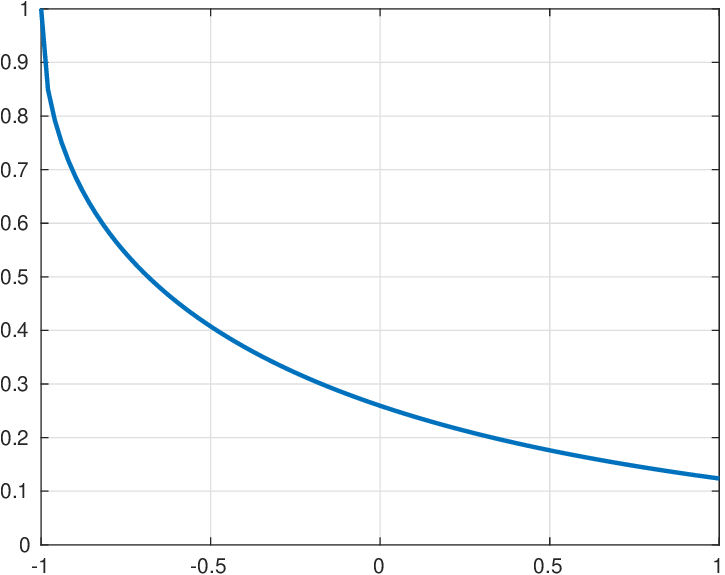}\hspace*{25pt}
\includegraphics[height=105pt,trim={0 15pt 0 1pt}]{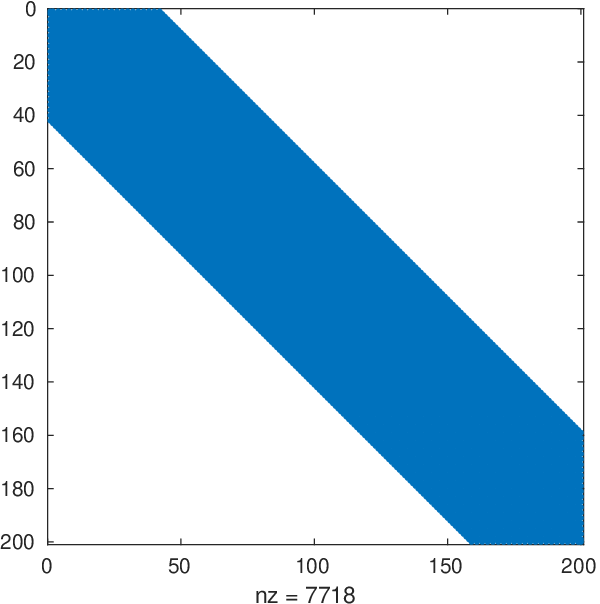}\hspace*{25pt}
\includegraphics[height=105pt]{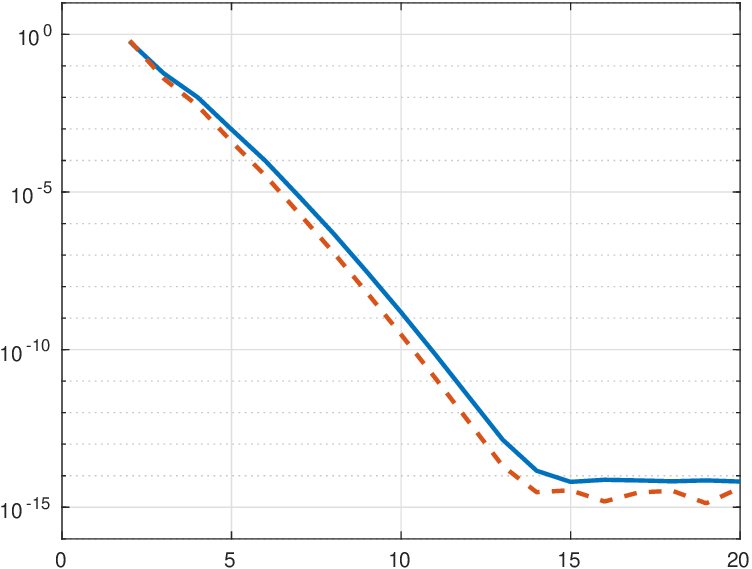}
\put(-73,107){Error}
\put(-242,107){Spy plot ($N = 100$)}
\put(-378,107){Solution}
\put(-140,45){\rotatebox{90}{error}}
\put(-443,45){\rotatebox{90}{$u(x)$}}
\put(-68,-5){$N$}
\put(-370,-5){$x$}
\caption{(a) Solution to~(\ref{eqn:ex2}). (b) MATLAB {\tt spy} plot of~(\ref{eqn:ex2spy}) showing the banded structure. 
(c) 
Solid line: Infinity norm error of solution approximated on a 100-point equally spaced grid. 
Dashed line: 2-norm difference between the coefficients of the approximated solution when 
truncating at sizes $N$ and $\lceil 1.1N \rceil$. Again, geometric convergence is observed.}\label{fig:ex2}
\end{figure}
\end{example}%
\begin{example}\label{subsec:ex3}
Here we solve
\be\label{eqn:ex3}
u(x) - {\rm erfc}{\sqrt{1+x}}\,_{-1}\calQ_x^{1/2}u(x) = 1.
\ee
The linear system satisfied by the coefficients $\underline{a}$ and $\underline{b}$ is then of the form
\be\label{eqn:ex3spy}
	\left( I + \Pi_0\!\left[-1, \frac{{\rm erf}(\sqrt{1+x})}{\sqrt{1+x}}\right]Q^{1/2}\right)
     \left(\begin{array}{c}
	    \underline{a}\vphantom{Q_0^{1/2}}\\\underline{b}\vphantom{Q_{1/2}^{1/2}}
     \end{array}\right) = 
     \left(\begin{array}{c}
	    1 \vphantom{Q_0^{1/2}}\\\underline{0}\vphantom{Q_{1/2}^{1/2}}
     \end{array}\right).
\ee
We may truncated and solve~(\ref{eqn:ex3spy}), and the results of such are shown in Figure~\ref{fig:ex3}. Here, in the {\tt spy} plot in the centre panel we see that, as 
expected, the required change of bases $\P\leftrightarrow\U_{1/2}$ are no longer banded, and hence neither is
the re-ordered version of~\eqref{eqn:ex3spy}. However, the re-ordered matrix is {\em lower}-banded, and MATLAB's $\backslash$ will  require ${\cal O}(N^2)$ operations
to solve such systems directly via Gaussian elimination.

In this case we do not know an explicit form for the solution, and so in the right panel show only the accuracy estimate based
upon comparison of successive approximations. We again observe geometric convergence in the number of degrees of freedom
until convergence plateaus at around the level of machine precision.
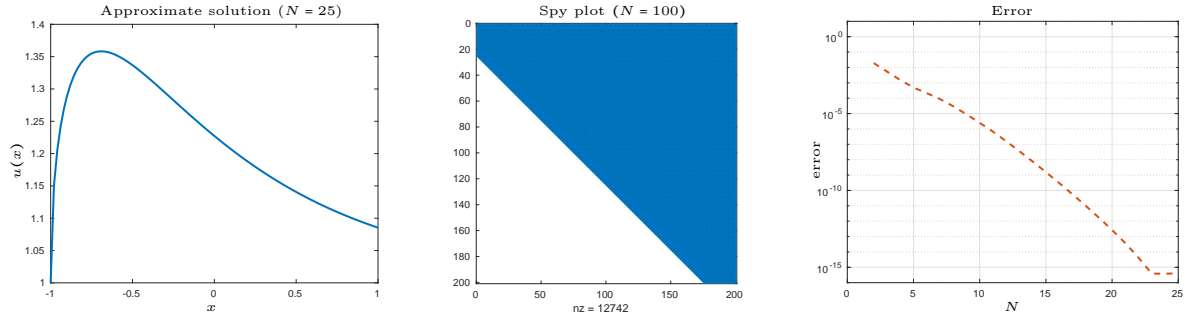
\begin{figure}[t]
\centering\tiny
\includegraphics[height=105pt]{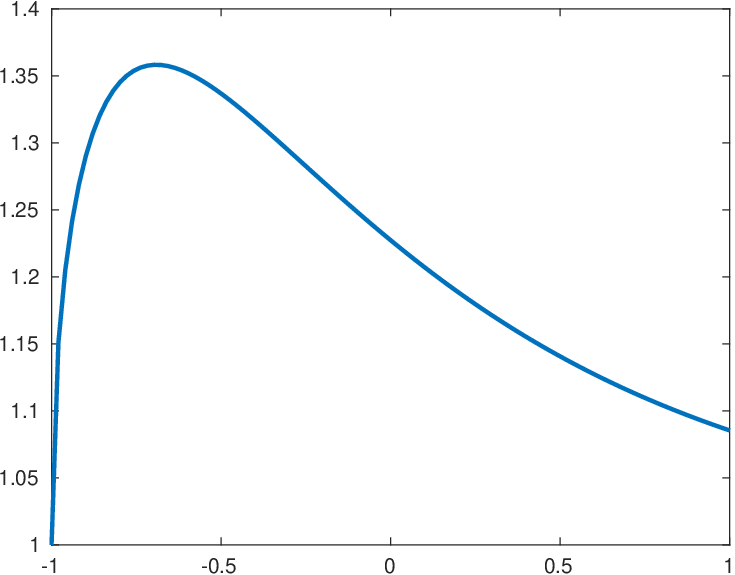}\hspace*{25pt}
\includegraphics[height=105pt,trim={0 15pt 0 1pt}]{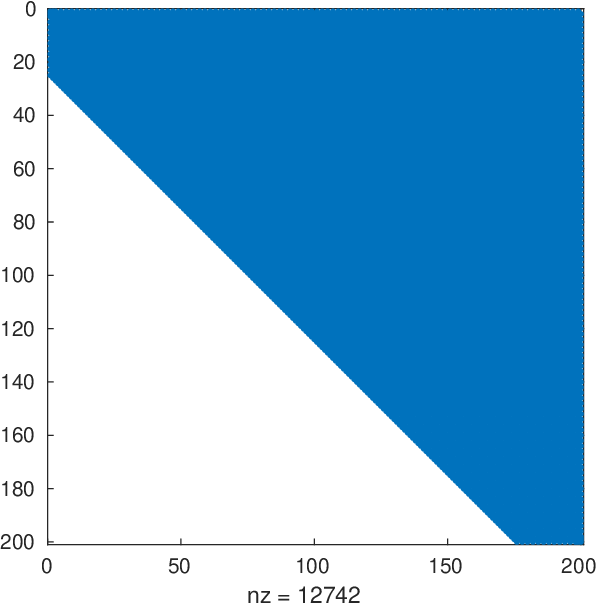}\hspace*{25pt}
\includegraphics[height=105pt]{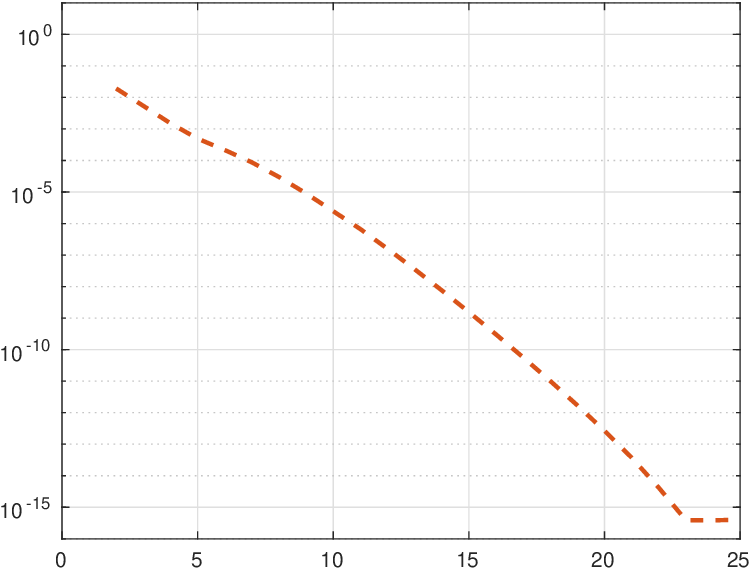}
\put(-73,107){Error}
\put(-244,107){Spy plot $(N = 100)$}
\put(-410,107){Approximate solution ($N = 25$)}
\put(-140,45){\rotatebox{90}{error}}
\put(-444,45){\rotatebox{90}{$u(x)$}}
\put(-68,-5){$N$}
\put(-370,-5){$x$}
\caption{(a) Approximate solution to~(\ref{eqn:ex3}). (b) MATLAB {\tt spy} plot of~(\ref{eqn:ex3spy}) 
showing the quasi-upper triangular structure. Such a system will require ${\cal{O}}(N^2)$ operations to invert.
Here a closed-form expression of the solution is not know, so (c) shows the 2-norm difference 
between the coefficients of the approximated solution when 
truncating at sizes $N$ and $\lceil 1.1N \rceil$. Although the system is no longer banded, 
geometric convergence is still maintained.}\label{fig:ex3}
\end{figure}
\end{example}

\subsection{Higher-Order Integral equations}
The approach outlined in the previous few examples extends readily to higher-order integral equations of half-integer order.
The general form of the problem we consider is
\be
{\cal{L}}u(x) = e(x) + \sqrt{1+x}f(x)
\ee
where 
\begin{eqnarray}
{\cal{L}}u(x) = \alpha^{[0]}(x)u(x) + \sum_{k=1}^{2m}\alpha^{[k]}(x)\,_{-1}\calQ^{k/2}_x[\beta^{[k]}u](x),\hspace*{75pt}\\
\alpha^{[k]}(x) = p^{[k]}(x) + \sqrt{1+x}q^{[k]}(x), \qquad \beta^{[k]}(x) = r^{[k]}(x) + \sqrt{1+x}s^{[k]}(x), \qquad k = 0,1,\ldots, 2m,
\end{eqnarray}
and all the functions $e(x), f(x), p^{[k]}(x), q^{[k]}(x), r^{[k]}(x), s^{[k]}(x), \,k = 0,1,\ldots,2m$ are assumed analytic in some neighbourhood of $[-1,1]$.
If we continue to take~(\ref{eqn:ansatz}) as our ansatz, 
then we arrive at the infinite dimensional linear system
\be
\left(\Pi_0[p^{[0]},q^{[0]}] + \sum_{k=1}^{2m}\Pi_0[p^{[k]},q^{[k]}]Q_\PoU^{k/2}\Pi_0[r^{[k]},s^{[k]}]\right)
     \left(\begin{array}{c}
	    \underline{a}\vphantom{Q_0^{1/2}}\\\underline{b}\vphantom{Q_{1/2}^{1/2}}
     \end{array}\right) = 
          \left(\begin{array}{c}
	    \underline{e}\vphantom{Q_0^{1/2}}\\\underline{f}\vphantom{Q_{1/2}^{1/2}}
     \end{array}\right),
\ee
where $\underline{e}$ and $\underline{f}$ are as in~(\ref{eqn:eandf}).

As before, the precise form of this operator will depend on both $m$ and the number of 
Chebyshev coefficients required to represent the functions $p^{[k]}(x), q^{[k]}(x), r^{[k]}(x), $ and $s^{[k]}(x)$. 
However, if the $q^{[k]}(x)$ and $s^{[k]}(x)$ are all identically zero, then (after re-ordering) the operator will remain banded
independently of $N$. Otherwise it will be lower-bounded, as in Example~\ref{subsec:ex3}.

\begin{example}
For simplicity, we choose a constant coefficient problem so that the banded structure of the resulting operator is readily observed. 
In particular,  we solve
\be\label{eqn:ex4}
u(x) - \,_{-1}{\cal{Q}}_x^{1/2}u(x) + \,_{-1}{\cal{Q}}_x^{1}u(x) - \,_{-1}{\cal{Q}}_x^{3/2}u(x) + \,_{-1}{\cal{Q}}_x^2u(x)= 1,
\ee
which may be expressed as
\be\label{eqn:ex4_sys}
	\left( I - Q^{1/2} + Q^{1} - Q^{3/2} + Q^{2}\right)
     \left(\begin{array}{c}
	    \underline{a}\vphantom{Q_0^{1/2}}\\\underline{b}\vphantom{Q_{1/2}^{1/2}}
     \end{array}\right) = 
     \left(\begin{array}{c}
	    1 \vphantom{Q_0^{1/2}}\\\underline{0}\vphantom{Q_{1/2}^{1/2}}
     \end{array}\right).
\ee
As in the previous examples, we truncate this operator at a given size $N$ and solve the resulting finite dimensional problem. 
The results are shown in Figure~\ref{fig:ex4} with the first and second panels showing the approximated solution 
and {\tt spy} plot of~\eqref{eqn:ex4_sys} when $N=20$, respectively. 
The third
panel shows the convergence of the solution. As in the Example~\ref{subsec:ex3} we do not know the true solution, 
so we estimate the error by the 2-norm difference between the coefficients of the approximated solution when 
truncating at sizes $N$ and $\lceil 1.1N \rceil$. We again observe geometric convergence, and as in Example~\ref{subsec:example1},
the error continues to decrease even below the level of machine precision for this constant coefficient problem.
\begin{figure}[t]
\centering\tiny
\includegraphics[height=105pt]{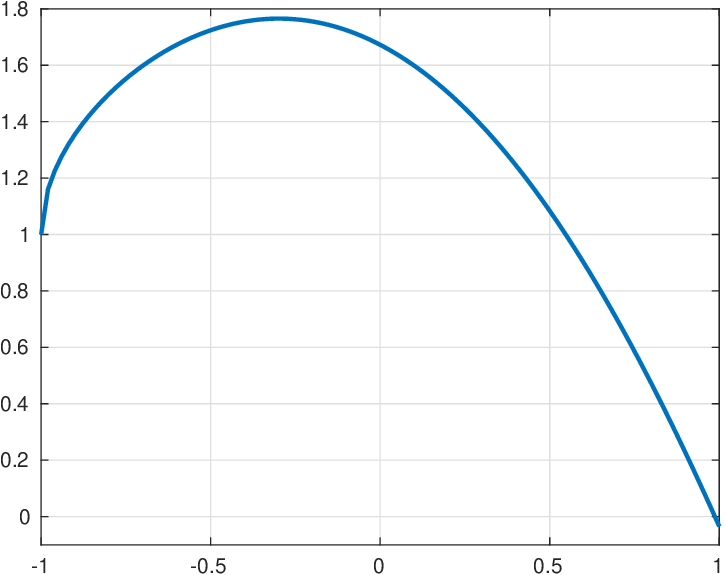}\hspace*{25pt}
\includegraphics[height=105pt,trim={0 15pt 0 1pt}]{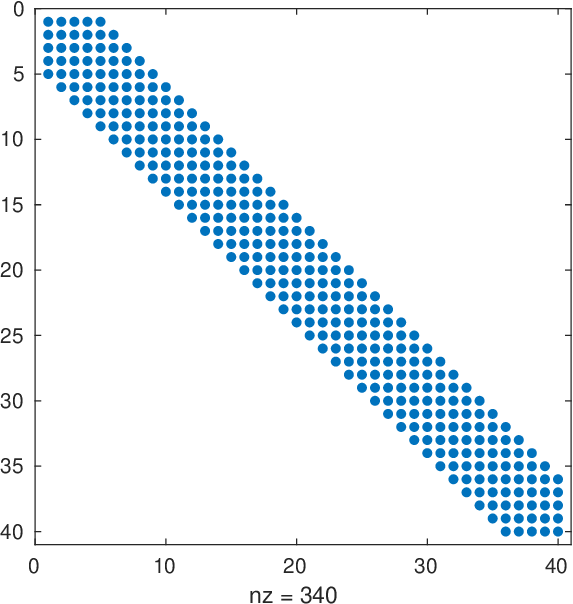}\hspace*{25pt}
\includegraphics[height=105pt]{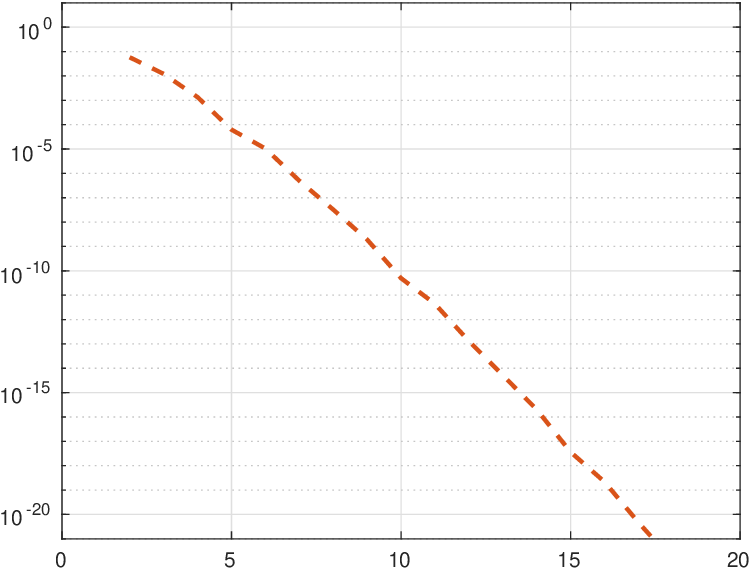}
\put(-73,107){Error}
\put(-240,107){Spy plot $(N = 20)$}
\put(-405,107){Approximate solution $(N = 20)$}
\put(-144,45){\rotatebox{90}{error}}
\put(-439,45){\rotatebox{90}{$u(x)$}}
\put(-68,-5){$N$}
\put(-365,-5){$x$}
\caption{(a) Approximate solution to~(\ref{eqn:ex4}). (b) MATLAB {\tt spy} plot of~(\ref{eqn:ex4_sys}) showing the banded structure of the linear system. 
(c) 2-norm difference between the coefficients of the approximated solution when 
truncating at sizes $N$ and $\lceil 1.1N \rceil$. Even for higher-order problems, geometric convergence is obtained.
As in Example 1, since there are no-constant coefficients whose Chebyshev coefficients must be computed, 
this measure of the error continues to converge below machine precision.}\label{fig:ex4}
\end{figure}
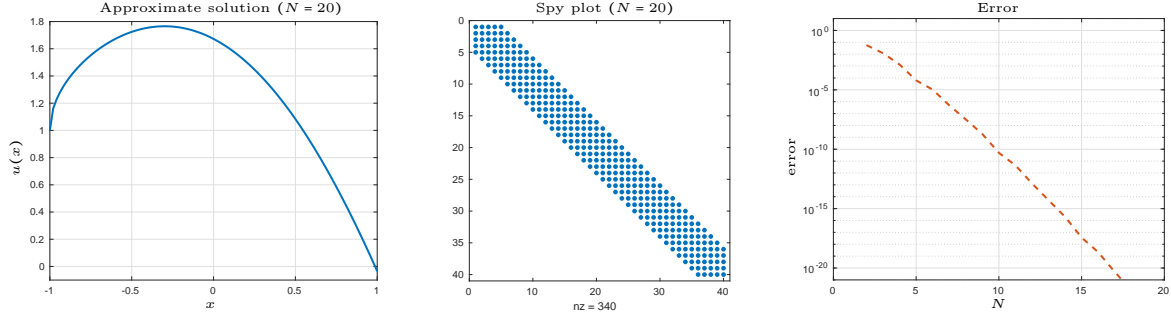
\end{example}

%% file: half_RL.tex
Our approach here for FDEs of Riemann--Liouville-type will be similar to the FIEs above. The main
difference will stem from the fact that the operators $D^{m}$ and $D^{m+1/2}$
defined in Section~\ref{subsec:blockops} no longer map to the same direct sum spaces
and we must make use of the block-banded conversion operators $E_{m}$ and $E_{m+1/2}$
(analogous to how the conversion operators ${\cal S}_\lambda$ are used in~\cite{olver2013}).


\subsection{Differential equation of order $1/2$}\label{subsec:rlhalf}

Consider the FDE
\be\label{eqn:fde05}
u(x) + \prescript{RL}{-1}{D^{1/2}_x}u(x) = e(x) + \frac{1}{\sqrt{1+x}}f(x), \qquad x\in[-1,1], \qquad u(-1) < \infty,
\ee
(sometimes called a ``fractional relaxation equation'') and make the same ansatz as before that
\begin{eqnarray}\label{eqn:trial}
u(x) &=& \sum_{n=0}^\infty a_nP_n(x) + {\sqrt{1+x}}\sum_{n=0}^\infty b_nU_n(x)
 = \left[\P(x), \U_{1/2}(x)\right]\left(\begin{array}{c}\underline{a}\\ \underline{b}\end{array}\right).
\end{eqnarray}
From Section~\ref{subsec:blockops} we  have that
\begin{eqnarray}\label{eqn:Dhalfu}
^{RL}_{-1}\calD^{1/2}_xu(x) &=& \left[\Cl[3/2](x), \U_{-1/2}(x)\right]D_\PoU^{1/2}\left(\begin{array}{c}\underline{a}\vphantom{D_{1/2}^{1/2}}\\ \underline{b}\vphantom{D_{1/2}^{0}}\end{array}\right),
\end{eqnarray}
where $D^{1/2}$ is defined in~(\ref{eqn:blockD}).
As mentioned above, and unlike in the case of the integral equations, the range of $\prescript{RL}{-1}{{\cal{D}}}^{1/2}_xu(x)$ is not the same as that of $u(x)$.
However, we can find a banded transform from $\P\oplus\U_{1/2}$ to $\Cl[3/2](x)\oplus\U_{-1/2}$ using $E_{1/2}$ so that 
\begin{eqnarray}
u(x) &=& \left[\Cl[3/2](x), \U_{-1/2}(x)\right]E_{1/2}\left(\begin{array}{c}\underline{a}\\ \underline{b}\end{array}\right).
\end{eqnarray}
This time letting 
\be
e(x) + \frac{1}{\sqrt{1+x}}f(x) = \left[\Cl[3/2](x), \U_{-1/2}(x)\right]\left(\begin{array}{c}\underline{e}\\\underline{f}\end{array}\right)
\ee
and equating coefficients leads to the linear system of equations
\be\label{eqn:system_diff05}
\left(\begin{array}{c c}
	     E_{1/2} + D_\PoU^{1/2}
     \end{array}\right)
          \left(\begin{array}{c}
	    \underline{a}\vphantom{D_{1/2}^{1/2}} \\\underline{b}\vphantom{D_{1/2}^{1/2}}
     \end{array}\right)
     =
     \left(\begin{array}{c c}
	    S_\P & D_{\U_{1/2}}^{1/2} \vphantom{D_0^{1/2}} \\
	    D_\P^{1/2} & R_\U \vphantom{D_0^{1/2}}
     \end{array}\right)
     \left(\begin{array}{c}
	    \underline{a} \vphantom{D_1^{1/2}} \\\underline{b} \vphantom{D_{1/2}^{1/2}}
     \end{array}\right) = 
     \left(\begin{array}{c}
	    \underline{e} \vphantom{D_1^{1/2}} \\\underline{f} \vphantom{D_{1/2}^{1/2}}
     \end{array}\right).
\ee
Again, each block of the operators in (\ref{eqn:system_diff05}) are banded, and by interleaving the coefficients so that $(\underline{a}^\top, \underline{b}^\top)$ $\mapsto$ $(a_0, b_0, a_1, b_1, \ldots)^\top$ we
can convert the above to a banded system.

{\bf Remark:} One can show that the null space of the operator on the left-hand side of~(\ref{eqn:fde05}) acting on functions in $L_1$ is 
\be
v(x) = \frac{E_{1/2,1/2}(-\sqrt{1+x})}{\sqrt{1+x}},
\ee
(where $E_{1/2,1/2}$ is the Mittag-Leffler function~\cite[10.46.3]{DLMF})~\cite[p.~13]{bajlekova2001}, which is unbounded at $x=-1$. Since 
our trial space~(\ref{eqn:trial}) contains only bounded functions on $[-1,1]$, we need not enforce a boundary condition explicitly in this case. We discuss boundary
constraints in more detail momentarily. To allow solutions which are unbounded at the left end of the domain
then one possibility is to use instead the ansatz $u(x) = \P(x)\underline{a} + \T_{-1/2}(x)\underline{b}$, 
where $\T(x)$ is the quasimatrix whose columns are formed of Chebyshev polynomials of the first kind, $T_n(x), \, n = 0, 1,\ldots$. One can derive 
similar banded operators to all those introduced in Section~\ref{sec:prelim}, but we omit the details (which are complicated by the fact that 
$T_n(x)$ is not an ultraspherical polynomial and therefore many of the formulae in Section~\ref{sec:prelim} differ subtly).

\begin{example}\label{subsec:example5} We consider a modification of the  second-kind Abel integral equation in Example~\ref{subsec:example1}, namely
\be\label{eqn:ex5}
 u(x) +  _{-1}\calD_x^{1/2}u(x) = \frac{1}{\sqrt{\pi}\sqrt{1+x}}, \qquad u(-1) < \infty,
\ee
with solution
\be\label{eqn:exsoln}
 u(x) = e^{1+x}\rm{erfc}(\sqrt{1+x}).
\ee
To solve we choose an $N$ and form the system~(\ref{eqn:system_diff05}) with $\underline{e} = \underline{0}$ and $\underline{f} = [1/\sqrt{\pi}, 0, 0, \ldots]^{\top}$.
Results are shown in Figure~\ref{fig:ex5}. As usual, the first two panels show a plot of the solution and a {\tt spy} plot of the discretised operator when $N=20$.
The {\tt spy} plot verifies that the matrix is banded (here with bandwidth four) and hence that the system~(\ref{eqn:system_diff05}) can be solved in linear time with MATLAB's $\backslash$. 
The final
panel shows both the infinity norm error (approximated on a 100-point equally spaced grid) of computed solution 
compared to the exact solution~(\ref{eqn:exsoln}) (solid line) and the 2-norm difference between the coefficients 
of the approximated solution when truncating at sizes $N$ and $\lceil 1.1N \rceil$ (dashed line). Again we observe geometric convergence.
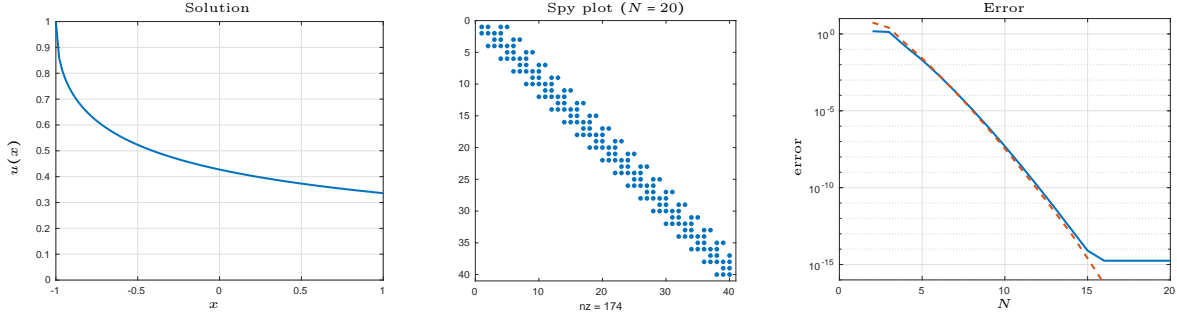
\begin{figure}[t]
\centering\tiny
\includegraphics[height=105pt]{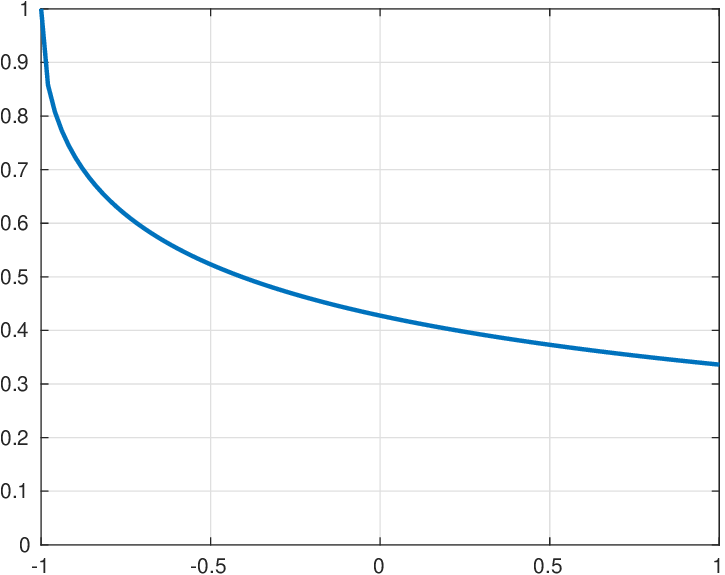}\hspace*{25pt}
\includegraphics[height=105pt,trim={0 15pt 0 1pt}]{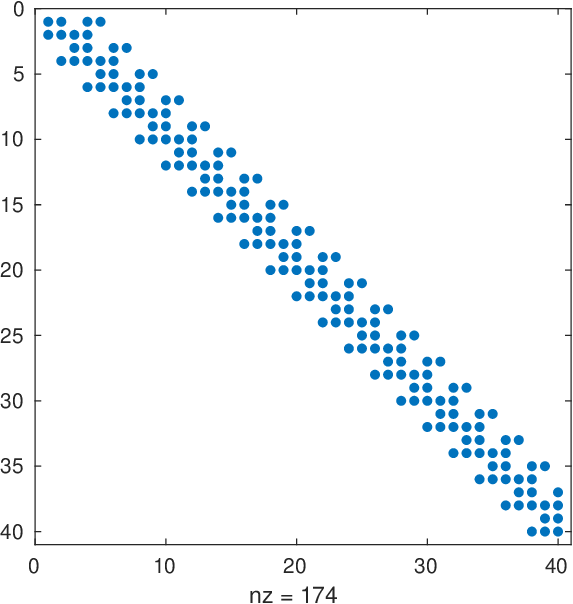}\hspace*{25pt}
\includegraphics[height=105pt]{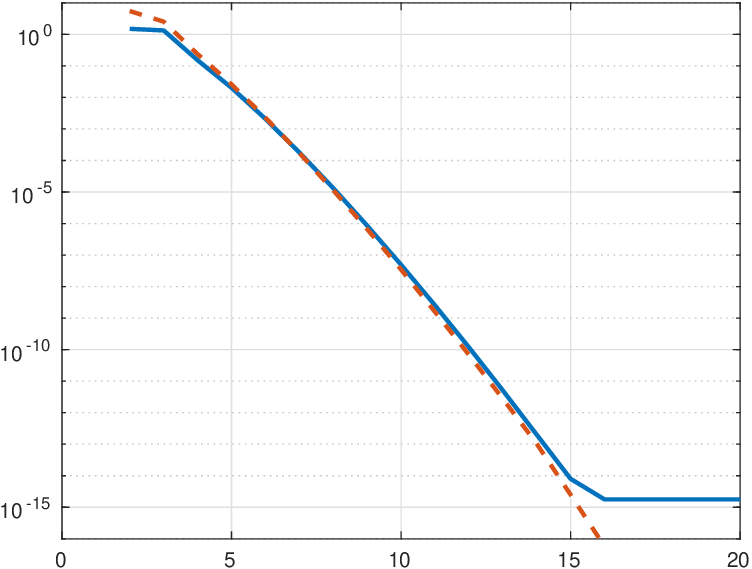}
\put(-73,107){Error}
\put(-238,107){Spy plot $(N = 20)$}
\put(-375,107){Solution}
\put(-144,45){\rotatebox{90}{error}}
\put(-442,45){\rotatebox{90}{$u(x)$}}
\put(-68,-5){$N$}
\put(-366,-5){$x$}
\caption{(a) Solution to~(\ref{eqn:ex5}). (b) MATLAB {\tt spy} plot of~(\ref{eqn:system_diff05}) showing the banded structure
(c) 
Solid line: Infinity norm error of solution approximated on a 100-point equally spaced grid. 
Dashed line: 2-norm difference between the coefficients of the approximated solution when 
truncating at sizes $N$ and $\lceil 1.1N \rceil$. As in the case of FIEs from the previous section, geometric convergence 
is observed here for this FDE.}\label{fig:ex5}
\end{figure}
\end{example}


\subsection{Differential equation of order $1$ and $1/2$}
Here we consider
\be\label{eqn:FDE1}
u(x) + \prescript{RL}{-1}{D^{1/2}_x}u(x) + u'(x) = e(x) + \frac{1}{\sqrt{1+x}}f(x), \qquad x\in[-1,1],
\ee
along with a suitable initial or boundary condition, or some other functional constraint (see below).
Now
\begin{eqnarray}
u'(x) &=& \left[\Cl[3/2](x), \Cl[2]_{-1/2}(x)\right]D\left(\begin{array}{c}\underline{a}\\ \underline{b}\end{array}\right),
\end{eqnarray}
where $D$ is defined in~(\ref{eqn:blockD}).
We must modify the spaces of both 
$\underline{u}$ and $D^{1/2}\underline{u}$ 
accordingly, 
and so arrive at
\be\label{eqn:system_diff1}
\left(E_{1}E_{1/2} + E_{1}D_\PoU^{1/2} + D_\PoU\right)
     \left(\begin{array}{c}
	    \underline{a}\foo\\\underline{b}
     \end{array}\right) = 
     \left(\begin{array}{c}
	    \underline{e}\foo\\\underline{f}
     \end{array}\right),
\ee
where 
\be
e(x) = \sum_{n=0}^\infty e_nC_n^{(3/2)}(x), \qquad f(x) = \sum_{n=0}^\infty f_nC_n^{(2)}(x).
\ee
Again, by interleaving the coefficients, we can make the above operator banded.


\subsubsection{Boundary conditions}

In this case, the kernel of the operator in~(\ref{eqn:FDE1}) is smooth  
and we must enforce a boundary condition to ensure that the
linear system~(\ref{eqn:system_diff1}) is invertible. The topic of boundary 
conditions in FDEs is complicated, and it is beyond the scope of this paper
to give a full treatment here. Here we simply show how certain boundary
conditions/side constraints can be applied to linear systems such as~(\ref{eqn:system_diff1})
and leave it to the reader to determine how many and what form of constraints
are applicable to their FDE.

For example, consider the functional constraint
$
{\cal B}^x u := u(x) = c.
$
Given scalar $x\in[-1,1]$ we can construct this functional acting on a basis in $\Cal$ as a row vector by defining
$B_{\lambda,\gamma}^x : \Cal \rightarrow \mathbb C$ as $B_{\lambda,\gamma}^x := \Cal(x)$.
In particular, $x = -1$ corresponds to a boundary/initial condition on the left and $x=+1$ to a boundary condition on the right.
Some useful cases are
\be
B_{\P}^{-1} = [1,-1,1,-1,\ldots], \quad B_{\U_{1/2}}^{-1} = [0,0,\ldots], \quad 
B_{\P}^{+1} = [1,1,\ldots], \quad \text{and} \quad B_{\U_{1/2}}^{+1} = \sqrt{2}[1,2,3,4,\ldots].
\ee
Combining such operators to act on our direct sum expansion of the solution $u(x)$, we have, 
for example,  $B_{\PoU}^{-1} : \P \oplus \U_{1/2} \rightarrow \mathbb{C}$ given by 
	$B_{\PoU}^{-1} = [B_{\P}^{-1}, B_{\U_{1/2}}^{-1}]$
and our system~(\ref{eqn:system_diff1}) augmented with the boundary condition $u(-1) = c$ becomes
\be\label{eqn:system_diff_bcs}
\begin{pmatrix}
B_{\PoU}^{-1} \cr
E_{1}E_{1/2} + E_{1}D_\PoU^{1/2} + D_\PoU
\end{pmatrix}
     \left(\begin{array}{c}
	    \underline{a}\\\underline{b}
     \end{array}\right) = 
     \left(\begin{array}{c}
     		c \\     
	    \left[\begin{array}{c}\underline{e}\foo\\\underline{f}\end{array}\right]
     \end{array}\right),
\ee
Upon the usual re-ordering of the coefficients, this becomes an {\it almost-banded} infinite matrix---that is, banded apart from a finite number of dense rows---and 
when truncated to a $(2N+1)\times(2N+1)$ finite matrix is solvable in $O(N)$ operations using either a Schur complement factorisation about the $(1,1)$ entry, the Woodbury matrix identity, 
or by using the adaptive QR method described in~\cite{olver2013}. See Section~\ref{subsec:solvesystem} for more details.

\begin{example}
Consider the case of~(\ref{eqn:FDE1}) where the right-hand side is zero and $u(-1) = 1$:
\be\label{eqn:ex6}
u(x) + \prescript{RL}{-1}{D^{1/2}_x}u(x) + u'(x) = 0, \qquad x\in[-1,1], 
\ee
which amounts to taking $c=1$ and $\underline{e} = \underline{f} = \underline{0}$ in~(\ref{eqn:system_diff_bcs}).
The computed solution is depicted in the left panel of Figure~\ref{fig:ex6}. The middle panel verifies the almost banded nature of the
operator~(\ref{eqn:system_diff_bcs}), and the right panel demonstrates geometric convergence.

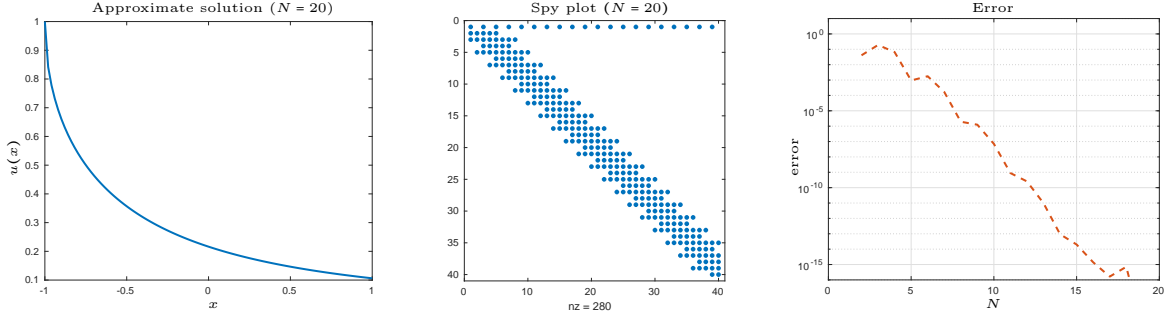
\begin{figure}[t]
\centering\tiny
\includegraphics[height=105pt]{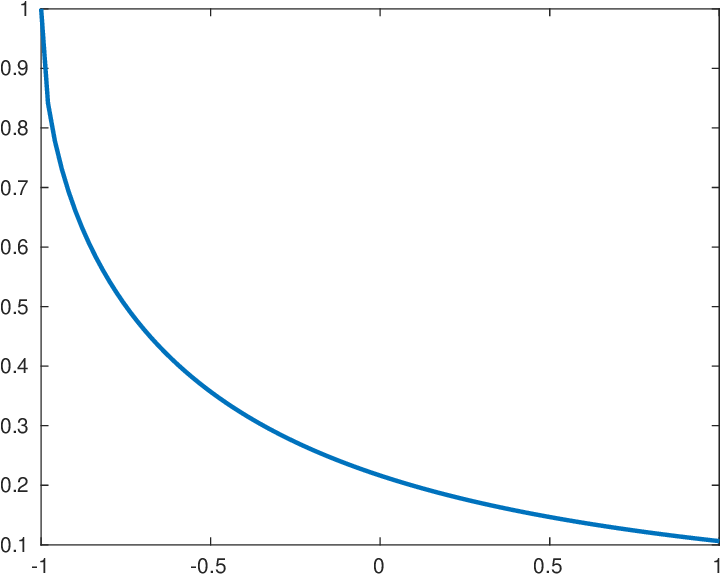}\hspace*{25pt}
\includegraphics[height=105pt,trim={0 15pt 0 1pt}]{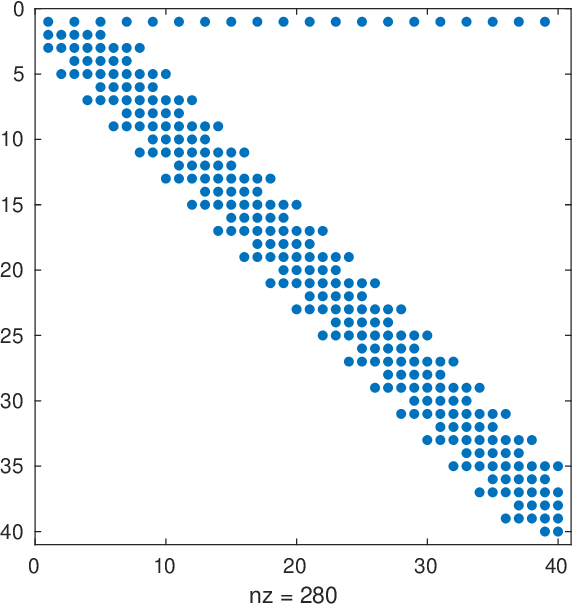}\hspace*{25pt}
\includegraphics[height=105pt]{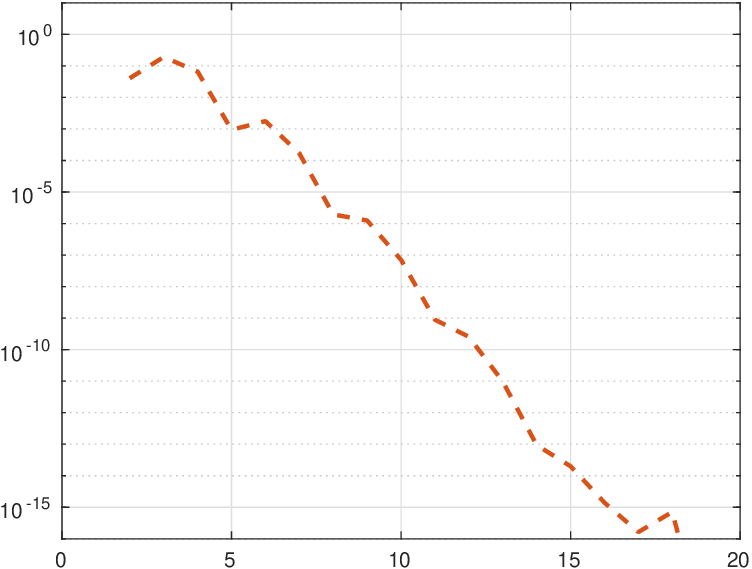}
\put(-73,107){Error}
\put(-240,107){Spy plot $(N = 20)$}
\put(-406,107){Approximate solution ($N = 20$)}
\put(-141,45){\rotatebox{90}{error}}
\put(-438,45){\rotatebox{90}{$u(x)$}}
\put(-68,-5){$N$}
\put(-362,-5){$x$}
\caption{(a) Approximate solution to~(\ref{eqn:ex6}). (b) MATLAB {\tt spy} plot of~(\ref{eqn:system_diff_bcs}) showing the {\em almost}-banded structure. 
(c) 2-norm difference between the coefficients of the approximated solution when 
truncating at sizes $N$ and $\lceil 1.1N \rceil$ showing geometric convergence.}\label{fig:ex6}
\end{figure}
\end{example}


\subsection{Non-constant coefficients}
Non-constant coefficients can be dealt with in a similar way as described for fractional integral equations in Section~\ref{subsec:nonconstant}. We omit the details.




\subsection{Higher order} Similarly to the case of integral equations, the approach outlined above can be extended to higher-order derivatives. 
Consider the general $m${th} half-integer order FDE:
\be
{\cal{L}}u(x) = \alpha^{[0]}(x)u(x) + \sum_{k=1}^{2m}\alpha^{[k]}(x)\,_{-1}\calD^{k/2}_x[\beta^{[k]}u](x),
\ee
where the nonconstant coefficients $\alpha^{[k]}(x)$ and $\beta^{[k]}(x)$ are analytic in some neighbourhood of $[-1,1]$.
If we continue to take as our anzatz solution the function
\be
u(x) = \left[\P(x), \U_{1/2}(x)\right]\left(\begin{array}{c}\underline{a}\\ \underline{b}\end{array}\right)
\ee
then we have $L: \P\oplus\U_{1/2}\rightarrow  \C^{(m+1/2)}\oplus\C^{(m+1)}_{-m+1/2}$ given by
\be
\left(\sum_{k=0}^{2m-1}\left(E_mE_{m-1/2}\ldots E_{(k+1)/2}\right)\Pi_{k/2}[\alpha^{[k]}]D_\PoU^{k/2}\Pi_\PoU[\beta^{[k]}]\right) +
\Pi_{m}[\alpha^{[2m]}]D_\PoU^{m}\Pi_\PoU[\beta^{[2m]}]
\ee
(where we have defined $\beta^{[0]}(x) = 1$ for the sake of brevity).

\begin{example}\label{ex:bt} Consider the classical 
Bagley--Torvik equation\cite{bagley1983, bagley1984}
\begin{eqnarray}\label{eqn:ex7}
u''(x) + \prescript{RL}{-1}{D^{1/2}_x}u(x) + u(x) = 0,&& \qquad x\in[-1,1]
\end{eqnarray}
but here treated as a boundary value problem with
\begin{eqnarray}
u(-1) = 1, \quad  \text{and} \quad u(1) = 0. 
\end{eqnarray}
Following the approach outlined above, we arrive at the infinite dimensional linear system
\be\label{eqn:sexy_system}
\begin{pmatrix}
B_{\PoU}^{-1} \\
B_{\PoU}^{+1} \\
D^2 + E_2E_{3/2}E_{1}(D_\PoU^{1/2} + E_{1/2})
\end{pmatrix}
     \left(\begin{array}{c}
	    \underline{a}\foo\\\underline{b}
     \end{array}\right) = 
     \left(\begin{array}{c}
     		0 \\ 1\\     
	    \left[\begin{array}{c}\underline{e}\foo\\\underline{f}\end{array}\right]
     \end{array}\right),
\ee
which can be solved in the same manner as before. The solution is depicted in Figure~\ref{fig:ex7}.  
\begin{figure}[t]
\centering\tiny
\includegraphics[height=105pt]{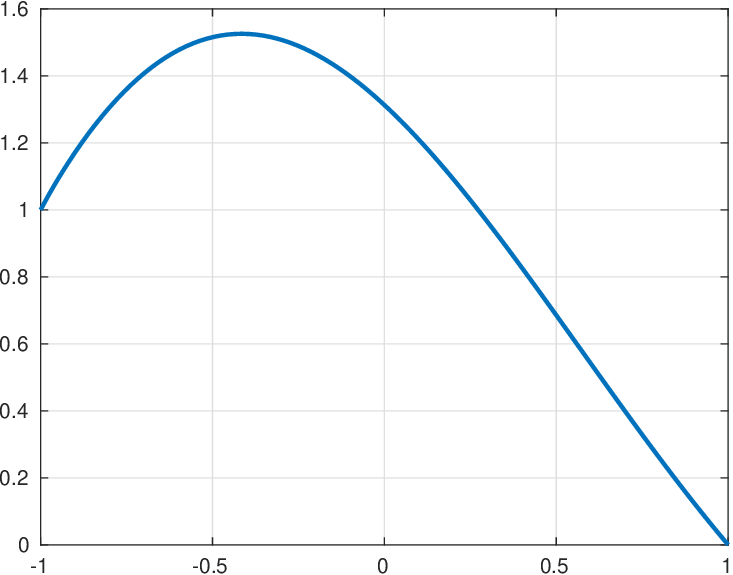}\hspace*{25pt}
\includegraphics[height=105pt,trim={0 15pt 0 1pt}]{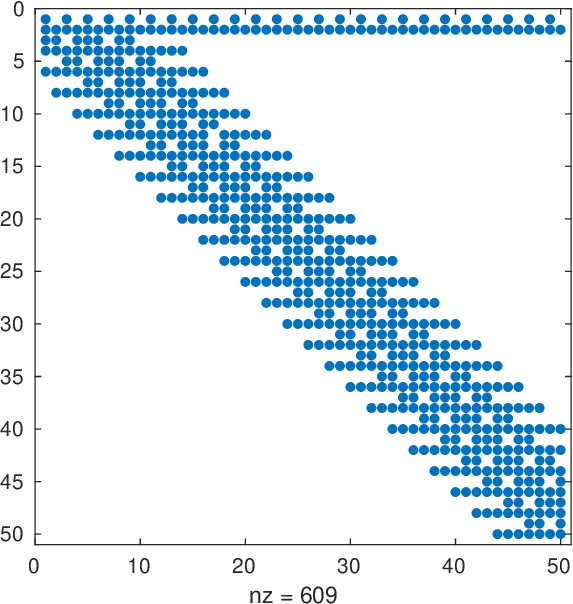}\hspace*{25pt}
\includegraphics[height=105pt]{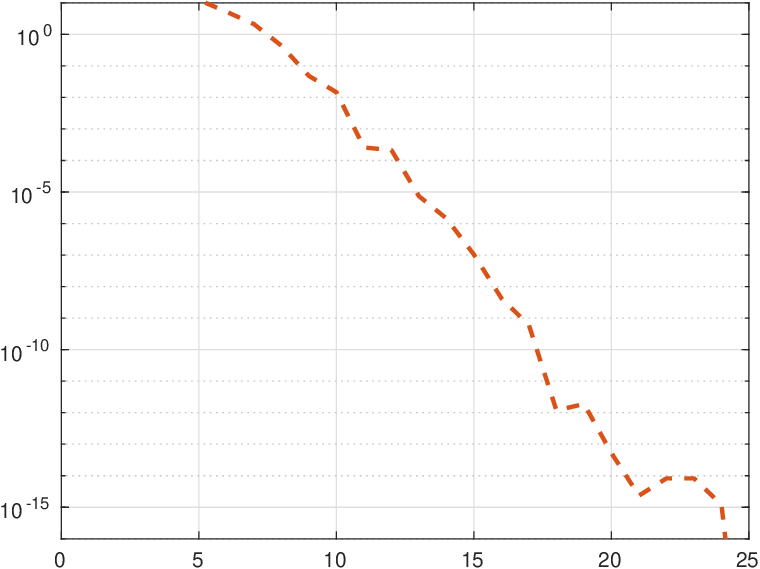}
\put(-73,107){Error}
\put(-239,107){Spy plot ($N = 25$)}
\put(-405,107){Approximate solution ($N = 25$)}
\put(-141,45){\rotatebox{90}{error}}
\put(-439,45){\rotatebox{90}{$u(x)$}}
\put(-68,-5){$N$}
\put(-366,-5){$x$}
\caption{(a) Approximate solution to the Bagley–Torvik equation~(\ref{eqn:ex7}). 
(b) MATLAB {\tt spy} plot of~(\ref{eqn:sexy_system}) demonstrating the almost-banded structure of the linear system. 
(c) 2-norm difference between the coefficients of the approximated solution when 
truncating at sizes $N$ and $\lceil 1.1N \rceil$ showing geometric convergence.\label{fig:ex7}}
\end{figure}
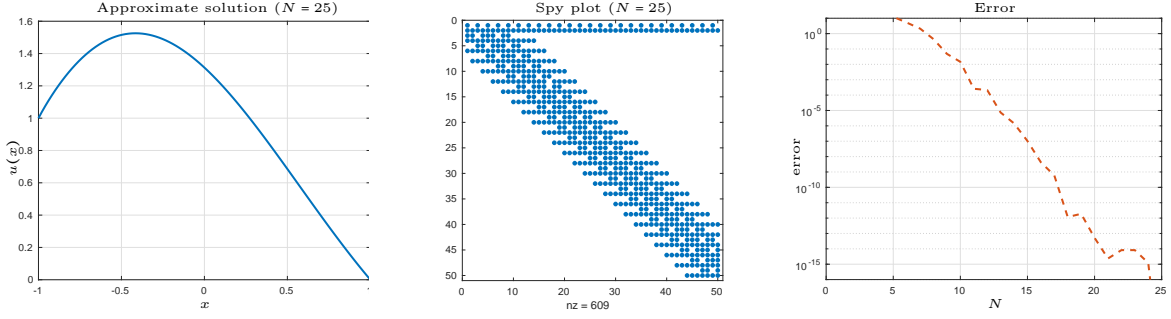
\end{example}\\
{\bf Remark:} If we instead consider the FDE: $u''(x) + \prescript{RL}{-1}{D}{}^{3/2}_xu(x) + u(x) = 0$, then we
simply change the final block-row of the system~(\ref{eqn:sexy_system}) to  $D^2 + E_2(D^{3/2} + E_{3/2}E_{1}E_{1/2})$.
Similarly, we could incorporate a Neumann or fractional Neumann boundary condition at, say, the right boundary 
by changing the $B^{+1}$ row to the appropriate functional row.

%

%% file: half_Cap.tex
FDEs with the Caputo definition of the fractional derivative can be readily solved by combining our approach for FIEs described in Section~\ref{sec:half}
with an integral reformulation of the problem. 
In particular, setting $v(x) = u^{(\lceil m \rceil)}(x)$ and therefore $u(x) = Q^{\lceil m \rceil}v(x) + p(x)$, 
$p(x)\in\mathbb{P}^{\lceil m\rceil-1}$,
it follows from the definition of the Caputo derivative that an $m$th-order FDE in $u(x)$ becomes an $m$th-order
FIE in $v(x)$ with $\lceil m\rceil$ additional boundary constraints to determine the coefficients of the polynomial 
$p(x) = c_0 + c_1P_1(x) + \ldots c_{\lceil m \rceil-1}P^{\lceil m \rceil-1}(x)$. We proceed by example.

\begin{example}\label{subsec:example_cap}
Consider the Caputo fractional relaxation equation
\be\label{eqn:ex9}
 u(x) +  \prescript{C}{-1}\calD_x^{1/2}u(x) = 0, \qquad u(-1) = 1,
\ee
which has the solution
\be
u(x) = e^{1+x}\text{erfc}(\sqrt{1+x}).
\ee
Letting $v = u'$ we have $u = \calQ v + c_0$ and~(\ref{eqn:ex9}) becomes
\begin{eqnarray}
\label{eqn:ex9b}
 \calQ v(x) +  _{-1\!\!}\calQ_x^{1/2}v(x) + c_0 &=& 0,\\
 \calQ v(-1) + c_0 &=& 1.
\end{eqnarray}
In operator form, we may write this as the infinite dimensional system
\be\label{eqn:system_cap}
\left(\begin{array}{c c}
	    1 & B^{-1}Q\\
	    \left[\begin{array}{c}\underline{e}_1\\\underline{0}\end{array}\right] & Q + Q^{1/2} \vphantom{\myab}
     \end{array}\right)
          \left(\begin{array}{c}
	   c_0 \\ \myab
     \end{array}\right)
     =
     \left(\begin{array}{c}
	    1\\\underline{0}\vphantom{\myab} \vphantom{D_1^{1/2}}
     \end{array}\right),
\ee
where 
\be
  v(x) = \sum_{n=0}^\infty \hat a_nP_n(x)  + \sqrt{1+x}\hat b_nU_n(x) = [\P(x), \U_{1/2}(x)]\left(\begin{array}{c}\underline{\hat a}\\\underline{\hat b}\end{array}\right)
\ee
and $B^{-1} = [B_{\P}^{-1}, B^{-1}_{\U_{1/2}}]$. After truncating and 
solving this system for the approximate coefficients of $v(x)$, we can recover those of $u(x)$ via
\be\label{eqn:ufromv}
\left(\begin{array}{c}\underline{a} \\ \underline{b}\end{array}\right)
     =
    Q\left(\begin{array}{c}\underline{\hat a} \\ \underline{\hat b}\end{array}\right) 
    + 
    \left(\begin{array}{c}c_0\\\underline{0}\end{array}\right),
\ee
so that, as usual, $u(x) = [\P(x), \U_{1/2}(x)][\underline{a}^\top, \underline{b}^\top]^\top$.
Figure~\ref{fig:ex9} shows the results. As in the case of RL FDEs we see that the resulting discretised system 
is almost banded and that the approximation converges geometrically in the number of degrees of freedom.
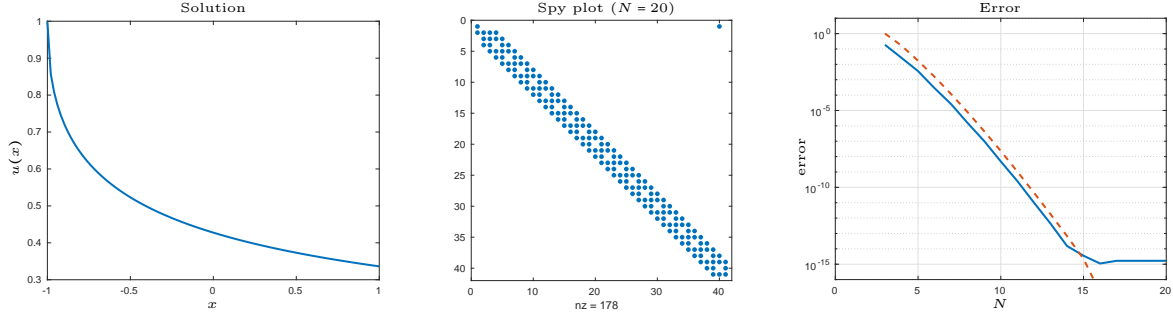
\begin{figure}[t]
\centering\tiny
\includegraphics[height=105pt]{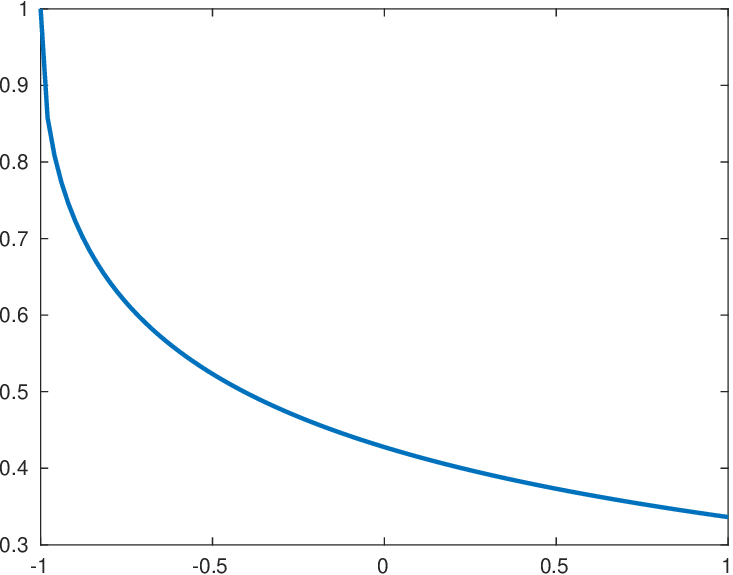}\hspace*{25pt}
\includegraphics[height=105pt,trim={0 15pt 0 1pt}]{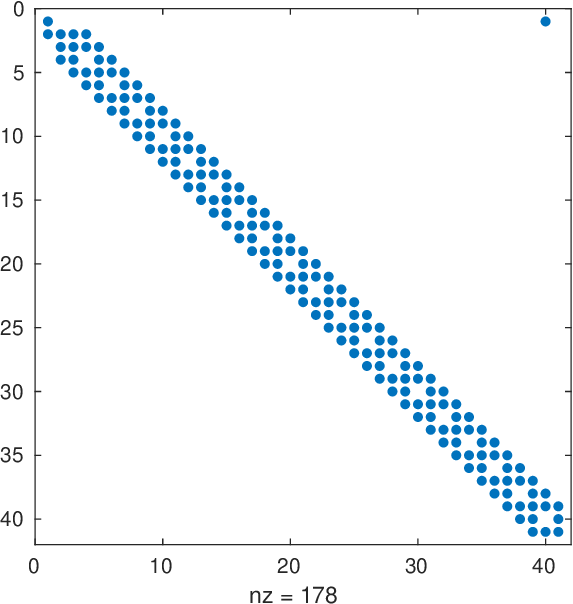}\hspace*{25pt}
\includegraphics[height=105pt]{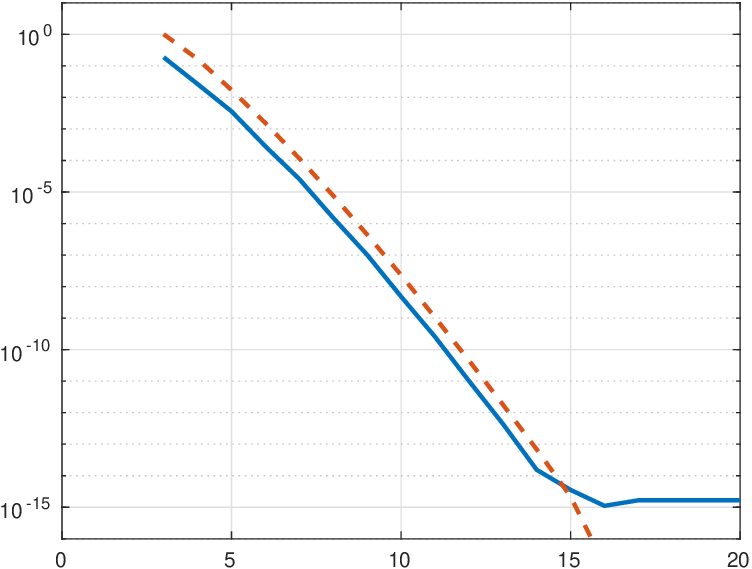}
\put(-73,107){Error}
\put(-240,107){Spy plot ($N = 20$)}
\put(-375,107){Solution}
\put(-141,45){\rotatebox{90}{error}}
\put(-440,45){\rotatebox{90}{$u(x)$}}
\put(-68,-5){$N$}
\put(-365,-5){$x$}
\caption{(a) Solution to the FDE~(\ref{eqn:ex9}). (b) MATLAB {\tt spy} plot of~(\ref{eqn:system_cap}) showing the almost-banded structure. 
(c) Solid line: Infinity norm error of solution approximated on a 100-point equally spaced grid. 
Dashed line: 2-norm difference between the coefficients of the approximated solution when 
truncating at sizes $N$ and $\lceil 1.1N \rceil$. As in the case of FIEs and RL-type FDEs of previous two sections, geometric convergence is observed.}\label{fig:ex9}
\end{figure}
\end{example}

\begin{example}\label{subsec:example_cap_BT}
Consider the Bagley--Torvik equation from Example~\ref{ex:bt}, but now using the Caputo definition of the half-derivative:
\begin{eqnarray}\label{eqn:ex10}
u''(x) + \prescript{C}{-1}{D}^{1/2}_xu(x) + u(x) = 0,&& \qquad x\in[-1,1],\\ 
u(-1) = 1, && \qquad u(1) = 0. 
\end{eqnarray}
This time letting $v = u''$ we have $u = \calQ^2 v + c_0P_0(x) + c_1P_1(x)$ and 
\begin{eqnarray}\label{eqn:ex10b}
 v(x) + _{-1\!\!}\calQ_x^{3/2}v(x) + c_1\calQ_x^{1/2}P_1'(x) + \calQ^2v(x) + c_0P_0(x) + c_1P_1(x) &=& 0,\\
\calQ^2 v(-1) + c_0 + c_1P_1(-1) &=& 1,\\
\calQ^2 v(1) + c_0 + c_1P_1(1) &=& 0.
\end{eqnarray}
We may write this as
\be\label{eqn:system_cap2}
\left(\begin{array}{c c c}
       	    1 & -1 & B^{-1}Q^2 \\
	    1 & 1 & B^{+1}Q^2\\
	    \myvec[\underline{e}_1,\underline 0] & 
	    Q^{1/2}\left[\begin{array}{c}S^{-1}_{1/2}D_{1/2}\underline{e}_1\\\underline{0}\end{array}\right]+\myvec[\underline{e}_1,\underline 0] & Q^{2} + Q^{3/2} + I \\
     \end{array}\right)
          \left(\begin{array}{c}
	   c_0 \\ c_1 \\ \myvecu[\hat a,\hat b]
     \end{array}\right)
     =
     \left(\begin{array}{c}
	    1\\0\\\myvecu[0,0\vphantom{\hat b}]
     \end{array}\right),
\ee
where 
$v(x) = \sum_{n=0}^\infty \hat a_nP_n(x)  + \sqrt{1+x}\hat b_nU_n(x)$ and we have used the fact that $P_1(\pm 1) = \pm1$. Once we have solved
this system for the approximate coefficients of $v$, we can recover those of $u$ via
\be\label{eqn:ufromv2}
\left(\begin{array}{c}\underline{a} \\ \underline{b}\end{array}\right)
     =
    Q\left(\begin{array}{c}\underline{\hat a} \\ \underline{\hat b}\end{array}\right) 
    + 
    \left(\begin{array}{c}c_0\\c_1\\\underline{0}\end{array}\right),
\ee
where here $\underline{0}$ is a vector of zeros of length $2N-2$. 
Figure~\ref{fig:ex10} shows the results. Compare with Riemann--Liouville version in Example~\ref{ex:bt}.
\begin{figure}[t]
\centering\tiny
\includegraphics[height=105pt]{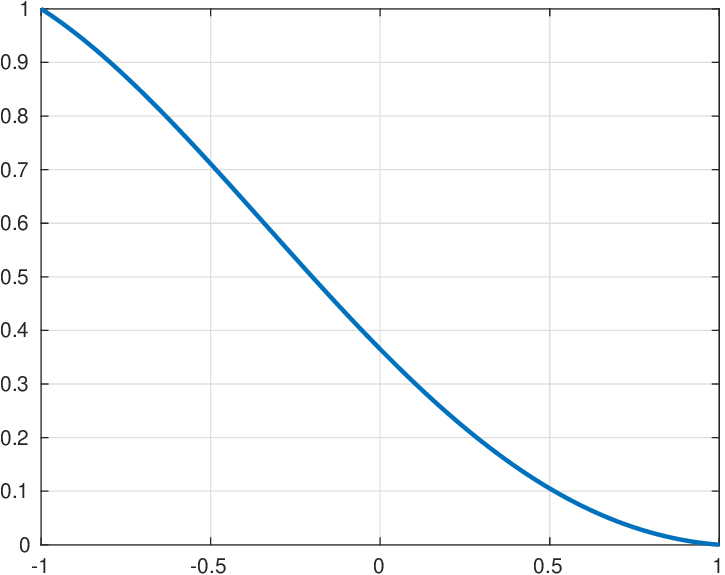}\hspace*{25pt}
\includegraphics[height=105pt,trim={0 15pt 0 1pt}]{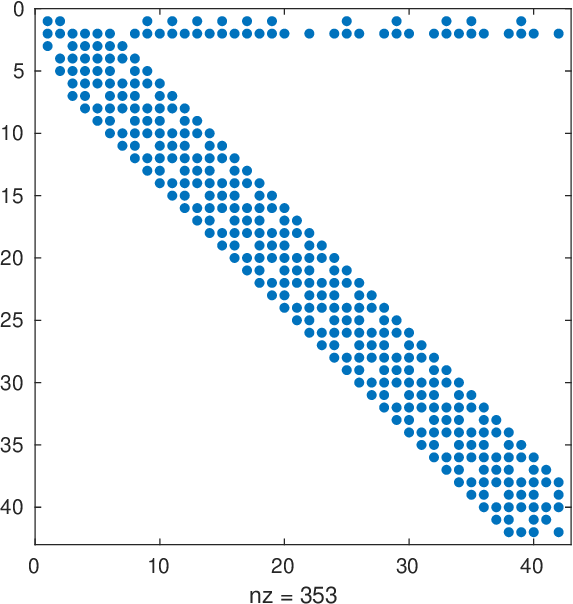}\hspace*{25pt}
\includegraphics[height=105pt]{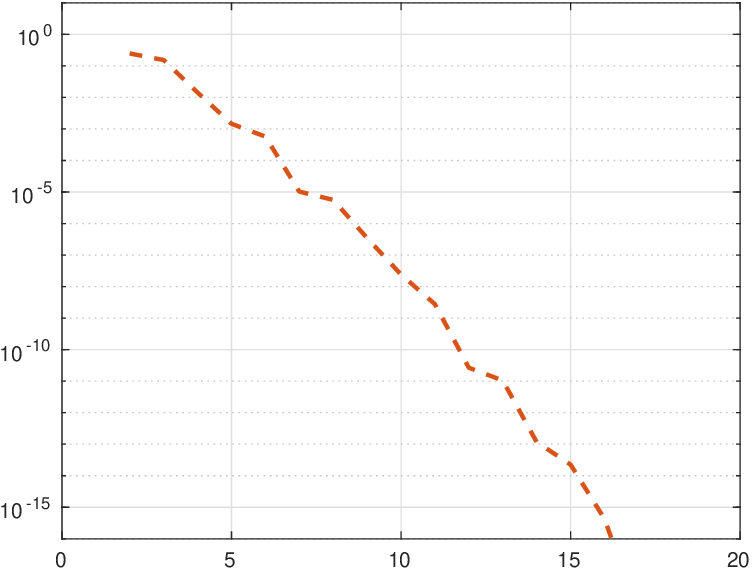}
\put(-73,107){Error}
\put(-240,107){Spy plot ($N = 20$)}
\put(-405,107){Approximate solution ($N = 20$)}
\put(-365,-5){$x$}
\put(-141,45){\rotatebox{90}{error}}
\put(-440,45){\rotatebox{90}{$u(x)$}}
\put(-68,-5){$N$}
\put(-365,-5){$x$}
\caption{(a) Approximate solution to the Bagley--Torvik equation~(\ref{eqn:ex10}). (b) MATLAB {\tt spy} plot of~(\ref{eqn:system_cap2}) showing the almost-banded structure. 
(c) 2-norm difference between the coefficients of the approximated solution when truncating at sizes $N$ and $\lceil 1.1N \rceil$. Compare with the solution on the RL Bagley--Torvik equation in Figure~\ref{fig:ex7}.}\label{fig:ex10}
\end{figure}
\end{example}

%% file: comp.tex
In this section we outline some practical considerations required to perform computations.
\subsection{Representing the right-hand side}\label{subsec:rhs}
An essential part of this approach is representing the right-hand side in the direct sum basis $\P\oplus\U_{1/2}$ 
and its higher-order cousins involving higher order ultraspherical polynomials.  A substantial issue is that given 
a general right-hand side $g(x)$ the decomposition as, for example, 
\be
    g(x) = \sum_{n=0}^\infty e_nP_n(x) + \sqrt{1+x}\sum_{n=0}^\infty f_nU_n(x),
\ee
is not unique: $\P(x)$ and $\U_{1/2}(x)$ form a {\em frame} \cite{christensen2003}.  In the context of this numerical approach, 
uniqueness is not critical as any expansion of this form is suitable provided we can approximate $g(x)$ well by taking finite number of terms.  

We will assume we are given $e(x)$ and $f(x)$ that can be evaluated pointwise\footnote{The case where we may only sample $gx)$ is beyond the scope of this paper, though solving a least squares system with more points than coefficients can perform well in practice.  Another situation that arises in practical settings is where $g(x)$ is specified by a formula such as $\exp(x) \sqrt{1+x} + \cos x  + \exp((1+x)/2)\,{\rm erfc}(\sqrt{1+x})$.  The approach taken by ApproxFun is to overload each operation to automatically determine an appropriate decomposition.}  so that
\be
    g(x) = e(x) + \sqrt{1+x} f(x).
\ee	  
In this case, we can calculate the number of Chebyshev coefficients of $e$ and/or $f$ to within a required tolerance using an adaptive algorithm \cite{aurentz2015, Chebfun}.  
The algorithm is based on the discrete cosine transform (DCT) and hence takes ${\cal{O}}(d\log d)$ operations to compute $d$ coefficients.
Calculating  coefficients in a basis $\C^{(\lambda)}, \lambda\in\mathbb{N}^+$ proceeds in $O(\lambda d)$ operations by applying the conversion operators \eqref{eq:ultraconv}.
Calculating $d$ Legendre coefficients from $d$ Chebyshev coefficients can be accomplished 
in $O(d \log^2 d)$ operations using recently developed fast transforms \cite{hale2014,townsend2016}, and
from these 
coefficients in a basis $\C^{(\lambda+1/2)}, \lambda\in\mathbb{N}^+$ can again be calculated in $O(\lambda d)$ operations by the conversion operators~\eqref{eq:ultraconv}.
The coefficients in non-constant coefficient problems can be computed in an analogous manner.

{\bf Remark:} In typical applications $d$ (the number of coefficients required to represent the right-hand side or non-constant terms)
is much smaller than $N$ (the discretisation size of the system)
and the claim that the proposed method is linear in the degrees of freedom is justified. The exception is when
the linear problem arises from the linearisation of a nonlinear problem. In this case the number
of polynomial terms required to approximate the non-constant coefficients will be the same as the for 
the solution (i.e., $d \approx N$). The linear systems resulting from discretisation are then dense and require ${\cal{O}}(N^3)$
operations to solve via Gaussian elimination. A spectrally accurate algorithm with linear complexity is 
still an open problem even in the case of ODEs.

\subsection{Solving the linear systems}\label{subsec:solvesystem}

We have described an approach to reduce fractional differential and integral equations to banded or almost-banded infinite-dimensional 
linear systems.  A natural approach to approximating the solutions to the resulting equations is the {\em finite section method}: 
truncate the infinite-dimensional systems to $2N\times 2N$ finite-dimensional linear systems.  This is an effective and easy to 
implement approach that achieves $O(N)$ complexity using standard LAPack routines in the banded case, or using the Woodbury formula 
in the almost-banded case.
	
Alternatively, one can solve using the adaptive QR method \cite{olver2013}, which can be thought of as 
performing linear algebra directly on the infinite-dimensional linear system \cite{olver2014}.  In this case, the number of 
coefficients needed to represent the solution within a specified tolerance of the error in residual are determined adaptively 
while preserving the linear complexity.  A benefit of this approach, in addition to the adaptivity, is that it is not prone 
to the discretization introducing ill-posed equations.  Left and right half-integral and half-derivative operators are implemented 
in the ApproxFun.jl package \cite{ApproxFun} for Julia which uses the adaptive QR method. 

\subsection{Evaluating the result}\label{subsec:evalresult}

The outputs of the algorithm we have described in the proceeding sections are coefficients of Legendre and weighted-Chebyshev expansion~(\ref{eqn:ansatz})
of the solution. Typically one is more interested in function values of the solution, but precisely what values are required depends entirely
on the application. If only a few functions values are required, then the simplest approach is to use Clenshaw's algorithm. 
This is the approach we have taken in the results above. If the solution is required at many points, then the fast transforms 
mentioned in Section~\ref{subsec:rhs} can again be utilized to do this efficiently.

%% file: extensions.tex
Here we consider the extension to problems involving rational-order integrals and derivatives.
The general principle is the same as that which we have seen previously for half-integer orders, 
but an immediate consequence of moving to the rational-order case is that ultraspherical discretisations are 
no longer sufficient. A rational-order derivative of an ultraspherical polynomial
does not typically have a short-term expansion in terms of other ultraspherical polynomials,
so instead we must consider weighted {\em Jacobi} polynomials, 
\be
\P_\gamma^{(\alpha, \beta)}(x) : = (1+x)^\gamma [P_0^{\alpha, \beta}(x), P_1^{\alpha, \beta}(x) , \ldots],
\ee
and their associated space of coefficients, $\P_\gamma^{(\alpha, \beta)}$. 
In the case of half-integer order FIEs and FDEs we required a direct sum space formed of two ultraspherical bases (i.e., Chebyshev and Legendre). 
Here, for a rational-order integral or derivative of order $p/q$, we require a direct sum space formed of $q$ such weighted Jacobi polynomials:
\begin{definition}\label{def:weightedjacobi}
 We denote by $\P_{[q]}$ the direct sum space formed of weighted Jacobi bases of the form $\P_{k/q}^{(1-k/q, k/q)}$, for $k = 0, \ldots, q-1$, i.e.,  
 \be
 \P_{[q]}:=\bigoplus_{k=0}^{q-1}\P_{k/q}^{(1-k/q, k/q)} = \P_{0}^{(1, 0)} \oplus \P_{1/q}^{(1-1/q, 1/q)} \oplus \ldots \oplus \P_{1-2/q}^{(2/q, 1-2/q)}(x) \oplus \P_{1-1/q}^{(1/q, 1-1/q)}(x),
 \ee%
 and by $\P_{[q]}(x)$ the quasimatrix\footnote{Observe that each `column' in~(\ref{eqn:jacobiquasi}) is itself a quasimatrix!}
 \be\label{eqn:jacobiquasi}
 \P_{[q]}(x):=[\P_{0}^{(1, 0)}(x), \P_{1/q}^{(1-1/q, 1/q)}(x), \ldots, \P_{1-2/q}^{(2/q, 1-2/q)}(x), \P_{1-1/q}^{(1/q, 1-1/q)}(x)].
 \ee%
 If $\underline{u}^{[k]}\in\P^{(1-k/q,k/q)}_{k/q}$ for $k=0, \ldots, q-1$, then we say $\underline{u}\in\P_{[q]}$ and may write
 \be\label{eqn:ratbasis}
  u(x) = \P_{[q]}(x)\underline{u} = 
\sum_{k=0}^{q-1}(1+x)^{k/q}\sum_{n=0}^\infty u^{[k]}_nP^{(1-k/q,k/q)}_n(x) \quad 
\text{where} \quad \underline{u} = \left(\begin{array}{c}\underline{u}^{[0]}\\ \underline{u}^{[1]} \\ \vdots \\  \underline{u}^{[q-1]}\end{array}\right).
\ee
\end{definition}%
We begin with rational-order integrals of order ${p}/{q}$, where $p, q\in\mathbb{N}^+$. 
For brevity we focus only on constant coefficient problems, but the ideas of Section~\ref{subsec:nonconstant} are readily applicable.

\subsection{Rational-order integral equations}
The foundation of our approach  is the following formula, similar to that of Theorem~\ref{thm:half_int_C}, but here
showing how the fractional integral of weighted Jacobi polynomials may be computed in closed form:
\begin{theorem}{\cite[Theorem 6.72(b)]{andrews1999}}\label{thm:frac_int_J}
For any $0 \le \mu  < 1$, $\alpha, \beta \ge 0$, $-1 < x < 1$, and $n \ge 0$,
 \be\label{eqn:frac_int_J}%
    _{-1}\calQ_x^{\mu}[(1+x)^{\beta}P^{(\alpha, \beta)}_n(x)] = \frac{\Gamma(\beta+n+1)}{\Gamma(\beta+\mu +n+1)}(1+x)^{\beta + \mu}P_n^{(\alpha-\mu, \beta+\mu)}(x).%
 \ee%
\end{theorem}%

\noindent 
We define the infinite-dimensional matrix $Q^{\mu}_\beta: \P^{(\alpha, \beta)}_\beta \rightarrow \P^{(\alpha-\mu, \beta+\mu)}_{\beta+\mu}$ induced by this relationship,
so that if $\underline{u}\in\P^{(\alpha, \beta)}_\beta$ then $_{-1}{\cal{Q}}_x^\mu \P^{(\alpha, \beta)}_\beta(x)\underline{u} = \P^{(\alpha-\mu, \beta+\mu)}_{\beta+\mu}(x)Q^{\mu}_\beta\underline{u}$.
We also consider two conversion operators, $S_{\alpha, \beta}:\P_\gamma^{(\alpha,\beta)}\rightarrow\P_\gamma^{(\alpha+1,\beta)}$ and 
$R_{\alpha,\beta}:\P_{\gamma+1}^{(\alpha,\beta+1)}\rightarrow\P_{\gamma}^{(\alpha,\beta)}$ 
(akin to~(\ref{eq:ultraconv}) and~(\ref{eqn:1plusxC})) induced by~\cite[18.9.5]{DLMF}
\be
(2n+\alpha+\beta+1)P^{(\alpha,\beta)}_n(x)=(n+\alpha+\beta+1)P^{(\alpha+1,\beta)}_n(x)-(n+\beta)P^{(\alpha+1,\beta)}_{n-1}(x),
\ee
and~\cite[18.9.6]{DLMF}
\be
(n+\tfrac{1}{2}\alpha+\tfrac{1}{2}\beta+1)(1+x)P^{(\alpha,\beta+1)}_{n}\left(x%
\right)=(n+1)P^{(\alpha,\beta)}_{n+1}\left(x\right)+(n+\beta+1)P^{(\alpha,%
\beta)}_{n}\left(x\right),
\ee
respectively, so that so that 
if $u(x) = \underline{u}\in\P_\gamma^{(\alpha,\beta)}$ then
$\P_\gamma^{(\alpha,\beta)}(x)\underline{u} = \P_\gamma^{(\alpha+1,\beta)}(x)S_{\alpha, \beta}\underline{u} = \P_{\gamma-1}^{(\alpha,\beta-1)}(x)R_{\alpha,\beta-1}\underline{u}$.
Combining $Q^{\mu}_\beta$,  $S_{\alpha, \beta}$, and $R_{\alpha,\beta}$, we construct a (1/q)th-order integral operator on $\P_{[q]}$ as follows:
\begin{theorem}\label{thm:holdmybeer}
Consider any $q\in\mathbb{N}^+$. If $\underline{u}\in \P_{[q]}$ so that $u(x) = \P_{[q]}(x)\underline{u}$ then the operator 
\be\label{eqn:holdmybeer}
Q_{[q]}^{1/q} := \left(\begin{array}{ccccccccccc}
       & & & & S_{0,0}R_{0,0}Q^{1/q}_{1-\frac{1}{q}}\\
       Q^{1/q}_{0}\\
       & Q^{1/q}_{\frac{1}{q}}\\
       & & \ddots\\
       & & & Q^{1/q}_{1-\frac{2}{q}}\\
      \end{array}\right)
\ee
satisfies 
\be\label{eqn:holdmybeer2}
_{-1}Q^{1/q}_x\P_{[q]}(x)\underline{u} = \P_{[q]}(x)Q_{[q]}^{1/q}\underline{u}.%
\ee%
\end{theorem}%
\begin{proof}
We have from Theorem~\ref{thm:frac_int_J} that for $k = 0\ldots q-2$,
 \be
 \begin{array}{llll}{_{-1}\cal{Q}}_x^{1/q}\P_{k/q}^{(1-k/q,k/q)}(x)\underline{u}^{[k]} &=& \P_{(k+1)/q}^{(1-(k+1)/q,(k+1)/q)}(x)Q_{k/q}^{1/q}\underline{u}^{[k]} \\
 &=& \P_{j/q}^{(1-j/q,j/q)}(x)Q_{k/q}^{1/q}\underline{u}^{[k]}, & j = k+1,
 \end{array}
 \ee
 and for the final block, from the definitions of $R_{0,0}$ and $S_{0,0}$, that 
 \be
 _{-1}{\cal{Q}}_x^{1/q}\P_{1-1/q}^{(1/q,1-1/q)}(x)\underline{u}^{[q-1]} = \P_{1}^{(0,1)}(x)Q_{1-k/q}^{1/q}\underline{u}^{[q-1]} = \P_{0}^{(1,0)}(x)S_{0,0}R_{0,0}Q_{1-k/q}^{1/q}\underline{u}^{[q-1]}.
 \ee%
\end{proof}

\begin{corollary}\label{cor:holdmybeer}
For any $p, q\in \mathbb{N}^+$ the operator
\be
Q_{[q]}^{p/q} := \left[Q_{[q]}^{1/q}\right]^p.%
\ee
is block banded and
satisfies
\be\label{eqn:Qpq}
_{-1}{\cal{Q}}^{p/q}_{[q]}\P_{[q]}(x)\underline{u} = \P_{[q]}(x)Q_{[q]}^{p/q}\underline{u}.
\ee
\end{corollary}%
\begin{proof}
Eqn.\~(\ref{eqn:Qpq}) follows from $p$ applications of $Q_{[q]}^{1/q}$ on $\P_{[q]}$. 
That $Q_{[q]}^{p/q}$ is block banded follows from the fact that each of the blocks is formed by a product of 
banded matrices.
\end{proof}

{\bf Remark:} It is possible to construct an equivalent representation of the operator $Q_{[q]}^{p/q}$ directly 
(rather than by repeated applications/multiplication of $Q_{[q]}^{1/q}$) by using a 
block matrix similar to that of~(\ref{eqn:holdmybeer}), but containing entries of the form $Q^{p/q}_{k/q}$ and other 
suitable $R-$ and $S$-type conversion matrices.
However, whilst this may have some performance benefits, for clarity of exposition and convenience implementation 
we give preference to the construction as given in Corollary~\ref{cor:holdmybeer}. 

To solve an integral equation with terms of the form $_{-1}{\cal{Q}}^{p/q}_xu(x)$, 
one then makes an ansatz that the solution $u(x)$ may therefore be written as in~(\ref{eqn:ratbasis}), i.e., 
$u(x)=\P_{[q]}(x)\underline{u}$ where $\underline{u}\in\P_{[q]}$,
and the required rational-order integral operators can be constructed as in described~(\ref{eqn:holdmybeer}) and~(\ref{eqn:Qpq}) above.
For problems with variable coefficients, block-multiplication operators can be constructed
in a similar manner to those in Section~\ref{subsec:mult}. The resulting infinite dimensional $q\times q$ 
block operator has banded blocks, but by interlacing the coefficients, i.e.,
\be
 [u^{[0]}_0, u^{[1]}_0, u^{[2]}_0, \ldots,  u^{[q-1]}_0, u^{[0]}_1, u^{[1]}_1, u^{[2]}_1, \ldots u^{[q-1]}_1, u^{[0]}_2, \ldots ],
 \ee
the operator becomes banded with bandwidth ${\cal{O}}(q)$. If each of the infinite sums 
in~(\ref{eqn:ratbasis}) are truncated at $N$ terms, then the resulting linear system
can be solved in ${\cal{O}}(qN)$ operations.

\begin{example}\label{subsec:example10}
We demonstrate our method on the generalised second-kind Abel integral equation:
\be\label{eqn:ex10c}
u(x) +\,
_{-1}{\cal{Q}}_x^{p/q}u(x) = 1.
\ee
Unfortunately, except for the special case of $p/q=1/2$ considered in Example~\ref{subsec:example1}, 
there is no closed form solution for~(\ref{eqn:ex10c}) in general. However, if $0 < p/q < 1$, there is a 
convergent series solution~\cite[2.1--7]{Polyanin2008}
\be\label{eqn:ex10series}
u(x) = 1 + \sum_{\ell=1}^\infty (-1)^\ell\frac{(1+x)^{(lp/q)}}{\Gamma(lp/q+1)}, \qquad x\in[-1,1].
\ee
In particular, we take $p = 2$ and $q = 3$, so that our basis consists of weighted Jacobi polynomials of the form
$
P^{(1,0)}_n(x)$, $(1+x)^{1/3}P^{(2/3,1/3)}_n(x)$, and $(1+x)^{2/3}P_n^{(1/3,2/3)}(x), 
$
and the infinite-dimensional linear system we must solve is 
\be
(I + Q_{[3]}^{2/3})
\left(\begin{array}{c}\underline{u}^{[0]}\\\underline{u}^{[1]}\\\underline{u}^{[2]}\end{array}\right) =
\left(\begin{array}{c} \underline{e}_0\\\underline{0}\\ \underline{0}\end{array}\right)
\ee
where $\underline{e}_0 = (1,0,0,\ldots)^\top$ and $\underline{0} = (0,0,0,\ldots)^\top$.  
Since $I$ is diagonal and $Q_{[3]}^{2/3}$ has banded  blocks, re-ordering the coefficients as described above results in a
banded linear system, as shown in the middle panel of Figure~\ref{fig:ex10b}, which can be solved as described in Section~\ref{subsec:solvesystem}. 
The resulting Jacobi polynomial coefficients of the approximate solution evaluated using Clenshaw's scheme (or a more efficient approach, such as~\cite{townsend2016}), 
to obtain the solution in the left panel of Figure~\ref{fig:ex10b}. The final panel of Figure~\ref{fig:ex10b} shows the error in the obtained 
solution as $n$ is increased, and also the magnitude of the coefficients in the solution for the case $N=20$.

\begin{figure}[t]
\centering\tiny
\includegraphics[height=105pt]{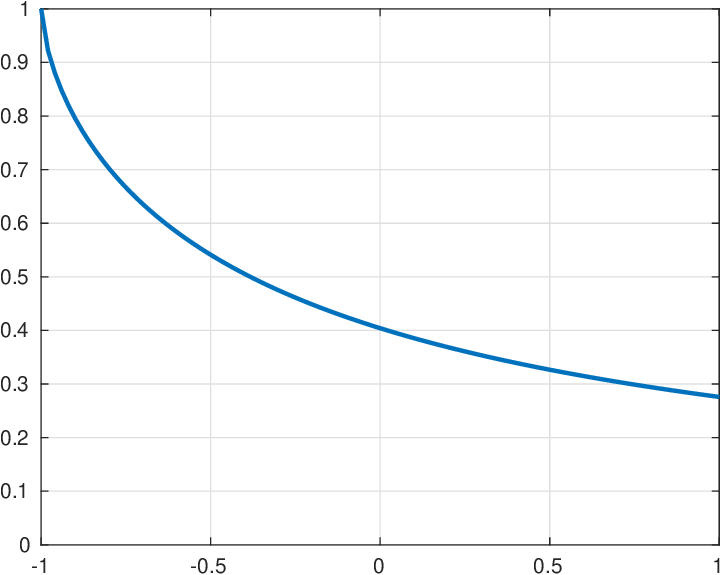}\hspace*{25pt}
\includegraphics[height=105pt,trim={0 15pt 0 1pt}]{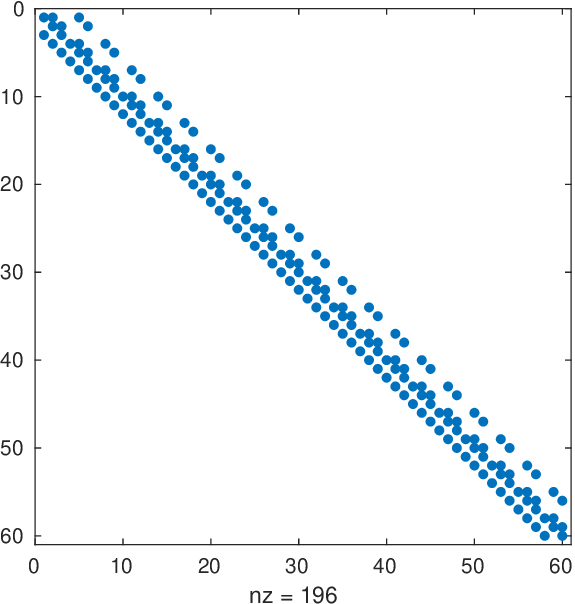}\hspace*{25pt}
\includegraphics[height=105pt]{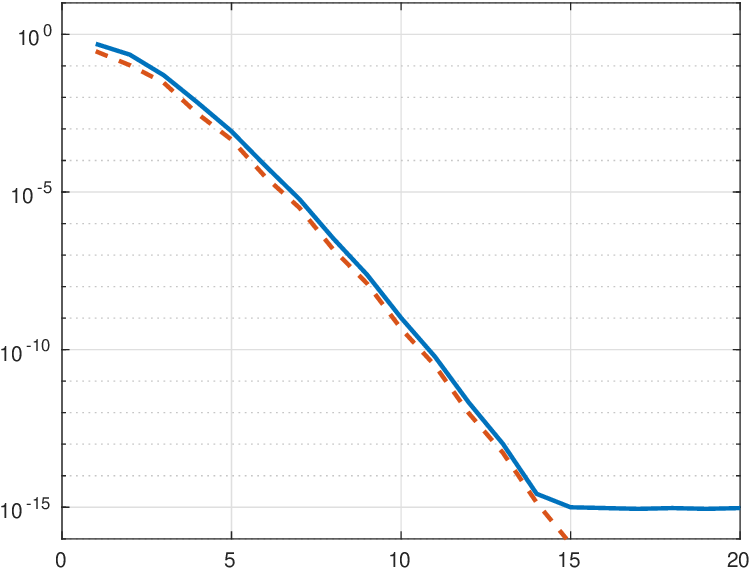}
\put(-116,107){Error and magnitude of coefficients}
\put(-229,107){Spy plot $(N = 20)$}
\put(-402,107){Approximate solution $(N = 20)$}
\put(-143,45){\rotatebox{90}{error}}
\put(-438,45){\rotatebox{90}{$u(x)$}}
\put(-68,-5){$N$}
\put(-364,-5){$x$}
\caption{(a) Approximate solution to~(\ref{eqn:ex10c}) when $p/q = 2/3$. 
(b) MATLAB {\tt spy} plot of the re-ordered linear system. As in the case of half-integer order FIEs, a banded operator is obtained.
(c)  (solid) Infinity norm error of solution (approximated on a 100-point equally spaced grid)
as compared to the series solution~(\ref{eqn:ex10series}). (dashed) Magnitude of the coefficients
in the weighted Jacobi polynomial expansion of the solution, $\underline{u}^{[0]}$, $\underline{u}^{[1]}$, $\underline{u}^{[2]}$.
As before, geometric convergence is observed.
}\label{fig:ex10b}
\end{figure}
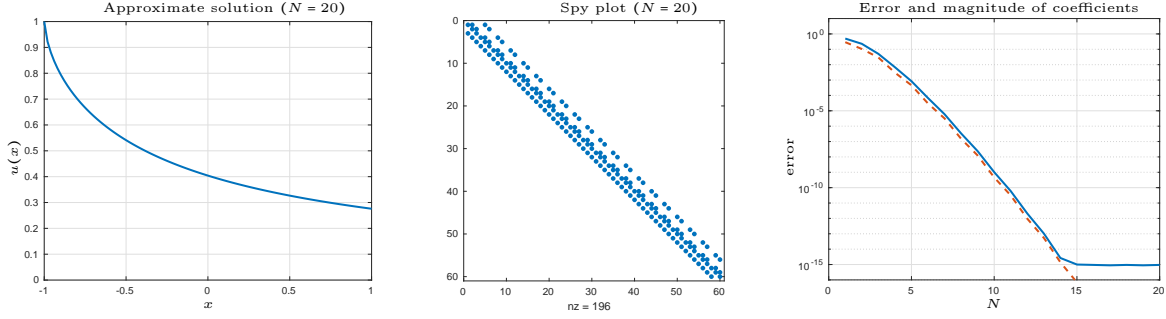

\end{example}

\subsection{Rational order fractional differential equations}
\subsubsection{Caputo-type derivatives}
Similar to the way described in Section~\ref{sec:half_Cap} for half-integer order FDES, 
Caputo FDEs of rational order can be reformulated as rational-order integral equations,
which can be solved as described in the previous section. We omit the details.

\subsubsection{Riemann--Liouville-type derivatives}

Here we may make use of the following:
\begin{theorem}\label{thm:frac_diff_J}
For any $0 \le \mu  < 1$, $\alpha, \beta \ge 0$, and $n \ge 0$%
 \be\label{eqn:frac_diff_J}%
    _{-1}^{RL}\calD_x^{\mu}[(1+x)^{\beta}P^{(\alpha, \beta)}_n(x)] = \frac{\Gamma(\beta+n+1)}{\Gamma(\beta-\mu +n+1)}(1+x)^{\beta - \mu}P_n^{(\alpha+\mu, \beta-\mu)}(x).%
 \ee%
\end{theorem}%
\begin{proof}%
Follows from the fundamental theorem of calculus 
applied to~(\ref{eqn:frac_int_J}).
\end{proof}\\
Similarly to before, we denote by $D_\beta^\mu$ the infinite dimensional operator 
induced by this relationship so that if $\underline{u}\in P^{(\alpha,\beta)}_\beta$ then 
\be\label{eqn:Dmu_rat}
_{-1}^{RL}{\cal{D}}^\mu_x P^{(\alpha,\beta)}_\beta(x)\underline{u} = P^{(\alpha+\mu,\beta-\mu)}_{\beta-\mu}(x)D_\beta^\mu\underline{u}.
\ee
The difficulty here, as in the half-integer case of Section~\ref{sec:half_RL}, 
is that one cannot construct a banded block operator from such operators which maps $\P_{[q]}$
to itself. One must use conversion matrices similar to $E_m$ and $E_{m+1/2}$ as 
described in Section~\ref{sec:half_RL}. An additional problem is that when $p \ge q$ 
then~(\ref{eqn:Dmu_rat}) naively applied to $\P_{[q]}$ will result in Jacobi
polynomials with negative integer parameters, which are not classically defined\footnote{Li and Xu have recently constructed a definition of negative parameter Jacobi polynomials in terms of orthogonality with respect to a Sobolev inner product, avoiding many of the pitfalls that arise from analytically continuing the classical Jacobi polynomials to negative integers \cite{LiXuSobolevBall}. Using these polynomials may allow for reliable generalization of the results to negative parameters.}. 
These difficulties are not insurmountable, and one can extend the approach we consider in 
this paper to such problems, however, in the interest of brevity, we omit the details 
for a later publication.

%% file: conc.tex
By writing the solution in an appropriately constructed basis (in particular a direct sum of Legendre, $P_n(x)$, 
and weighted Chebyshev polynomials of the first kind, $\sqrt{1+x}U_n(x)$) we have successfully solved 
a broad class of half-integer order fractional integral and differential equations with spectral accuracy in linear complexity.  
Some analysis of the half-integral equation 
described in Section~\ref{subsec:halfinteqns} can be found in Section~\ref{subsec:convergence}.
We also described how the approach can be extended to arbitrary rational-order 
FIEs and FDEs by using appropriate weighted Jacobi polynomial bases. For the rational-order
case we demonstrated that the linear complexity  and geometric convergence were maintained, 
but the implied constant in the former is proportional to the denominator, $q$, in the rational degree of the problem.

The main objective of this paper was 
to introduce the algorithm and demonstrate its applicability,
and several examples of both constant and non-constant coefficient linear problems were presented. 
There are several opportunities for future extensions.  Nonlinear problems (through linearisation and Newton's method),
time-dependent problems (through method of lines), and partial fractional differential equations (FDPEs) on rectangular domains
(using ideas related to \cite{townsend2015}) should be relatively straightforward, and we hope to solve problems of this type in 
a future publication. The questions of stability raised in Section~\ref{sec:appendixC} also require further investigation.
An open problem is adapting the approach to problems involving the two-sided fractional derivative, 
used to define the fractional Laplacian.  The known formula for the fractional (or even half-) integral
of Jacobi polynomials does not allow for weighting at both the left and right end of the domain simultaneously, 
which would be required to capture the singular behaviour of two-sided derivatives. 

\begin{example} We close with one final example which demonstrates both the high accuracy and linear complexity of the 
approach described in this paper when applied to a more challenging problem than those shown in the previous few sections. 
In particular, let's consider fractional Airy equations of the form 
\begin{eqnarray}\label{eqn:fracairy}
\varepsilon i^{3/2} \prescript{RL}{-1}{D^{3/2}_x}u(x) - xu(x) = 0,&& \qquad x\in[-1,1], \qquad u(-1) = 0, \ u(1) = 1,
\end{eqnarray}
with $\varepsilon > 0$. Although complex-valued, this non-constant coefficient FDE is of the form discussed in Section~\ref{sec:half_RL}
and we may solve accordingly using the algorithm described.
\begin{figure}[t]
\centering\tiny
\includegraphics[height=105pt]{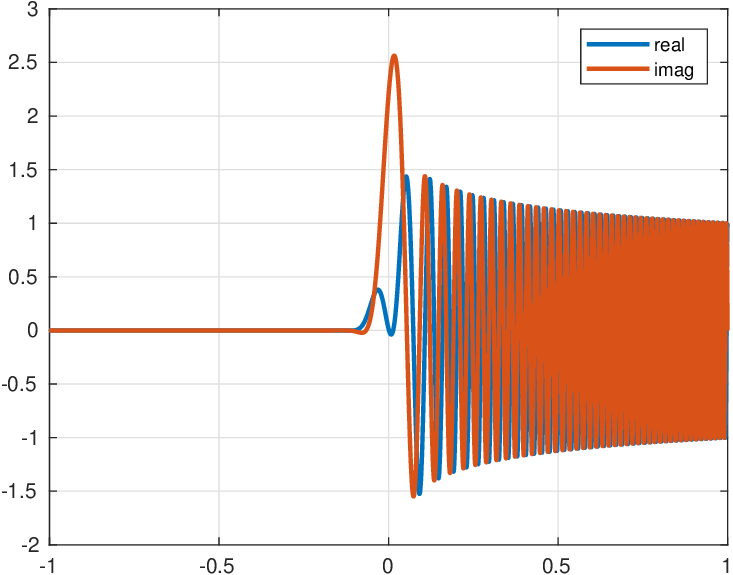}\hspace*{25pt}
\includegraphics[height=105pt]{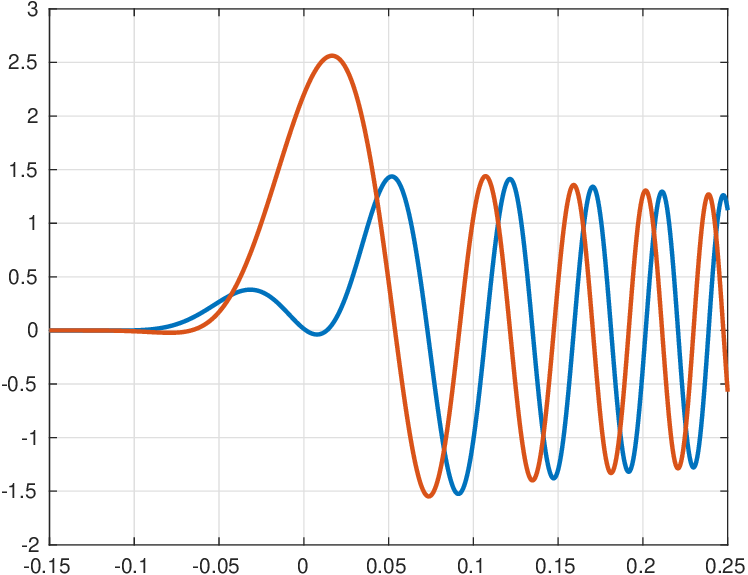}\hspace*{25pt}
\includegraphics[height=105pt]{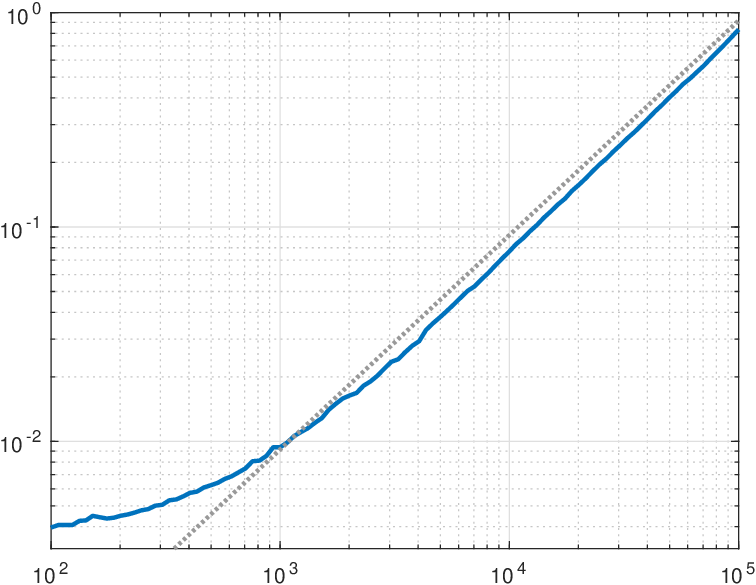}
\put(-76,107){Timing}
\put(-265,107){Close-up near the origin}
\put(-423,107){Approximate solution}
\put(-143,37){\rotatebox{90}{time (secs)}}
\put(-305,45){\rotatebox{90}{$u(x)$}}
\put(-468,45){\rotatebox{90}{$u(x)$}}
\put(-68,-5){$N$}
\put(-230,-5){$x$}
\put(-393,-5){$x$}
\caption{(a) \& (b) Approximate solution to the fractional Airy equation~(\ref{eqn:fracairy}) with $\varepsilon=10^{-4}$. 
(c) Timings for building and solving the linear system for increasing degrees of freedom. Note that these scale linearly
with the truncation length $N$.
\label{fig:fracairy}}
\end{figure}
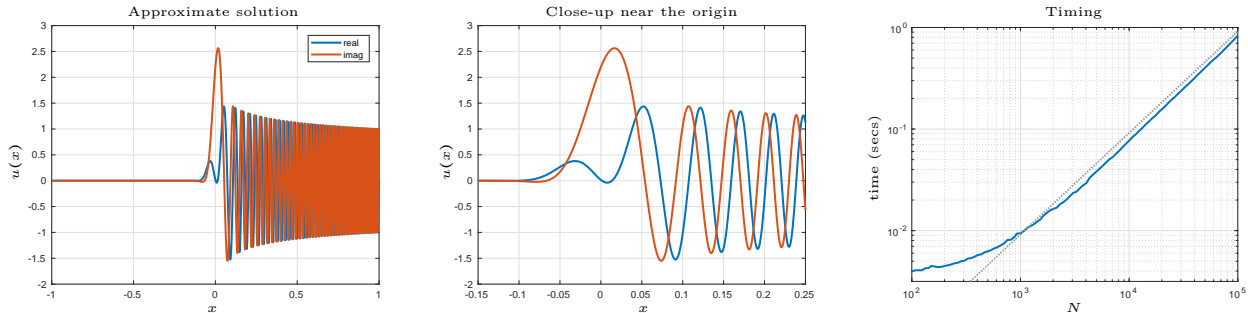
The first and second panels of Figure~\ref{fig:fracairy} show the real and imaginary parts of the solution for $\varepsilon=10^{-4}$, and we see behaviour qualitatively similar to that of the well-known classical Airy equation. 
Experimentally we find that an accuracy of $10^{-10}$ requires around 750 degrees of freedom (i.e., $N\approx375$), and forming and solving the almost-banded linear system representing the fractional 
differential operator and boundary conditions takes under a tenth of a second on a 2014 Desktop PC using the MATLAB implementation\cite{fracspeccode}.
The third panel shows the computational times to form and solve the systems when the number degrees of freedom is artificially increased
(as would be required for smaller values of $\varepsilon$). 
Using the Woodbury formula to solve the almost-banded linear system, we see that linear complexity is obtained.
Finally, we note that this Riemann–Liouville FDE can be readily solved using ApproxFun~\cite{ApproxFun} with just a few commands:
\begin{table}[h!]
\begin{mdframed}
{\footnotesize
{\tt using ApproxFun}

{\tt S = Legendre() $\oplus$ JacobiWeight(0.5, 0, Ultraspherical(1))}

{\tt D1\_5 = LeftDerivative(S, 1.5)}

{\tt x = Fun()}\vspace*{-2pt}

{\tt u = [Dirichlet() ; 0.0001*im\textasciicircum1.5*D1\_5 - x] \textbackslash\ [[0, 1], 0]}}%
\end{mdframed}\vspace*{1pt}
\caption{ApproxFun code for solving the fractional Airy equation~(\ref{eqn:fracairy}) with $\varepsilon = 10^{-4}$.}%
\end{table}\vspace*{-15pt}
\end{example}%

%% file: appendixA.tex
The following results are required in the proofs of Corollaries~\ref{cor:half_int} and~\ref{cor:half_diff_P} in Sections~\ref{subsec:halfint} and~\ref{subsec:halfdiff}, respectively.
\begin{lemma}\label{thm:appendix}
For any $n > 0$ and $\lambda > 0$, the ultraspherical polynomials $C^{(\lambda)}_n(x)$ satisfy the relationship:
 \begin{equation}\label{eqn:appendix}
 2\lambda(1+x)\big(C_{n}^{(\lambda+1)}(x)-C_{n-1}^{(\lambda+1)}(x)\big) = \big((n+1)C_{n+1}^{(\lambda)}(x)+(n+2\lambda)C_{n}^{(\lambda)}(x)\big).
 \end{equation}
\end{lemma}%
\begin{proof}
Applying~(\ref{eqn:1plusxC}) to $C_{n}^{(\lambda+1)}(x)$ and $C_{n-1}^{(\lambda+1)}(x)$ gives, upon rearrangement, 
\be
\begin{array}{r c l}
(1+x)\big(C_{n}^{(\lambda+1)}(x)-C_{n-1}^{(\lambda+1)}(x)\big) &\!\!\!\!=\!\!\!\!\!&\\ 
\frac{1}{2}\frac{n+1}{n+\lambda+1}\big(C_{n+1}^{(\lambda+1)}(x)&\!\!\!\!-\!\!\!\!\!&C_{n-1}^{(\lambda+1)}(x)\big) + 
\frac{1}{2}\frac{n+2\lambda}{n+\lambda}\big(C_{n}^{(\lambda+1)}(x)-C_{n-2}^{(\lambda+1)}(x)\big).
\end{array}
\ee
Applying~(\ref{eqn:C_convert}) to each of the bracketed terms on the right-hand side and cancelling common terms gives the required result.
\end{proof}

\begin{corollary}
The Legendre polynomials, $P_n(x)$, and the Chebyshev polynomials, $U_n(x)$, satisfy
\be 
    n\big(P_n(x)+P_{n-1}(x)\big) = (1+x)\big(C_{n-1}^{(3/2)}(x)-C_{n-2}^{(3/2)}(x)\big)
\ee 
and
\be
    U_n(x)+(n+1)U_{n-1}(x) = 2(1+x)\big(C_{n-1}^{(2)}(x)-C_{n-2}^{(2)}(x)\big).
\ee
\end{corollary}%
\begin{proof}
 Take $n\mapsto n-1$ with $\lambda = \frac{1}{2}$ and $\lambda = 1$ in~(\ref{eqn:appendix}), respectively.
\end{proof}

%% file: half_analysis.tex
\subsection{Convergence}\label{subsec:convergence}
Note that the decompositions of the right-hand side and solution of~(\ref{eqn:halfinteq}) in the forms~(\ref{eqn:ansatz}) and~(\ref{eqn:eandf}) are not unique, 
so the well-posedness of \eqref{eqn:system_int} is not 
immediate. However, the Schur complement of the $(1,1)$ block of~(\ref{eqn:system_int}) yields
\begin{eqnarray}
 \big(Q_\P - \sigma^2 I\big)\underline{a} &=& Q^{1/2}_\U\underline{f}-\sigma\underline{e}, \label{eqn:schureq}\\
 \sigma \underline{b} &=& \underline{f}-Q_\P^{1/2}\underline{a}, \notag
\end{eqnarray}
where $Q_\P = Q^{1/2}_\U Q^{1/2}_\P$ is the indefinite integral operator acting on the Legendre basis (recall~(\ref{eqn:Q1})).   The fact that $Q_\P$ is banded along with the decaying properties  of its entries leads to a proof of convergence whenever the original equation \eqref{eqn:halfinteq} is solvable in $L^2[-1,1]$.

\begin{definition}
	Define the Banach space $\ell_\lambda^2$ with norm
	\be
	    \|\underline f\|_{\ell_\lambda^2}^2 = \sum_{k=0}^\infty (k+1)^{2\lambda} f_k^2.
	\ee
\end{definition}


\begin{lemma}
Let
	$\Psi := \diag\left(\sqrt{2},\sqrt{2\over 3},\sqrt{2\over 5},\sqrt{2\over 7},\ldots\right).$
	If $\sigma^2$ is an $\ell^2$ eigenvalue of $\tilde Q_{\P} := \Psi Q_{\P}\Psi^{-1}$, then $\sigma$ (as well as $-\sigma$) is an $L^2[-1,1]$ eigenvalue of  $_{-1}\calQ_x^{1/2}$.
\end{lemma}
\begin{proof}
	Note that $\|P_n\| = \sqrt{2 \over 2n+1}$, hence conjugating by $\Psi$ recasts the operator to acting on expansions in the orthonormalized Legendre polynomials $\tilde P_n(x) := P_n(x) \sqrt{2n+1 \over 2}$.  	The assumption on $\sigma^2$ being an $\ell^2$ eigenvalue enforces that any eigenvector $\underline a$ of $\tilde Q_{\P}$ corresponds to the normalized Legendre coefficients of a function $a(x)$ in $L^2[-1,1]$, with norm $\|\underline a\|_{\ell^2}$. 

	The entries of $\tilde Q_{\P}$ decay like $1/k$, see \eqref{eqn:Q1}, which implies that $\tilde Q_{\P} : \ell_\lambda^2 \rightarrow \ell_{\lambda+1}^2$.  It follows immediately that $\underline a \in \ell_{\lambda}^2$ for all $\lambda$:   $\underline a \in \ell_{\lambda}^2$ implies that $\underline a = \sigma^{-2} \tilde Q_{\P} \underline a \in \ell_{\lambda+1}^2$.  In particular, $\underline a \in \ell^1$.    
We can bound
\be	
	\|\sqrt{1+x} U_k\|^2 = \int_{-1}^1 (1+x) {\sin^2 (k+1)\cos^{-1} x \over \sin^2 x} dx = 
				\int_0^\pi (1+\cos \theta) {\sin^2 (k+1) \theta \over \sin \theta} d\theta \leq
				2 \pi (k+1)
\ee
since Lagrange's trigonometric identities ensure that
\be
	\left|{\sin(k+1) \theta \over \sin \theta }\right| \leq k + 1.
\ee
 Thus the $O(1/\sqrt{k})$ decay in $Q_{\P}^{1/2} \Psi^{-1}$ cancels the $O(\sqrt{k})$ growth from $\|\sqrt{1+x} U_k\|$, and we have 
	\be
	\|\left(\sqrt{1+x} U_0(x),\sqrt{1+x} U_1(x),\ldots\right) Q_{\P}^{1/2} \Psi^{-1} \underline a\| \leq C \|\underline a\|_{\ell^1} < \infty.
	\ee
That is, the entries of $\underline b = \sigma^{-1}Q_{\P}^{1/2} \Psi^{-1} \underline a$  correspond to the second-kind Chebyshev coefficients of a function $b(x)$ such that $\sqrt{1+x} b(x)$ is in  $L^2[-1,1]$.  We therefore have an $L^2[-1,1]$ eigenvector $a(x) + \sqrt{1+x} b(x)$, satisfying:
\begin{align*}
{\cal Q}_x^{1/2} (a(x) + \sqrt{1+x} b(x)) &=  {\cal Q}_x^{1/2} ({\mathbf{P}}(x)  \Psi^{-1}\underline a + {\mathbf{U}_{1/2}}(x)\underline b ) \\
							&=   {\mathbf{U}_{1/2}}(x) Q_{\P}^{1/2}  \Psi^{-1} \underline a + {\mathbf{P}}(x) Q_{\U}^{1/2} \underline b \\		
							&=     \sigma {\mathbf{U}_{1/2}}(x) \underline b + \sigma {\mathbf{P}}(x)   \Psi^{-1} \underline a \\
							&= \sigma (a(x) + \sqrt{1+x} b(x)) 
\end{align*}
	\end{proof}

\begin{lemma}
	If $\sigma I + _{-1}\calQ_x^{1/2}$ is invertible in $L^2[-1,1]$ then $\sigma^2 I + \tilde Q_{\P}$ is invertible in $\ell_\lambda^2$ for all $\lambda$ and in $\ell^1$.  If $\underline{e},\underline{f} \in \ell^1$, and $\underline a = (\sigma^2 I + \tilde Q_{\P})^{-1}(Q_\U^{1/2} \underline f - \sigma \underline e)$, then $u(x) = a(x) + \sqrt{1+x} b(x)$ satisfies
	$$(\sigma I + _{-1}\calQ_x^{1/2})u(x) = e(x) + \sqrt{1+x} f(x)$$
for $e(x) = \P(x) \underline e$, $f(x) = \U_{1/2}(x) \underline f$, $a(x) = \P(x) \underline a$ and $b(x) = \sigma^{-1}(f(x) - \U_{1/2}(x) Q_\P^{1/2} \underline a)$.
\end{lemma}	
\begin{proof}
	The decay in the entries of $\tilde Q_\P$ and bandedness imply that $ \| P_N \tilde Q_\P - \tilde Q_\P\|_{\ell_\lambda^2} \rightarrow 0$: $\tilde Q_\P$ is compact in $\ell_\lambda^2$ (and by a similar argument, in $\ell^1$).  Compactness guarantees that the operator only has discrete eigenvalues.  However, the previous lemma ensures that if $\sigma I + _{-1}\calQ_x^{1/2}$ is invertible in $L^2[-1,1]$, then $\sigma^2$ is not an $\ell^2$ eigenvalue of $\tilde Q_\P$, and hence $\sigma^2 I + \tilde Q_{\P}$ is invertible in $\ell^2$.  But any $\ell^2$ eigenvector is an eigenvector in $\ell_\lambda^2$ for all $\lambda \geq 0$ (and in $\ell^1$) as $\tilde Q_\P$ induces additional decay, and trivially, any $\ell_\lambda^2$ eigenvector is automatically an $\ell^2$ eigenvector.  Thus we know that $\sigma^2$ is also not an $\ell_\lambda^2$ (or $\ell^1$) eigenvalue, and the operator is invertible.
	
	Therefore, if $\underline{e},\underline{f} \in \ell^1$ then $\underline a \in \ell^1$, hence (by the logic of the previous lemma) $a(x) + \sqrt{1+x} b(x) \in L^2[-1,1]$.  We have thus constructed the unique $L^2[-1,1]$ solution of $(\sigma I + _{-1}\calQ_x^{1/2}) u(x) = e(x) + \sqrt{1+x} f(x)$
	\end{proof}
	
We now consider the finite section approximation of $\eqref{eqn:system_int}$, i.e., we define the projection operator 
$P_N : \ell^2 \rightarrow \mathbb R^N$ and consider the $2N \times 2N$ finite section approximation
\be\label{eqn:system_int_fs}
\left(\begin{array}{c c}
	    \sigma I_N & P_N Q_\U^{1/2} P_N^\top \\
	    P_N Q_\P^{1/2} P_N^\top & \sigma I_N
     \end{array}\right)
     \left(\begin{array}{c}
	    \underline{a}_N\vphantom{Q_\P^{1/2}}\\\underline{b}_N\vphantom{Q_\P^{1/2}}
     \end{array}\right) = 
     \left(\begin{array}{c}
	    P_N \underline{e}\vphantom{Q_\P^{1/2}}\\P_N \underline{f}\vphantom{Q_\P^{1/2}}
     \end{array}\right).
\ee
This leads to an approximation
\be
\begin{array}{r c l c l}
 	a(x) & \approx &a_N(x) & = & {\mathbf P}(x)\underline a_N \\
	b(x) & \approx &b_N(x) & = & \sigma^{-1} {\mathbf U}(x) (P_N \underline f - Q_\P^{1/2} \underline a_N)\\
	u(x) & \approx &u_N(x) & = & a_N(x) + \sqrt{1+x}b_N(x).
\end{array}	
\ee

\begin{theorem}\label{th:conv}
	If $\sigma I + _{-1}\calQ_x^{1/2}$ is invertible in $L^2[-1,1]$ and $\underline e$, $\underline f$ are in $\ell^1$, then the finite section approximation to \eqref{eqn:schureq} $u_N$ converges to the true solution of \eqref{eqn:halfinteq} in $L^2[-1,1]$.
\end{theorem}	
\begin{proof}
Note that, because $Q_\P^{1/2}$ is upper triangular and $Q_\U^{1/2}$ is lower triangular, we have
\be
    P_N Q_\U^{1/2} P_N^\top P_N Q_\P^{1/2}   P_N^\top = P_N Q_\U^{1/2} Q_\P^{1/2} P_N^\top = P_N Q_\P P_N^\top.
\ee
  It follows that $\underline{a}_N$ is also a solution to  the $n \times n$ finite section of \eqref{eqn:schureq}:
\begin{equation}\label{eqn:schur-finitesection}
 P_N (Q_\P  -    \sigma^2 I) P_N^\top \underline a_N = P_N ( Q_\U^{1/2} \underline f - \sigma \underline e).
\end{equation}
	If the condition of this theorem holds, then by the previous lemma, $\sigma^2$ is not an eigenvalue of $\tilde Q_{\P}$ .    $\tilde Q_{\P}$ is a compact operator on $\ell^1$, therefore the finite-section approximation $\underline a_N$ converges to $\underline a$ in an $\ell^1$ sense (this follows from a Neumann series argument, see e.g., \cite[Theorem 4.5]{olver2013}). This implies convergence of $a_N(x)$ to $a(x)$ in $L^2[-1,1]$ and convergence of $b_N(x)$ to $b(x)$ in $L^2[-1,1]$, thence $u_N(x)$ converges to $u(x)$ in $L^2[-1,1]$.  
	
\end{proof}
%

\begin{corollary}\label{cor:convrate}
	If $\underline e,\underline f \in \ell_\lambda^2$  then the finite section approximation converges in $\ell_\lambda^2$.  If this condition holds for all $\lambda$, then $u_N$ converges in $L^2[-1,1]$ at a spectrally fast rate.  Similarly, if $\underline e, \underline f$ decay exponentially, then $u_N$ converges in $L^2[-1,1]$ exponentially fast.  
\end{corollary}
\begin{proof}
	The first statement follows from the operator being a compact perturbation of the identity in all  $\ell_\lambda^2$ spaces, hence the same argument as Theorem~\ref{th:conv} applies.  The second statement follows from relating convergence in $\ell_\lambda^2$ to fast convergence in $\ell^1$.     The exponentially fast convergence follows similarly by adapting the results to the exponentially weighted norm $\sqrt{\sum_{k=0}^\infty |R^k  f_k|^2}$.  
\end{proof}


\subsection{Stability}
Unfortunately, solvability of the resulting equation is not the only issue: we must also consider conditioning.   Now, $v(x) = e^{x/\sigma^2}$ is the solution to $\calQ u(x) - \sigma^2 u(x)= e^{-1/\sigma^2}$, hence, for $\sigma\ll1$,
$v(x)$ is approximately in the kernel of $\calQ u(x) - \sigma^2 I$. Therefore, we should expect the solution of the above system, 
and hence the system~(\ref{eqn:system_int}) to be ill-conditioned when $\sigma\ll1$. Indeed, the pseudo-spectral plot of 
the two (truncated) linear systems in Figure~\ref{fig:pseudo} confirms this.
\begin{figure}[H]\label{fig:pseudo}
\center
\includegraphics[height=100pt]{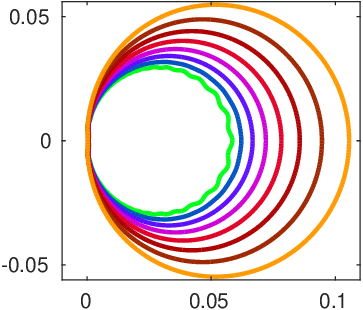}\hspace*{30pt}
\includegraphics[trim={0 .05pt 0 3pt},clip,height=100pt]{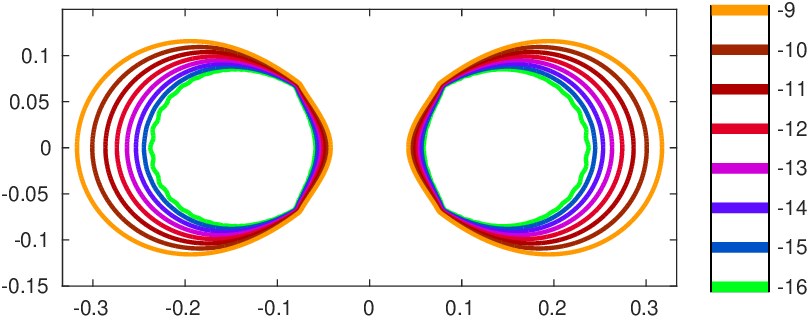}
\caption{Left: Pseudospectra (computing using EigTool~\cite{eigtool}) of the operator $Q_{1/2}$ truncated to a $200\times200$ matrix. Right: The same for~(\ref{eqn:system_int}) with $\sigma = 0$.}
\end{figure}
Decreasing $\sigma$ in this way is equivalent to a change of variables from $[-1,1]$ to 
a longer `time' domain (if we consider the independent variable as time). In particular, 
let $y = \frac{1}{\sigma^2}x + c$ and $u(x) = v(y)$, then substituting to~(\ref{eqn:halfint_def}) we find
\be
    \ _{-1}\calQ^{1/2}_xu(x) = \sigma\ _{-\alpha}\calQ^{1/2}_yv(y).
\ee
Investigating the singular values of the operator suggests that as $\sigma\rightarrow0$ it
is only a single singular value that decays to zero and that it might be possible to regularise
the problem. However, this is beyond the scope of the current paper and we avoid this limiting case for now.

%% file: HaleOlver2015.bbl
\begin{thebibliography}{10}

\bibitem{alpert1991}
{\sc B.~K. Alpert and V.~Rokhlin}, {\em A fast algorithm for the evaluation of
  {L}egendre expansions}, SIAM Journal on Scientific and Statistical Computing,
  12 (1991), pp.~158--179.

\bibitem{andrews1999}
{\sc G.~E. Andrews, R.~Askey, and R.~Roy}, {\em Special functions}, vol.~71,
  Cambridge University Press, 1999.

\bibitem{aurentz2015}
{\sc J.~L. Aurentz and L.~N. Trefethen}, {\em Chopping a chebyshev series}, ACM
  Transactions on Mathematical Software, 43 (2017), p.~33.

\bibitem{bagley1983}
{\sc R.~L. Bagley and P.~J. Torvik}, {\em Fractional calculus -- a different
  approach to the analysis of viscoelastically damped structures}, 1983.

\bibitem{bagley1984}
\leavevmode\vrule height 2pt depth -1.6pt width 23pt, {\em On the appearance of
  the fractional derivative in the behavior of real materials}, 1984.

\bibitem{bajlekova2001}
{\sc E.~G. Bajlekova}, {\em Fractional evolution equations in {B}anach spaces},
  ProQuest LLC, Ann Arbor, MI, 2001.
\newblock Thesis (Dr.)--Technische Universiteit Eindhoven (The Netherlands).

\bibitem{chen2016}
{\sc S.~Chen, J.~Shen, and L.-L. Wang}, {\em Generalized {J}acobi functions and
  their applications to fractional differential equations}, Math. Comp., 85
  (2016), pp.~1603--1638.

\bibitem{christensen2003}
{\sc O.~Christensen}, {\em An Introduction to Frames and Riesz Bases},
  Birkhauser, 2003.

\bibitem{cui2009}
{\sc M.~Cui}, {\em Compact finite difference method for the fractional
  diffusion equation}, Journal of Computational Physics, 228 (2009),
  pp.~7792--7804.

\bibitem{dalir2010}
{\sc M.~Dalir and M.~Bashour}, {\em Applications of fractional calculus},
  Applied Mathematical Sciences, 4 (2010), pp.~1021--1032.

\bibitem{deng2008}
{\sc W.~Deng}, {\em Finite element method for the space and time fractional
  {F}okker--{P}lanck equation}, SIAM Journal on Numerical Analysis, 47 (2008),
  pp.~204--226.

\bibitem{DLMF}
{\em {NIST Digital Library of Mathematical Functions}}.
\newblock http://dlmf.nist.gov/, Release 1.0.10 of 2015-08-07.
\newblock Online companion to \cite{DLMFprint}.

\bibitem{Chebfun}
{\sc T.~A. Driscoll, N.~Hale, and L.~N. Trefethen}, {\em Chebfun Guide},
  Pafnuty Publications, 2014.

\bibitem{ford2011}
{\sc N.~Ford, J.~Xiao, and Y.~Yan}, {\em A finite element method for time
  fractional partial differential equations}, Fractional Calculus and Applied
  Analysis, 14 (2011), pp.~454--474.

\bibitem{fracspeccode}
{\sc N.~Hale}, {\em Companion code to this paper}.
\newblock https://github.com/nickhale/fracspec\_code.
\newblock Last accessed 17 Oct 2016.

\bibitem{hale2014}
{\sc N.~Hale and A.~Townsend}, {\em A fast, simple, and stable
  {C}hebyshev--{L}egendre transform using an asymptotic formula}, SIAM J. Sci.
  Comput., 36 (2014), pp.~A148--A167.

\bibitem{hilfer2000}
{\sc R.~Hilfer}, {\em Applications of Fractional Calculus in Physics}, World
  Scientific, 2000.

\bibitem{LiXuSobolevBall}
{\sc H.~Li and Y.~Xu}, {\em Spectral approximation on the unit ball}, SIAM J.
  Numer. Anal., 52 (2014), pp.~2647--2675.

\bibitem{li2009}
{\sc X.~Li and C.~Xu}, {\em A space-time spectral method for the time
  fractional diffusion equation}, SIAM J. Numer. Anal., 47(3) (2009),
  p.~2108–2131.

\bibitem{liu2004}
{\sc F.~Liu, V.~Anh, and I.~Turner}, {\em Numerical solution of the space
  fractional {F}okker--{P}lanck equation}, Journal of Computational and Applied
  Mathematics, 166 (2004), pp.~209--219.

\bibitem{magin2006}
{\sc R.~L. Magin}, {\em Fractional Calculus in Bioengineering}, Begell House
  Redding, 2006.

\bibitem{magin2010}
\leavevmode\vrule height 2pt depth -1.6pt width 23pt, {\em Fractional calculus
  models of complex dynamics in biological tissues}, Computers \& Mathematics
  with Applications, 59 (2010), pp.~1586--1593.

\bibitem{meerschaert2006}
{\sc M.~M. Meerschaert and C.~Tadjeran}, {\em Finite difference approximations
  for two-sided space-fractional partial differential equations}, Applied
  numerical mathematics, 56 (2006), pp.~80--90.

\bibitem{oldham2010}
{\sc K.~B. Oldham}, {\em Fractional differential equations in
  electrochemistry}, Advances in Engineering Software, 41 (2010), pp.~9--12.

\bibitem{DLMFprint}
{\sc F.~W.~J. Olver, D.~W. Lozier, R.~F. Boisvert, and C.~W. Clark}, eds., {\em
  {NIST Handbook of Mathematical Functions}}, Cambridge University Press, New
  York, NY, 2010.
\newblock Print companion to \cite{DLMF}.

\bibitem{ApproxFun}
{\sc S.~Olver}, {\em Approx{F}un.jl v0.7, {\tt
  https://github.com/approxfun/approxfun.jl}}.

\bibitem{olver2013}
{\sc S.~Olver and A.~Townsend}, {\em A fast and well-conditioned spectral
  method}, SIAM Rev., 55 (2013), pp.~462--489.

\bibitem{olver2014}
{\sc S.~Olver and A.~Townsend}, {\em A practical framework for
  infinite-dimensional linear algebra}, in Proceedings of the 1st First
  Workshop for High Performance Technical Computing in Dynamic Languages, 2014,
  pp.~57--62.

\bibitem{Ortigueira2006}
{\sc M.~D. Ortigueira and J.~T. Machado}, {\em Fractional calculus applications
  in signals and systems}, Signal Processing, 86 (2006), pp.~2503 -- 2504.
\newblock Special Section: Fractional Calculus Applications in Signals and
  Systems.

\bibitem{Polyanin2008}
{\sc A.~D. Polyanin and A.~V. Manzhirov}, {\em Handbook of integral equations},
  Chapman \& Hall/CRC, Boca Raton, FL, second~ed., 2008.

\bibitem{riesz1949}
{\sc M.~Riesz}, {\em L'int\'egrale de {R}iemann--{L}iouville et le probl\`eme
  de {C}auchy}, Acta Math., 81 (1949), pp.~1--223.

\bibitem{sabatier2007}
{\sc J.~Sabatier, O.~P. Agrawal, and J.~T. Machado}, {\em Advances in
  Fractional Calculus}, vol.~4, Springer, 2007.

\bibitem{scalas2000}
{\sc E.~Scalas, R.~Gorenflo, and F.~Mainardi}, {\em Fractional calculus and
  continuous-time finance}, Physica A: Statistical Mechanics and its
  Applications, 284 (2000), pp.~376--384.

\bibitem{sheng2011}
{\sc H.~Sheng, Y.~Chen, and T.~Qiu}, {\em Fractional Processes and
  Fractional-order Signal Processing: Techniques and Applications}, Springer
  Science \& Business Media, 2011.

\bibitem{slevinsky2016}
{\sc R.~M. Slevinsky and S.~Olver}, {\em A fast and well-conditioned spectral
  method for singular integral equations}, J. Comput. Phys., 332 (2017),
  pp.~290--315.

\bibitem{stewart1998}
{\sc G.~W. Stewart}, {\em Afternotes goes to graduate school}, Society for
  Industrial and Applied Mathematics (SIAM), Philadelphia, PA, 1998.
\newblock Lectures on advanced numerical analysis.

\bibitem{townsend2015}
{\sc A.~Townsend and S.~Olver}, {\em The automatic solution of partial
  differential equations using a global spectral method}, J. Comput. Phys., 299
  (2015), pp.~106--123.

\bibitem{townsend2016}
{\sc A.~Townsend, M.~Webb, and S.~Olver}, {\em Fast polynomial transforms based
  on {T}oeplitz and {H}ankel matrices}, Maths Comp.,  (2016).
\newblock To appear.

\bibitem{ATAP}
{\sc L.~N. Trefethen}, {\em Approximation theory and approximation practice},
  Society for Industrial and Applied Mathematics (SIAM), Philadelphia, PA,
  2013.

\bibitem{vasil2016}
{\sc G.~M. Vasil, K.~J. Burns, D.~Lecoanet, S.~Olver, B.~P. Brown, and J.~S.
  Oishi}, {\em Tensor calculus in polar coordinates using {J}acobi
  polynomials}, J. Comp. Phys., 325 (2016), pp.~53--73.

\bibitem{eigtool}
{\sc T.~G. Wright.}, {\em Eigtool}, 2002.

\bibitem{yuste2005}
{\sc S.~B. Yuste and L.~Acedo}, {\em An explicit finite difference method and a
  new von {N}eumann-type stability analysis for fractional diffusion
  equations}, SIAM Journal on Numerical Analysis, 42 (2005), pp.~1862--1874.

\bibitem{zayernouri2014}
{\sc M.~Zayernouri and G.~E. Karniadakis}, {\em Fractional spectral collocation
  method}, SIAM J. Sci. Comput., 36 (2014), pp.~A40--A62.

\bibitem{zhao2016}
{\sc L.~{Zhao}, W.~{Deng}, and J.~S. {Hesthaven}}, {\em Spectral methods for
  tempered fractional differential equations}, ArXiv e-print: 1603.06511,
  (2016).

\end{thebibliography}
